 \documentclass[11pt, letterpaper]{amsart}
\usepackage{amsmath,amsthm, amssymb, amsfonts, amscd}
\usepackage{verbatim}

\usepackage[all]{xy}
\usepackage{amsxtra}
\usepackage{url}
\usepackage{enumerate}
\DeclareFontEncoding{OT2}{}{} 
\DeclareTextFontCommand{\textcyr}{\fontencoding{OT2}
     \fontfamily{wncyr}\fontseries{m}\fontshape{n}\selectfont}

\theoremstyle{plain}

\newtheorem{lemma}{Lemma}[section]
\newtheorem{proposition}[lemma]{Proposition}
\newtheorem{theorem}[lemma]{Theorem}
\newtheorem{corollary}[lemma]{Corollary}

\theoremstyle{definition}

\newtheorem{example}[lemma]{Example}
\newtheorem{definition}[lemma]{Definition}

\newtheorem{notation}[lemma]{Notation}

\theoremstyle{remark}

\newtheorem{remark}[lemma]{Remark}

\numberwithin{equation}{section}

\DeclareFontEncoding{OT2}{}{} 
\DeclareTextFontCommand{\textcyr}{\fontencoding{OT2}
    \fontfamily{wncyr}\fontseries{m}\fontshape{n}\selectfont}

\begin{document}

\newcommand{\m}{^{\times}}

\newcommand{\Gspl}{{G}_{\rm spl}}
\newcommand{\Ginn}{{G}_{\rm inn}}
\newcommand{\Gout}{{G}_{\rm out}}
\newcommand{\Tspl}{{T}_{\rm spl}}
\newcommand{\Bspl}{{B}_{\rm spl}}

\newcommand{\ff}{F^{\times}}
\newcommand{\fs}{F^{\times 2}}
\newcommand{\llg}{\longrightarrow}
\newcommand{\tens}{\otimes}
\newcommand{\inv}{^{-1}}
\newcommand{\Dfn}{\stackrel{\mathrm{def}}{=}}
\newcommand{\iso}{\stackrel{\sim}{\to}}
\newcommand{\leftexp}[2]{{\vphantom{#2}}^{#1}{#2}}
\newcommand{\mult}{\mathrm{mult}}
\newcommand{\sep}{\mathrm{sep}}
\newcommand{\id}{\mathrm{id}}
\newcommand{\diag}{\mathrm{diag}}
\newcommand{\op}{^{\mathrm{op}}}
\newcommand{\ra}{\rightarrow}
\newcommand{\xra}{\xrightarrow}
\newcommand{\gen}{\mathrm{gen}}
\newcommand{\CH}{\operatorname{CH}}
\newcommand{\Span}{\operatorname{Span}}
\renewcommand{\Im}{\operatorname{Im}}
\newcommand{\Ad}{\operatorname{Ad}}
\newcommand{\NN}{\operatorname{N}}
\newcommand{\Ker}{\operatorname{Ker}}
\newcommand{\Pic}{\operatorname{Pic}}
\newcommand{\Tor}{\operatorname{Tor}}
\newcommand{\Lie}{\operatorname{Lie}}
\newcommand{\ind}{\operatorname{ind}}
\newcommand{\ch}{\operatorname{char}}
\newcommand{\Inv}{\operatorname{Inv}}
\newcommand{\Int}{\operatorname{Int}}
\newcommand{\Inn}{\operatorname{Inn}}
\newcommand{\SInn}{\operatorname{SInn}}
\newcommand{\res}{\operatorname{res}}
\newcommand{\cor}{\operatorname{cor}}
\newcommand{\Br}{\operatorname{Br}}
\newcommand{\Nil}{\operatorname{Nil}}
\newcommand{\Spec}{\operatorname{Spec}}
\newcommand{\Proj}{\operatorname{Proj}}
\newcommand{\SK}{\operatorname{SK}}
\newcommand{\Gal}{\operatorname{Gal}}
\newcommand{\SL}{\operatorname{SL}}
\newcommand{\PGL}{\operatorname{PGL}}
\newcommand{\GL}{\operatorname{GL}}
\newcommand{\gSL}{\operatorname{\mathbf{SL}}}
\newcommand{\gO}{\operatorname{\mathbf{O}}}
\newcommand{\gSO}{\operatorname{\mathbf{SO}}}
\newcommand{\gPSO}{\operatorname{\mathbf{PSO}}}
\newcommand{\gG}{\operatorname{\mathbf{G}}}
\newcommand{\gSp}{\operatorname{\mathbf{Sp}}}
\newcommand{\Sympl}{\operatorname{\gSp}}
\newcommand{\gGL}{\operatorname{\mathbf{GL}}}
\newcommand{\gPGL}{\operatorname{\mathbf{PGL}}}
\newcommand{\gPGU}{\operatorname{\mathbf{PGU}}}
\newcommand{\gSpin}{\operatorname{\mathbf{Spin}}}
\newcommand{\gSU}{\operatorname{\mathbf{SU}}}
\newcommand{\gPSU}{\operatorname{\mathbf{PSU}}}
\newcommand{\gU}{\operatorname{\mathbf{U}}}
\newcommand{\End}{\operatorname{End}}
\newcommand{\Hom}{\operatorname{Hom}}
\newcommand{\Mor}{\operatorname{Mor}}
\newcommand{\Map}{\operatorname{Map}}
\newcommand{\Aut}{\operatorname{Aut}}
\newcommand{\Coker}{\operatorname{Coker}}
\newcommand{\Ext}{\operatorname{Ext}}
\newcommand{\Nrd}{\operatorname{Nrd}}
\newcommand{\Norm}{\operatorname{Norm}}
\newcommand{\spann}{\operatorname{span}}
\newcommand{\Symd}{\operatorname{Symd}}
\newcommand{\Sym}{\operatorname{S}}
\newcommand{\red}{\operatorname{red}}
\newcommand{\Prp}{\operatorname{Prp}}
\newcommand{\Prd}{\operatorname{Prd}}
\newcommand{\tors}{\operatorname{tors}}
\newcommand{\Tr}{\operatorname{Tr}}
\newcommand{\Trd}{\operatorname{Trd}}
\newcommand{\disc}{\operatorname{disc}}
\newcommand{\divi}{\operatorname{div}}
\newcommand{\GCD}{\operatorname{g.c.d.}}
\newcommand{\rank}{\operatorname{rank}}
\newcommand{\A}{\mathbb{A}}
\renewcommand{\P}{\mathbb{P}}
\newcommand{\Z}{\mathbb{Z}}
\newcommand{\bbZ}{\mathbb{Z}}
\newcommand{\bbG}{\mathbb{G}}
\newcommand{\bbQ}{\mathbb{Q}}
\newcommand{\N}{\mathbb{N}}
\newcommand{\F}{\mathbb{F}}
\newcommand{\Q}{\mathbb{Q}}
\newcommand{\R}{\mathbb{R}}
\newcommand{\C}{\mathbb{C}}
\newcommand{\QZ}{\mathop{\mathbb{Q}/\mathbb{Z}}}
\newcommand{\gm}{\mathbb{G}_m}
\newcommand{\hh}{\mathbb{H}}

\newcommand{\cA}{\mathcal A}
\newcommand{\cB}{\mathcal B}
\newcommand{\cC}{\mathcal C}
\newcommand{\cU}{\mathcal U}
\newcommand{\cI}{\mathcal I}
\newcommand{\cJ}{\mathcal J}
\newcommand{\cO}{\mathcal O}
\newcommand{\cF}{\mathcal F}
\newcommand{\cG}{\mathcal G}
\newcommand{\cL}{\mathcal L}
\newcommand{\cP}{\mathcal P}

\newcommand{\falg}{F\mbox{-}\mathfrak{alg}}
\newcommand{\fgroups}{F\mbox{-}\mathfrak{groups}}
\newcommand{\fields}{F\mbox{-}\mathfrak{fields}}
\newcommand{\groups}{\mathfrak{Groups}}
\newcommand{\abelian}{\mathfrak{Ab}}
\newcommand{\p}{\mathfrak{p}}

\newcommand{\Kbar}{\overline{K}}
\newcommand{\Kgenbar}{\overline{K_{\gen}}}

\def\kbar{{\bar{k}}}
\def\Out{{\mathrm{Out}}}
\def\ad{{\mathrm{ad}}}
\def\ttt{{\mathfrak{t}}}
\def\ggg{{\mathfrak{g}}}
\def\G{{\mathbf{G}}}
\def\inn{{\mathrm{inn}}}

\def\AA{{\mathbf{A}}}
\def\BB{{\mathbf{B}}}
\def\CC{{\mathbf{C}}}
\def\DD{{\mathbf{D}}}
\def\GG{{\mathbf{G}}}
\def\ve{{\varepsilon}}
\def\vk{{\varkappa}}
\def\half{{\frac{1}{2}}}
\def\fl{{\mathrm{fl}}}

\def\A{{\mathbb{A}}}

\def\into{\hookrightarrow}

\newcommand{\Sha}{\textcyr{Sh}}

\def\Sh{{\Sha^2}}

\def\gg{{\mathfrak{g}}}

\def\C{{\mathbb{C}}}
\def\Q{{\mathbb{Q}}}
\def\Z{{\mathbb{Z}}}

\def\coker{{\rm coker}}
\def\Hom{{\rm Hom}}

\def\ba{{\mathbf{a}}}

\def\bo{{\mathbf{1}_m}}
\def\go{{\langle\bo\rangle}}

\def\aa{{\mathbf{a}}}
\def\ga{{\langle\aa\rangle}}

\def\bb{{\boldsymbol{\beta}}}
\def\bgen{{\langle\bb\rangle}}

\def\ov{\overline}

\def\Gbar{{\overline{G}}}
\def\Bbar{{\overline{B}}}
\def\Tbar{{\overline{T}}}
\def\im{{\rm im}}

\newcommand{\isoto}{\overset{\sim}{\to}}
\def\SAut{{\mathrm{SAut}}}
\def\lsig{{{}^\sigma}}
\def\sP{{\mathcal{P}}}
\newcommand{\labelto}[1]{\xrightarrow{\makebox[1.5em]{\scriptsize ${#1}$}}}

\newcommand{\X}{{\mathcal{X}}}
\newcommand{\XX}{{\textsf{X}}}
\newcommand{\rPsi}{{\operatorname{RD}}}
\newcommand{\rXi}{{\operatorname{BRD}}}
\newcommand{\V}{{\operatorname{W^{\rm ext}}}}

\def\Gammac{{\Gamma_{\text{\rm{c}}}}}
\def\Stab{{\mathrm{Stab}}}
\def\spl{{\rm spl}}
\def\SAut{{\rm SAut}}

\def\twisted{{\rm twisted}}
\def\sign{{\rm sign\,}}
\def\sS{{{\rm SU}_3}}

\def\SU{{\bf SU}}
\def\SL{{\bf SL}}
\def\GL{{\bf GL}}
\newcommand{\lt}{\mathfrak{t}}
\newcommand{\Lt}{\mathfrak{t}_L}
\def\bbP{{\mathbb{P}}}
\newcommand{\birat}{\overset{\simeq}{\dashrightarrow}}

\def\Gm{\bbG_{\rm m}}
\def\AAA{{\mathfrak{A}}}
\def\gen{{\rm gen}}
\def\der{{\rm der}}
\def\Kbar{{\overline{K}}}
\def\Stab{{\rm Stab}}
\def\ov{\overline}
\def\G{{\bbG}}
\def\tV{{^\theta\V}}
\def\ombar{{\overline{\omega}}}
\def\M{{\mathfrak{M}}}
\def\tA{{{}^\theta\!A}}
\newcommand{\st}{{\mathrm{st}}}




\title[Stably Cayley groups]{Stably Cayley  groups over
fields of characteristic zero}

\author{M.~Borovoi}
\address{Borovoi: Raymond and Beverly Sackler School of Mathematical Sciences,
Tel Aviv University, 69978 Tel Aviv, Israel}
\email{borovoi@post.tau.ac.il}
\thanks{Borovoi was supported in part
by the Hermann Minkowski Center for Geometry}

\author{B.\`E~Kunyavski\u\i}
\address{Kunyavski\u\i: Department of Mathematics, Bar-Ilan University, 52900
Ramat Gan, Israel}
\email{kunyav@macs.biu.ac.il}
\thanks{Kunyavski\u\i \  was supported in part
by the Minerva Foundation through the Emmy Noether Institute for
Mathematics}

\author{N.~Lemire}
\address{Lemire: Department of Mathematics, University of Western Ontario, London,
ON N6A 5B7, Canada}
\email{nlemire@uwo.ca}
\thanks{Lemire was supported in part by an NSERC Discovery Grant}

\author{Z.~Reichstein}
\address{Reichstein: Department of Mathematics, University of British Columbia,
  Vancouver, BC V6T 1Z2, Canada}
\email{reichst@math.ubc.ca}
\thanks{Reichstein was supported in part by an NSERC Discovery Grant}


\begin{abstract}
A linear algebraic group $G$  over a field $k$
  is called a Cayley group if it admits
a Cayley map, i.e., a $G$-equivariant birational isomorphism over $k$
between the group variety $G$ and its Lie algebra.
A Cayley map can be thought of as a partial algebraic analogue
of the exponential map.  A prototypical example is the classical
``Cayley transform" for the special orthogonal group $\gSO_n$ defined
by Arthur Cayley in 1846. A linear algebraic group
$G$ is called {\em stably Cayley}
if $G \times \bbG_m^r$ is Cayley for some $r \ge 0$.
Here $\bbG_m^r$ denotes the split $r$-dimensional $k$-torus.
These notions were introduced in 2006 by Lemire, Popov and Reichstein,
who  classified Cayley and stably Cayley simple groups
over an algebraically closed field of characteristic zero.\\

In this paper we study reductive Cayley groups over
an arbitrary field $k$ of characteristic zero.
Our main results are a criterion for a reductive group $G$
to be stably Cayley, formulated in terms of its character
lattice, and a classification of stably Cayley simple groups.
\end{abstract}


\maketitle

\section{Introduction}

Let $k$ be a field of characteristic 0 and
$\kbar$  a fixed algebraic closure of $k$.
Let  $G$ be a connected linear algebraic $k$-group.
A birational isomorphism $\phi \colon \Lie(G)
\stackrel{\simeq}{\dashrightarrow} G$ is called a {\em Cayley map}
if it is equivariant with respect to the conjugation action of $G$
on itself and the adjoint action of $G$ on its Lie algebra
$\Lie(G)$, respectively. A Cayley map can be thought of as a
partial algebraic analogue of the exponential map.  A prototypical
example is the classical ``Cayley transform'' for the special
orthogonal group $\gSO_n$ defined by Arthur Cayley \cite{cayley} in
1846. A linear algebraic $k$-group $G$ is called {\em
Cayley} if it admits a Cayley map and {\em stably Cayley} if $G
\times_k \bbG_m^r$ is Cayley for some $r \ge 0$. Here $\bbG_m$
denotes the split one-dimensional $k$-torus. These notions were introduced
by Lemire, Popov and Reichstein \cite{LPR06}; for a more detailed
discussion and numerous classical examples, we refer the reader
to~\cite[Introduction]{LPR06}.  The main results of~\cite{LPR06} are
the classifications of  Cayley and stably Cayley simple groups in
the case where the base field $k$ is algebraically closed and of
characteristic $0$. The goal of this paper is to extend some of
these results to the case where $k$ is an arbitrary field of
characteristic $0$.
By a reductive $k$-group we always mean a {\em connected} reductive $k$-group.

\begin{example} \label{ex.alg-closed} If $k$ is algebraically closed
and $G$ is a reductive $k$-group, then by~\cite[Theorem 1.27]{LPR06}
$G$ is stably Cayley if and only if its character lattice is
quasi-permutation; see Definition~\ref{def.qp}.
\end{example}

\begin{example} \label{ex.torus} Let $T$ be a $k$-torus of dimension $d$.
By definition, $T$ is Cayley (respectively, stably Cayley) over $k$
if and only if $T$ is $k$-rational (respectively, stably $k$-rational).
If $k$ is algebraically closed, then $T \simeq \bbG_m^d$, hence $T$ is always
rational, and thus always Cayley. More generally,
Voskresenski\u\i's criterion for stable
rationality~\cite[Theorem 4.7.2]{Voskresenskii-book}
asserts that $T$ is stably rational if and only
if the character lattice $\XX(T)$ is quasi-permutation
(see Definition~\ref{def.qp}).

It has been conjectured that every stably rational torus is rational.
To the best of our knowledge, this conjecture is still open. Moreover,
we are not aware of any simple lattice-theoretic criterion for
the rationality of $T$.
\end{example}

Note that the term ``character lattice" is used in different ways in
Examples~\ref{ex.alg-closed} and~\ref{ex.torus}.  In both cases the
underlying $\mathbb{Z}$-module is $\XX(\Tbar)$ (where
$\Tbar=T\times_k \kbar$, $\kbar$ is an algebraic closure of $k$, and
$T$ is a maximal torus of $G$ in Example~\ref{ex.alg-closed}) but
the group acting on $\XX(\Tbar)$ is the Weyl group $W = W(G,T)$ in
Example~\ref{ex.alg-closed} and the Galois group $\Gal(\kbar/k)$ in
Example~\ref{ex.torus}.  A key role in this paper will be played by
the {\em character lattice} $\X(G)$ of a reductive $k$-group $G$, a
notion that bridges the special cases considered in these two
examples.
The underlying $\mathbb Z$-module in this general setting
is still $\XX(\Tbar)$, but the group acting on it
is the {\em extended Weyl group} $\V=W \rtimes A$, where $W$ is the usual Weyl group of $\Gbar$ and
$A$ is the image of $\Gal(\kbar/k)$ under the so-called ``$\ast$-action"
(see Tits \cite[\S\,2.3]{Tits} for a construction of the $*$-action).
For the definition of $\V$, see Section~\ref{sect.character-lattice}.
Equivalently, $\X(G)$ is the character lattice $\XX(T_\gen)$ of the generic torus
$T_{\text{\rm{gen}}}$ of $G$. This torus is defined over a certain
transcendental field extension $K_{\gen}$ of $k$;
see~\cite[\S4.2]{Voskresenskii-book}.
Informally speaking, we think of the Weyl group $W$ as
``the geometric part"  of $\V$,
and of the image $A$ of the $\ast$-action as
``the arithmetic part". Examples~\ref{ex.alg-closed}
and~\ref{ex.torus} represent two opposite extremes, where the group $\V$
is  ``purely geometric" and ``purely arithmetic", respectively.
As we pass from a reductive group $G$ to its generic torus $T_{\rm gen}$,
the geometric part migrates to the arithmetic part, while
the overall group $\V$ remains the same.
\medskip

We are now ready to state our first main theorem.

\begin{theorem} \label{thm.main1}
Let $G$ be a reductive $k$-group.  The following are equivalent:

\begin{enumerate}[\upshape(a)]
\item[\rm(a)] $G$ is stably Cayley;
\item[\rm(b)] for every field extension $K/k$, every maximal $K$-torus
$T \subset G_K$ is stably rational over $K$;
\item[\rm(c)] the generic $K_{\text{\rm{gen}}}$-torus $T_\gen$  of $G$ is stably rational;
\item[\rm(d)] the character lattice $\X(G)$ of $G$
is quasi-permutation.
\end{enumerate}
\end{theorem}

Next we turn our attention to classifying stably Cayley simple
groups over an arbitrary field $k$ of characteristic zero.
The following results extend~\cite[Theorem 1.28]{LPR06}, where $k$
is assumed to be algebraically closed.

\begin{theorem} \label{cor.main2}
Let $k$ be a field of characteristic $0$ and $G$ an absolutely
simple $k$-group. Then the following conditions are equivalent:

\begin{enumerate}[\upshape(a)]
\item $G$ is stably Cayley over $k$;
\item $G$ is an arbitrary $k$-form of one of the following groups:
\[ \text{$\gSL_3$, $\gPGL_n$ ($n=2$ or $n \ge 3 {\textit{ odd}})$, $\gSO_n$ $(n \ge 5)$,
$\Sympl_{2n}$ $(n \ge 1)$, $\gG_2$}, \] or an {\em inner} $k$-form of
$\gPGL_n$ $(n \ge 4 {\textit{ even}})$.
\end{enumerate}
\end{theorem}

Using Theorem \ref{cor.main2} we can give a complete classification of
stably Cayley simply connected and adjoint semisimple groups
over an arbitrary field $k$ of characteristic zero; see
Section~\ref{sect.sc-ad}.
A complete description of all stably Cayley semisimple $k$-groups
is out of our reach at the moment, even if $k$ is algebraically closed.
However, we will prove the following classification of
stably Cayley {\em simple} $k$-groups
(not necessarily absolutely simple).

\begin{theorem} \label{thm.main3}
Let $G$ be a simple (but not necessarily absolutely simple)
$k$-group over a field $k$ of characteristic 0.
Then the following conditions are equivalent:
\begin{enumerate}[\upshape(a)]
\item $G$ is stably Cayley over $k$;
\item $G$ is isomorphic to $R_{l/k}(G_1)$,
where $l/k$ is a finite field extension and $G_1$ is either a
stably Cayley absolutely simple group over $l$ $($i.e., one of the
groups listed in Theorem~\ref{cor.main2}(b)) or
an outer $l$-form of $\gSO_4$.
\end{enumerate}
Here $R_{l/k}$ denotes the Weil functor of restriction of scalars.
\end{theorem}

The rest of this paper is structured as follows.
Sections~\ref{sect.lattices}--\ref{sect.(G,S)}
are devoted to preliminary material on quasi-permutation lattices,
automorphisms and semi-automorphisms
of algebraic groups over non-algebraically closed fields, and
$(G, S)$-fibrations.  While some of this material is known,
we have not been able to find references, where the definitions and
results we need are proved in full generality. We have thus
opted for a largely self-contained exposition.

Theorem~\ref{thm.main1} is proved in
Section~\ref{sect.proof-of-thm.main1}.  Theorem~\ref{cor.main2} is an easy
consequence of Theorem~\ref{thm.main1} and previously known results
on character lattices of absolutely simple groups from  \cite{LPR06}
and Cortella and Kunyavski\u\i's paper \cite{CK}; the details
of this argument are presented in Section~\ref{sect.proof-of-cor.main2}.
The proof of Theorem~\ref{thm.main3} relies on new results of
character lattices and thus requires considerably more work.
After passing to an algebraic closure $\kbar$ of $k$, we are faced with
the problem of classifying semisimple stably Cayley groups of the form
$G=H^m/C$, where $H$ is a simply connected simple group over $\kbar$
and $C\subset H^m$ is a central subgroup. Our classification theorem
for such groups is stated in Section~\ref{sect7}; see
Theorem~\ref{thm:product-closed-field}. In Section~\ref{sect8}
we present criteria
for a lattice not to be quasi-permutation (or even quasi-invertible).
In Section~\ref{sect.sc-ad} we classify stably Cayley simply
connected and adjoint groups.
The proof of Theorem~\ref{thm:product-closed-field}, based on
case-by-case analysis, occupies Section~\ref{sect.family}--\ref{sect.sl3}.
In Section~\ref{sect.proof-of-thm.main3} we deduce Theorem~\ref{thm.main3}
  from Theorem~\ref{thm:product-closed-field} by passing
back from $\kbar$ to $k$.

\begin{remark} \label{rem.char}
The assumption that $\operatorname{char}(k) = 0$ is used primarily in
Section~\ref{sect.(G,S)}, which, in turn, relies on~\cite{CTKPR}.
It seems plausible that our main results should remain true in arbitrary
characteristic, but we have not checked this.
\end{remark}

\begin{remark} \label{rem.Cayley}
A key consequence
of Theorem~\ref{thm.main1} is that, for a reductive $k$-group $G$,
being stably Cayley is a property of its character lattice.
If ``stably Cayley" is replaced by ``Cayley", this is no longer the case,
even for absolutely simple groups.
Indeed, the  groups  $\gSU_3$ and
split $\gG_2$,  defined over the field $\mathbb R$ of real numbers,
have isomorphic character lattices; both are stably Cayley.  By a theorem of
Iskovskikh~\cite{Isk}, $\gG_2$ is not Cayley over $\mathbb R$
(not even over $\mathbb C$); cf.~\cite[Proposition 9.10]{LPR06}.
On the other hand,  $\gSU_3$ is Cayley, see
Borovoi--Dolgachev \cite[Theorem 1.2]{BD}.

For reasons illustrated by the above example, the problem of classifying
simple Cayley groups, in a manner analogous
to Theorems~\ref{cor.main2} and~\ref{thm.main3},
appears to be out of reach at the moment.  In particular,
we do not know which outer forms of $\gPGL_n$ (if any)
are Cayley, for any odd integer $n \ge 5$.
\end{remark}

\begin{remark} \label{rem.outer}
Suppose $\Gspl$ is a split reductive group over $k$,
$\Ginn$ is an inner form of $\Gspl$ over $k$,
and $\Gbar := \Gspl \times_k \bar{k} \simeq \Ginn \times_k \bar{k}$.
As a consequence of Theorem~\ref{thm.main1} we
see that $\Ginn$ is stably Cayley over $k$
if and only if $\Gspl$ is stably Cayley over $k$ if and only if $\Gbar$
is stably Cayley over $\bar{k}$; see Corollary~\ref{cor:inner}.
The reason is that these three groups have the same character lattice.

Stably Cayley groups over
$\bar{k}$ were studied in~\cite{LPR06}. Thus the new and most interesting
phenomena in this paper occur only for outer forms. In particular,
an outer form $\Gout$ of $\Gspl$ may not be stably Cayley over $k$ even if
$\Gspl$ is Cayley; see Theorem~\ref{cor.main2}.
\end{remark}

\begin{remark} \label{rem.cayley-vs-rational} For a reductive group
$G$ the condition of being Cayley is much stronger than the condition of being rational.
For example, for $n \ge 4$ the special linear group $\SL_n$ is rational but is not stably Cayley;
see~\cite[Theorem 1.28]{LPR06}.

Let $G$ be a split reductive $k$-group.
As we pointed out above, if $G$ is stably Cayley over $k$ then any inner form of $G$
over any field extension $K/k$ is also stably Cayley and hence, stably rational over $K$.
We do not know whether or not the converse to the last assertion is true.
That is, suppose for every field extension $K/k$ every inner $K$-form of $G$ is $K$-stably rational.
Can we conclude that $G$ is $k$-stably Cayley?
\end{remark}

\begin{remark} \label{rem.weakly-Cayley}
The choice of $\bbG_m$ in the definition of stably Cayley groups
may seem arbitrary. An alternative definition is as follows.  Let us
say that a linear algebraic $k$-group $G$ is {\em  weakly  Cayley} if
$G\times_k H$ is Cayley for some Cayley $k$-group $H$.
However, an easy modification of the proof of \cite[Lemma 4.7]{LPR06}
shows that $G$ is weakly  Cayley if and only if it is stably Cayley.
\end{remark}

\begin{remark} \label{rem.Cayley-invertible}
An even broader class of groups can be defined as follows.
Let us say that $G$ is {\em Cayley invertible} if $G \times H$
is Cayley for some reductive $k$-group $H$.
Some of the arguments in this paper intended to show that a certain group
$G$ is not stably Cayley, in fact, prove that $G$ is not Cayley invertible;
cf. Proposition~\ref{cor.directproduct}.
Note that a Cayley invertible group need not be stably Cayley. 
See \cite[Theorem 9.1 and Remark 9.3]{CTS2} for an example of a torus 
that is Cayley invertible but not stably Cayley.
\end{remark}

\section{Preliminaries on quasi-permutation lattices}
\label{sect.lattices}

Let $\Gamma$ be a finite group.
By a $\Gamma$-lattice we mean a finitely generated free
abelian group $M$ viewed together with an integral representation
$\Gamma \to \Aut(M)$. We also think of $M$ as a $\Z[\Gamma]$-module;
by a morphism (or exact sequence) of lattices we mean a morphism (or
exact sequence) of $\Z[\Gamma]$-modules.
When we write ``lattice'', rather than ``$\Gamma$-lattice'', we
mean a $\Gamma$-lattice for some finite group $\Gamma$. We say that
a lattice is faithful if the underlying integral representation is faithful.
In those cases where we want to emphasize the dependence on $\Gamma$,
we will sometimes write a lattice as a pair $(\Gamma, M)$.
(The  integral representation $\Gamma \to \Aut(M)$ is assumed
to be clear from the context.) This notation will be particularly
useful when we view $M$ as a $\Gamma_0$-lattice
with respect to  different subgroups $\Gamma_0$ of $\Gamma$.

If $\varphi \colon \Gamma \isoto \Gamma'$ is an isomorphism
of finite groups, then by a $\varphi$-isomorphism of lattices
$(\Gamma, L)$, $(\Gamma', L')$ we will mean an isomorphism
$\psi\colon L\isoto L'$ such that
\[ \psi(\gamma x)=\varphi(\gamma)\psi(x)\text{ for all
$\gamma\in\Gamma, x\in L$.} \]
By abuse of notation we will sometimes say that the lattices
$(\Gamma, M)$ and $(\Gamma', M')$ are {\em isomorphic} instead
of ``$\varphi$-isomorphic" in two special cases:
(i)  if $\Gamma = \Gamma'$ and $\varphi = \id$, or (ii)
$(\Gamma, M)$ and $(\Gamma', M')$ are $\varphi$-isomorphic for some
$\varphi \colon \Gamma \isoto \Gamma'$.

Now let $k$ be a field, $\Tspl = \bbG^d_{\text{\rm{m}}}$
the split $d$-dimensional $k$-torus,  and $\Gamma$  a finite group.
By a multiplicative action of $\Gamma$ on $\Tspl$
we mean an action by automorphisms of $\Tspl$ as an algebraic group over $k$.
Recall that the following objects are in a natural bijective
correspondence:
\begin{enumerate}[\upshape(i)]
\item $\Gamma$-lattices of rank $d$ (up to isomorphism);
\item integral representations $\phi \colon \Gamma \to
\GL_d(\mathbb{Z})$ (up to conjugacy in $\GL_d(\mathbb{Z})$);
\item multiplicative actions $\Gamma \to \Aut_{\text{$k$-grp}}(\Tspl)$
(up to an automorphism of $\Tspl$ as an algebraic $k$-group).
\end{enumerate}

A $\Gamma$-lattice $L$ is called {\em permutation} if it
has a $\bbZ$-basis permuted by $\Gamma$.
We say that two $\Gamma$-lattices $L$ and $L'$ are
{\em equivalent}, and write $L\sim L'$,
if there exist short exact sequences
\[
0\to L\to E\to P\to 0 \quad \quad {\rm and} \quad \quad
0\to L'\to E\to P'\to 0
\]
with the same $\Gamma$-lattice $E$, where $P$ and $P'$
are permutation $\Gamma$-lattices.
For a proof that this is indeed an equivalence relation,
see \cite[Lemma 8, p.~182]{CTS1}.
Note that if there exists a short exact sequence
\[
0\to L\to L'\to Q\to 0,
\]
where $Q$ is a permutation $\Gamma$-lattice, then
the trivial short exact sequence
\[
0\to L'\to L'\to 0\to 0
\]
shows that $L\sim L'$.
In particular, if $P$ is a permutation $\Gamma$-lattice, then the  short exact sequence
\[
0\to 0\to P\to P\to 0
\]
shows that $P\sim 0$.
If $\Gamma$-lattices $L,L',M,M'$ satisfy $L\sim L'$ and $M\sim M'$
then $L\oplus M\sim L'\oplus M'$.

\begin{definition} \label{def.qp}
A $\Gamma$-lattice $L$ is called {\em quasi-permutation}
if it is equivalent to a permutation lattice, i.e.,
if there exists a short exact sequence
\begin{equation}\label{eq:quasi-projective}
0\to L\to P\to P'\to 0,
\end{equation}
where both $P$ and $P'$ are permutation $\Gamma$-lattices.
\end{definition}

\begin{lemma} \label{lem:quasi-perm}
Let $\Gamma_1\twoheadrightarrow \Gamma$ be a surjective homomorphism of finite groups, and let $L$ be a $\Gamma$-lattice.
Then $L$ is quasi-permutation as a $\Gamma_1$-lattice if and only if it is quasi-permutation as a $\Gamma$-lattice.
\end{lemma}

\begin{proof}
It suffices to prove ``only if''.
Assume that $L$ is quasi-permutation as a $\Gamma_1$-lattice
and let $\Gamma_0$ denote the kernel $\ker[\Gamma_1\to\Gamma]$.
From the short exact sequence \eqref{eq:quasi-projective} of $\Gamma_1$-lattices, where $P$ and $P'$ are some permutation $\Gamma_1$-lattices,
we obtain the  $\Gamma_0$-cohomology exact sequence
$$
0\to L\to P^{\Gamma_0}\to (P')^{\Gamma_0}\to 0
$$
(because $L^{\Gamma_0}=L$ and  $H^1(\Gamma_0,L)=0$), which is a short exact sequence of $\Gamma$-lattices.
It is easy to see that $P^{\Gamma_0}$ and $(P')^{\Gamma_0}$ are permutation $\Gamma$-lattices,
thus $L$ is a quasi-permutation $\Gamma$-lattice.
\end{proof}

\begin{definition} We say two algebraic varieties
$X$ and $Y$ defined over $k$ and both equipped with a $\Gamma$-action are $\Gamma$-equivariantly
stably birationally isomorphic if there exist $r,s\ge 0$
such that $X\times \Gm^r$ is $\Gamma$-equivariantly birationally
isomorphic to $Y\times \Gm^s$ where the $\Gamma$-actions on $\Gm^r$
and $\Gm^s$ are both trivial.
\end{definition}

\begin{proposition} \label{prop.stabeq} Let $L,M$ be two faithful
$\Gamma$-lattices, and $T_L$ and $T_M$ the associated split $k$-tori
with associated multiplicative $\Gamma$-actions
(i.e., $\XX(T_L)=L$ and $\XX(T_M)=M$\,).
The following statements are equivalent:
\begin{enumerate}
\item[(i)] $L\sim M$.
\item[(ii)]$T_L$ and $T_M$ are
$\Gamma$-equivariantly stably birationally isomorphic.
\end{enumerate}
\end{proposition}

\begin{proof} Since the function field of $T_L$ is $k(L)$, the field of fractions of the group algebra $k[L]$ of $L$,  this
follows from~\cite[Proposition 1.4]{LL00}.
\end{proof}

\begin{definition}
We say that a $\Gamma$-action on an algebraic variety $X$,
defined over $k$,
is {\em linearizable} (respectively, {\em stably linearizable}) if
$X$ is $\Gamma$-equivariantly birationally isomorphic
(respectively, $\Gamma$-equivariantly stably birationally
isomorphic) to a finite-dimensional $k$-vector space $V$
with a linear $\Gamma$-action.
\end{definition}

\begin{remark} \label{rem.no-name}
By the no-name lemma any two faithful linear actions
of a finite group $\Gamma$ on $k$-vector spaces $V_1$ and
$V_2$ are stably $\Gamma$-equivariantly birationally
equivalent; see, e.g., \cite[Lemma 2.12(c)]{LPR06}. This
makes stable linearizability a particularly natural notion.
\end{remark}

\begin{lemma} \label{lem.stably-linearizable}
Let $L$ be a  $\Gamma$-lattice,
and let $T_L$ be the associated split $k$-torus
with multiplicative $\Gamma$-action.
\begin{enumerate}[\upshape(a)]
\item  If $L$ is a permutation lattice then
the $\Gamma$-action on $T_L$ is linearizable.
\item  $L$ is quasi-permutation if and only if
the $\Gamma$-action on $T_L$ is stably linearizable.
\end{enumerate}
\end{lemma}

\begin{proof}
(a) Suppose $L \simeq \mathbb Z[S]$ for some finite $\Gamma$-set $S$.
Let $V$ be the $k$-vector space with basis  $(e_s)_{s\in S}$.
Then $V$ carries a natural (permutation) $\Gamma$-action. The morphism
$T_L \to V$ given by \[ t \to \sum_{s \in S} s(t) e_s \]
is easily seen to be a $\Gamma$-equivariant birational isomorphism.

(b) By Lemma \ref{lem:quasi-perm} we may assume that $\Gamma$ acts faithfully on $L$.
Let $P$ be a faithful permutation $\Gamma$-lattice (e.g.,
$P = \mathbb Z[\Gamma]$).
Let $V$ be the linear representation of $G$ constructed
in part (a).  It now suffices to show that the following conditions
are equivalent:
\begin{enumerate}[\upshape(i)]
\item $L$ is quasi-permutation,
\item $L \sim P$,
\item $T_L$ and $T_P$ are $\Gamma$-equivariantly stably birationally
isomorphic,
\item $T_L$ and $V$ are $\Gamma$-equivariantly stably birationally isomorphic,
\item $T_L$ is stably linearizable.
\end{enumerate}

\noindent
Indeed, (i) and (ii) are equivalent by Definition~\ref{def.qp}.
Conditions (ii) and (iii) are equivalent
by Proposition~\ref{prop.stabeq}.
In the proof of part (a) we showed that $T_P$ and $V$
are $\Gamma$-equivariantly birationally isomorphic. Consequently,
(iii) is equivalent to (iv). Finally, (iv) $\Longrightarrow$ (v)
by definition, and (v) $\Longrightarrow$ (iv) by
the no-name lemma; see Remark~\ref{rem.no-name}.
\end{proof}

\begin{lemma}[cf.~\cite{LPR06}, Proposition 4.8]
\label{LPR-reduction-two}
  Let $W_1,\dots,W_m$ be finite groups.
  For each $i=1,\dots,m$, let $V_i$ be a finite-dimensional
  $\Q$-representation of $W_i$. Set $V:=V_1\oplus\dots\oplus V_m$.
  Suppose $L\subset V$ is a free abelian subgroup, invariant under
  $W :=W_1\times\dots\times W_m$.

  If $L$ is a quasi-permutation $W$-lattice,
  then  $L_i:=L\cap V_i$ is a quasi-permutation $W_i$-lattice,
for each $i=1, \dots, m$.
  \end{lemma}

  \begin{proof}
  It suffices to prove the lemma for $i=1$.
  Set $V' := V/V_1=V_2\oplus\dots\oplus V_m$ and $L'=L/L_1\subset V'$.
  Then  $W_1$ acts trivially on $V'$ and on $L'$,
  in particular,  $L'$ is a permutation $W_1$-lattice.
  It follows from the short exact sequence of $W_1$-lattices
  \[
  0\to L_1\to L\to L'\to 0
  \]
  that the $W_1$-lattices $L_1$ and $L$ are equivalent.

Now assume that $L$ is a quasi-permutation $W$-lattice.
Then it is a quasi-permutation $W_1$-lattice, and hence so is $L_1$.
\end{proof}

\begin{lemma}[cf.~\cite{LPR06}, Lemma 4.7]
\label{LPR-reduction-product}
  Let $W_1,\dots,W_m$ be finite groups.
  For each $i=1,\dots,m$, let $L_i$ be a $W_i$-lattice.
  Set  $W :=W_1\times\dots\times W_m$ and
construct a $W$-lattice $L:=L_1\oplus \dots\oplus L_m$.

  Then $L$  is a quasi-permutation
  $W$-lattice if and only if
  $L_i$ is a quasi-permutation $W_i$-lattice
  for each $i=1,\dots, m$.
  \end{lemma}

  \begin{proof}
  The ``if" assertion is obvious from the definition.
  The ``only if'' assertion follows from Lemma \ref{LPR-reduction-two}.
  \end{proof}

\begin{lemma}\label{lem:Voskresenskii}
Let $\Gamma$ be a finite group and $L$ a $\Gamma$-lattice
of rank $1$ or $2$.  Then $L$ is quasi-permutation.
\end{lemma}

\begin{proof} This is easily deduced
from ~\cite[\S4.9, Examples 6, 7]{Voskresenskii-book}.
\end{proof}

\section{Automorphisms and semi-automorphisms of split reductive groups}
\label{sect.aut-group}

\subsection{Notational conventions}
\label{subsec:B,T}

Let $G$ be a split reductive group over a field $k$.
We will write $T$ for a maximal $k$-torus of $G$, $B$ for a Borel subgroup,
$Z=Z(G)$ for the center of $G$, $G^\ad$ for $G/Z$, and $\ T^\ad$ for $T/Z$.
We identify $G^\ad$  with the algebraic group
$\Inn(G)$ of inner automorphisms of $G$.
If $g\in G^\ad(k)$ (or $g\in T^\ad(k)$), we write
$\inn(g)$ for the corresponding inner automorphism of $G$.

We will sometimes refer to a pair $(T, B)$, where $T$ is a split
maximal $k$-torus and $T \subset B \subset G$ is a Borel subgroup
defined over $k$, as a {\em Borel pair}.
It is well known that the natural action of $G^\ad(k)$
on the set of Borel pairs is transitive and that the stabilizer
in $G^\ad(k)$ of a Borel pair $(T,B)$ is $T^\ad(k)$.

Given a split maximal torus $T\subset G$, let
$\rPsi(G,T) :=(X,X^\vee, R, R^\vee)
$
be the {\em root datum} of $(G,T)$. Here
$X=\XX(T)$ is the character group of $T$,  $X^\vee=\Hom(X,\Z)$
is the cocharacter group of $T$, $R=R(G,T)\subset X$ is the root
system of $G$ with respect to $T$, and  $R^\vee\subset X^\vee$ is the
coroot system of $G$ with respect to $T$.
The bijection $R\to R^\vee$ sending a root to the corresponding coroot
is a part of the root datum structure.
For details, see
\cite[\S1.1]{Springer79} or \cite[\S7.4]{Springer-book}.

Given a Borel pair $(T, B)$, let
$
\rXi(G,T,B) :=(X,X^\vee, R, R^\vee, \Delta,\Delta^\vee)
$
be the {\em based root datum of} $(G,T,B)$. Here
$\Delta\subset R$ is the basis of $R$ defined by $B$,
and $\Delta^\vee\subset R^\vee$ is the
corresponding basis of $R^\vee$.
For details, see \cite[\S1.9]{Springer79}.

The automorphism group $\Aut(G)$ is known to carry
the structure of a $k$-group scheme; note however,
that this $k$-group scheme may not be of finite type.
The automorphism groups $\Aut\rPsi(G, T)$ and
$\Aut\rXi(G, T, B)$ are closed group subschemes of
$\Aut(G)$ defined over $k$. These $k$-group schemes are
discrete, in the sense that their identity components
are trivial.

\subsection{Semi-automorphisms}
\label{subsec:G,T,B}
Let $\Gbar$ be a reductive group over an algebraic
closure $\kbar$ of $k$.  We denote by $\SAut(\Gbar)$ the group
of $\kbar/k$-semi-automorphisms of $\Gbar$. We view
$\SAut(\Gbar)$ as an abstract group.
For a definition of a semi-automorphism, see~\cite[\S 1.1]{Borovoi-Duke}
or \cite[\S 1.2]{FSS}. (Note that in these papers semi-automorphisms
are called ``semialgebraic" and ``semilinear" automorphisms, respectively.)
If $G$ is a $k$-form of $\Gbar$, then any
element $\sigma\in\Gal(\kbar/k)$ defines a $\sigma$-semi-automorphism
$\sigma_*\colon \Gbar\to\Gbar$, and any semi-automorphism of $\Gbar$
is of the form $a=\alpha\circ\sigma_*$ where $\sigma\in\Gal(\kbar/k)$ and $\alpha\colon\Gbar\to\Gbar$
is a $\kbar$-automorphism of the $\kbar$-group $\Gbar$.

Fix $(\Tbar,\Bbar)$ as above. For any $a \in \SAut(\Gbar)$ there
exists $g\in \Gbar^\ad(\kbar)$ such
that $\inn(g) (a(\Tbar),a(\Bbar))=(\Tbar,\Bbar)$.
The semi-automorphism $\inn(g) a$ of $\Gbar$ defines
a semi-automorphism of $\Tbar$ depending only on $a$
(since the coset $\Tbar^\ad g$ is uniquely determined).
The automorphism of $X=\XX(\Tbar)$ induced by $\inn(g) a$
preserves $R=R(\Gbar,\Tbar)$ and $\Bbar$ and thus permutes
the elements of the basis $\Delta$ of $R$ defined by $\Bbar$.
In other words, it gives rise to an automorphism
$
\rXi(\Gbar,\Tbar,\Bbar) \to \rXi(\Gbar,\Tbar,\Bbar),
$
depending only on $a$, which we denote by $\varphi_{\Tbar, \Bbar}(a)$.

\begin{proposition} \label{prop.semi}

\begin{enumerate}[\upshape(a)]
\item $\varphi_{\Tbar,\Bbar} \colon \SAut(\Gbar) \to \Aut \rXi(\Gbar, \Tbar, \Bbar)$
is a group homomorphism.

\item  $\Inn(\Gbar) \subset \Ker(\varphi_{\Tbar,\Bbar})$.

\item Suppose $(\Tbar', \Bbar')$ is another Borel pair for $\Gbar$.
Choose $u \in \Gbar^\ad(\kbar)$ so that $(\Tbar, \Bbar) = \inn(u)(\Tbar', \Bbar')$.  Then
the following diagram commutes:
\[ \xymatrix{
   & & \Aut \rXi(\Gbar, \Tbar, \Bbar) \ar[dd]^{\inn(u)^\ast} \\
\SAut (\Gbar) \ar[rrd]^-{\varphi_{\Tbar', \Bbar'}} \ar[rru]^-{\varphi_{\Tbar,\Bbar}}
& & \\
  & & \Aut \rXi(\Gbar, \Tbar', \Bbar').} \]
Moreover, the automorphism $\inn(u)^*$ in this diagram
is independent of the choice of $u$.
\end{enumerate}
\end{proposition}

\begin{proof} (a) Given $a_1 , a_2 \in  \SAut(\Gbar)$, choose
$g_1, g_2 \in \Gbar^\ad$ so that $$\inn(g_i) \, a_i(\Tbar,\Bbar) =
(\Tbar,\Bbar).$$ Then $\inn(g_1) \, (a_1 \, \inn(g_2) \, a_1^{-1})
\in \Inn(\Gbar)$; denote this inner automorphism by $\inn(g)$ for
some $g \in \Gbar^{\ad}$. Then $\inn(g) a_1 a_2 (\Tbar,\Bbar) =
(\Tbar,\Bbar)$ and thus
\[ \varphi_{\Tbar,\Bbar}(a_1 a_2) = \inn(g) \, a_1 a_2 = \inn(g_1) \, a_1
\, \inn(g_2) \, a_2 = \varphi_{\Tbar,\Bbar}(a_1) \,
\varphi_{\Tbar,\Bbar}(a_2) \, . \]
Therefore, $\varphi_{\Tbar,\Bbar}$ is a homomorphism.

\smallskip
(b) is obvious from the definition.

\smallskip
(c) Let $a \in \SAut(\Gbar)$.
By our choice of $u \in \Gbar^\ad$, we have  $(\Tbar,\Bbar)= \inn(u) (\Tbar',\Bbar')$.
Choose $g \in \Gbar^\ad$ such  that $\inn(g) \, a (\Tbar,\Bbar) =
(\Tbar,\Bbar)$.
Then
$$
\inn(u^{-1}) \, \inn(g) \, (a \, \inn(u) \, a^{-1}) \in \Inn(\Gbar);
$$
denote this automorphism by $\inn(g')$ for some $g' \in \Gbar^\ad$.
One readily checks that $\inn(g') a (\Tbar',\Bbar') = (\Tbar',\Bbar')$
and thus \[ \varphi_{\Tbar', \Bbar'}(a)= \inn(g') \, a =
\inn(u^{-1}) \, \inn(g) \, a \, \inn(u) =
\inn(u^{-1}) \, \varphi_{\Tbar,\Bbar}(a) \, \inn(u) , \]
as desired. To prove the last assertion of part (c),
note that the coset $uT$ is independent of the choice of $u$.
Hence, so is the map $\inn(u)^*$ in the diagram.
\end{proof}

\subsection{Automorphisms of split reductive groups}
\label{subsec:split-triple}

\begin{proposition}[cf.~{\cite[Expos\'e XXIV, Theorem 1.3]{SGA3}}]
\label{prop.auto}
Let $G$ be a split  reductive group defined over $k$,
$T\subset G$ a split maximal torus, and $B \supset T$
a Borel subgroup of $G$ defined over $k$. Set $\overline{G} := G \times_k \kbar$.
\begin{enumerate}[\upshape(a)]
\item The composite homomorphism of abstract groups
$$\phi_{T, B} \colon \Aut(\Gbar)\into \SAut(\Gbar) \labelto{\varphi_{\Tbar,\Bbar}} \Aut \rXi(G, T, B)$$
admits a $\Gal(\kbar/k)$-equivariant splitting (homomorphic section) $\psi$ of the form
\[ \psi\colon \Aut\rXi(G, T, B) \hookrightarrow \Aut(G, T, B) \hookrightarrow
\Aut(\Gbar). \]
Here $\Aut(G, T, B)$ denotes the subgroup of $\Aut(G)$ consisting
of automorphisms that preserve the Borel pair $(T, B)$.

\item The homomorphism $\phi_{T, B}$ of part (a) fits into a split
short exact sequence of abstract groups
\[ 1 \longrightarrow \Inn(\Gbar) \longrightarrow \Aut(\Gbar)
\labelto{\phi_{T, B}} \Aut\rXi(G, T, B) \longrightarrow 1,  \]
which comes from a split short exact sequences of group schemes over $k$
\begin{equation} \label{eq:SGA3}
  1 \longrightarrow G^\ad \longrightarrow \Aut(G) \labelto{\phi}
\Aut\rXi(G, T, B) \longrightarrow 1.
\end{equation}
\end{enumerate}
\end{proposition}

Note that since $T$ is split over $k$, the $\Gal(\kbar/k)$-action on
$\Aut\rXi(G, T, B)$ is trivial.

\begin{proof} (a) Recall that a {\em pinning} of $(G,T,B)$ is
a choice of a nonzero $X_\alpha\subset \ggg_\alpha$ for each
$\alpha\in \Delta$, where
$$
\Lie(G)=\Lie(T)\oplus \bigoplus_{\alpha\in R} \ggg_\alpha
$$
is the root decomposition, and $\Delta$ is the basis of $R=R(G,T)$
associated with $B$. By the isomorphism theorem,  see \cite[Expos\'e
XXIII, Theorem 4.1]{SGA3} or \cite[Proposition~1.5.5]{Conrad},
the  canonical homomorphism
$$
\Aut(G,T,B, (X_\alpha)_{\alpha\in\Delta})\,\to \Aut\,\rXi(G, T, B)
$$
is an isomorphism.  Composing the inverse isomorphism with
the natural embeddings
\[ \Aut(G,T,B, (X_\alpha)_{\alpha\in\Delta}) \hookrightarrow
\Aut(G, T, B) \hookrightarrow \Aut(G)\into\Aut(\Gbar), \]
we obtain a section $\psi$ of $\phi_{T, B}$ of the desired form.

(b) See \cite[Expos\'e XXIV, Proof of Theorem 1.3]{SGA3}.
\end{proof}

\begin{corollary}\label{cor:M}
  Every  abstract subgroup $\M\subset\Aut(\Gbar)$,
containing $\Inn(\Gbar)$ as a subgroup of finite index,
is of the form $\M=M(\kbar)$ for some linear algebraic
$k$-group $M=\Inn(G)\rtimes A\subset\Aut(G)$.
Here $A \subset \Aut(G,T,B)$ is a finite $k$-group, all
of whose $\kbar$-points are defined over $k$.
\end{corollary}

\begin{proof}
Set $A':= \phi_{T,B}(\M)\subset \Aut \rXi(G, T, B)$. Then $A'$ is
a finite algebraic $k$-group all of whose $\kbar$-points are defined
over $k$.  Set $M=\phi^{-1}(A')\subset\Aut(G)$, where
$\phi\colon \Aut(G)\to \Aut \rXi(G, T, B)$ is a homomorphism
of $k$-group schemes, as in~\eqref{eq:SGA3}.
Then $M$ is a $k$-group scheme and $M(\kbar)=\M$.  Set
$A=\psi(A')\subset \Aut(G,T,B)$, where $\psi$ is
the splitting of Proposition~\ref{prop.auto}, then
$M= \Inn(G)\rtimes A$.  Since $M$ has finitely many connected
components, and the identity component $G^\ad$ of $M$ is
an affine algebraic $k$-group, we conclude
that $M$ is affine algebraic as well.
In other words, $M$ is a linear algebraic $k$-group, as desired.
\end{proof}

\section{The character lattice and the generic torus}
\label{sect.character-lattice}

Throughout this section
$G$ will denote a (connected) reductive $k$-group, not necessarily split,
and $T \subset G$ will denote a maximal $k$-torus.
We write $\Gbar:=G\times_k \kbar$, $\Tbar=T\times_k \kbar$,
and choose a Borel subgroup $\Bbar\supset \Tbar$ of $\Gbar$.

\subsection{The character lattice of a reductive group}

\begin{definition} \label{def:character-lattice}
{\rm (a)} We define $A_{T, \Bbar}$ to be the image of the composite homomorphism
$$\Gal(\kbar/k)\into\SAut(\Gbar)\labelto{\varphi_{\Tbar,\Bbar}}
\Aut\,\rXi(\Gbar,\Tbar,\Bbar) \hookrightarrow \Aut \, \XX(\Tbar).$$

\smallskip
{\rm(b)} We define the {\em extended Weyl group} $\V(G, T, \Bbar)$ by
  $$
  \V(G, T, \Bbar):=
W(\Gbar, \Tbar)\cdot A_{ T ,\Bbar} \subset \Aut\, \XX(\Tbar).
  $$
Note that $\V(G, T,\Bbar)$ is a subgroup  of
$\Aut\,\rPsi(\Gbar,\Tbar)$ (and therefore of $\Aut\,\XX(\Tbar)$),
because $W(\Gbar,\Tbar)$ is normal in $\Aut\,\rPsi(\Gbar,\Tbar)$. We
call the pair $$(\V(G, T, \Bbar), \XX(\Tbar))$$ {\em the character
lattice of} $G$.
\end{definition}

\begin{remark} \label{rem.SW}
Let $T'\subset G$ be another maximal $k$-torus, and $\Bbar'\supset \Tbar'$
a Borel subgroup of $\Gbar$. Then
it is easy to see that for $u$ as in Proposition \ref{prop.semi}(c),
the isomorphism $\inn(u)^*\colon \XX(\Tbar)\to \XX(\Tbar')$
induces an isomorphism of groups
$$
A_{\Tbar,\Bbar}\isoto A_{\Tbar',\Bbar'}
$$
and an isomorphism of lattices
$$ (\V(G, T, \Bbar), \XX(\Tbar))
\isoto ( \V(G, T', \Bbar')), \XX(\Tbar')).$$ In other words, the
character lattice $( \V(\Gbar, \Tbar,\Bbar), \XX(\Tbar))$ is defined
unique\-ly up to a canonical isomorphism.

Moreover, if $T = T'$ then $A_{ T, \Bbar'} = w
A_{ T, \Bbar} w^{-1}$ for some $w \in W(\Gbar, \Tbar)$.
Thus different choices of $\Bbar$ give rise to the same (and not
just isomorphic)
subgroups $\V(G, T,\Bbar)$ of $\Aut\,\XX(\Tbar)$. For this reason
we will write $\V(G, T)$ in place of $\V(G, T, \Bbar)$ from
now on.
\end{remark}

\begin{lemma} \label{lem.SW1}
$\V(G, T)=W(\Gbar, \Tbar)\cdot\im\,\lambda_T$, where
$$
\lambda_T \colon \Gal(\kbar/k) \to \Aut\,\XX(\Tbar)
$$
is the usual action of the Galois group on the characters of $T$.
\end{lemma}

\begin{proof}
By the definition of $\varphi_{\Tbar,\Bbar}$, for any
$\sigma\in\Gal(\kbar/k)$ there exists $w_\sigma\in W(\Gbar,\Tbar)$
such that $
\varphi_{\Tbar,\Bbar}(\sigma)=w_\sigma\,\lambda_T(\sigma). $
\end{proof}

\begin{corollary} \label{cor.SW1}
Suppose $G$ is a reductive $k$-group,
$T$ is a maximal $k$-torus, and $K/k$ is a field extension such
that $k$ is algebraically closed in $K$. Then $\V(G, T) = \V(G_K, T_K)$
as subgroups of $\Aut \, \XX(\overline{T})=\Aut\,\XX(T_{\Kbar})$.
\end{corollary}

\begin{proof} By Lemma~\ref{lem.SW1}, it suffices to show that
$\im \, \lambda_T = \im \, \lambda_{T_K}$.
Since $T$ splits over $\kbar$, the action of the
Galois group $\Gal(\Kbar/K)$
on $\XX(\Tbar)$ factors through the natural
homomorphism $\Gal(\Kbar/K)\to\Gal(\kbar/k)$.
By \cite[Theorem VI.1.12]{lang}, this homomorphism
is surjective. Thus
$\im \, \lambda_T = \im \, \lambda_{T_K}$, as desired.
\end{proof}

\begin{lemma} \label{lem.SW2}
Let $T$ be a maximal $k$-torus of $G$, and $\Bbar \supset \Tbar$  a Borel
subgroup of $\Gbar$. Then $\V(G,T)$ is a semi-direct product:
  $\V(G, T)=W(\Gbar,\Tbar)\rtimes A_{ T, \Bbar}$.
\end{lemma}

\begin{proof}
Since $A_{ T, \Bbar}\subset\Aut\,\rXi(\Gbar,\Tbar,\Bbar)$, every element of
$A_{ T,\Bbar}$ preserves the basis $\Delta$  of $R(\Gbar, \Tbar)$ corresponding to $\Bbar$,
while in $W(\Gbar, \Tbar)$ only the identity element $1$ preserves $\Delta$.
Thus $W(\Gbar, \Tbar)\cap A_{ T, \Bbar}=\{1\}$.
By definition  $W(\Gbar, \Tbar)\cdot A_{ T,\Bbar}=\V(G, T)$,
and the lemma follows.
\end{proof}

\subsection{The generic torus}

Let $T$ be a maximal $k$-torus of $G$ and
$T_\gen$ the generic torus of $G$.
Recall that $T_\gen$ is defined over the field $K_{\gen} := k(G/N_G(T))$,
where $N_G(T)$ denotes the normalizer of $T$ in $G$. For details of
this construction, see~\cite[\S4.2]{Voskresenskii-book}. For notational
simplicity we will write $K$ in place of $K_{\gen}$ for the remainder
of this section.

\begin{proposition}
\label{prop:Voskr}
Let $G$ be a  reductive $k$-group and $T$ a maximal $k$-torus.
Then
\begin{enumerate}[\upshape(a)]
\item  the image $\AAA$ of $\Gal(\Kbar/K)$ in
$\Aut\,\XX(\overline{T_\gen})$ coincides with the extended Weyl
group $\V(G_K, T_\gen)$.

\item The character lattice
$(\AAA, \XX(\overline{T_{\gen}}) )$ of the generic torus is isomorphic
to the character lattice of $G$.
\end{enumerate}
\end{proposition}

If $G$ is semisimple then the proposition is an immediate consequence of
a theorem of Voskresenski\u\i's~\cite[Theorem~4.2.2]{Voskresenskii-book};
cf.~Lemma~\ref{lem.SW2}.

\begin{proof}
(a) We claim that the image of the Galois group $\Gal(\Kbar/\kbar K)$
in $\Aut \, \XX(\overline{T_\gen})$ coincides with the Weyl
group $W(\overline{G_K}, \overline{T_\gen})$. If $G$ is semisimple
this is Theorem 4.2.1 in~\cite{Voskresenskii-book}. In the general
case, we consider the derived subgroup $G^\der = [G, G]$ of $G$
which is a connected semisimple group. Consider the radical $R$ of
$G$ (the identity component of the center), which is a $k$-torus.
The generic torus $T_\gen$ of $G$ and the generic
torus $T_\gen' \subset G_{K}^\der $ of $G^\der$ are defined over the
same field $K = k(G/N_G(T)) = k(G^\der/N_{G^\der}(T \cap G^\der))$.
Note that $T_\gen=T_\gen ' \cdot R_K$ and $T_\gen ' \cap R_K$ is
finite. Hence, there is a canonical isomorphism
$$
\XX(T_\gen)\otimes\Q =\XX(T_\gen ')\otimes\Q\ \oplus\ \XX(R_\Kbar)\otimes \Q
\, ,
$$
where $\XX(T_\gen)$ stands for $\XX(\overline{T_\gen})$.
Let
\begin{align*}
\rho\colon& \Gal(\Kbar/\kbar K)\to\Aut\,\XX(T_\gen)\otimes\Q, \\
\rho'\colon& \Gal(\Kbar/\kbar K)\to\Aut\,\XX(T_\gen ')\otimes \Q
\end{align*}
be the corresponding actions.
Since $R_K$  splits over $\kbar K$,
the Galois group
$\Gal(\Kbar/\kbar K)$ acts trivially on $\XX(R_\Kbar)$.
Hence, for every $\sigma\in\Gal(\Kbar/\kbar K)$ we have
$$
\rho(\sigma)=(\rho'(\sigma),1)\ \in\   \Aut\, \XX(T_\gen ')\otimes\Q\ \times\ \Aut\,\XX(R_\Kbar)\otimes\Q
\ \subset\ \Aut\,\XX(T_\gen)\otimes\Q.
$$
By Voskresenski\u\i's theorem~\cite[Theorem
4.2.1]{Voskresenskii-book}, we have $$\im\,\rho'=W(G^\der_\Kbar,
T_\gen ')$$ and hence
$$\im\,\rho=W(G^\der_\Kbar, T_\gen ')\times\{1\}=W(G_\Kbar,{T_\gen}).$$
This proves the claim.

Now recall that by Lemma~\ref{lem.SW1}, $\V(G_{K}, T_\gen)$ is
generated by $\AAA$ and $W(\overline{G_{K}}, {T_\gen})$. The claim
tells us that, in fact, $W(\overline{G_{K}}, {T_\gen}) \subset \AAA$.
Hence,   $\V(G_{K}, T_\gen)=\AAA$.

(b) Consider two maximal tori in $G_{K}$, $T_\gen$ and $T_{K} = T
\times_k K$. By Remark~\ref{rem.SW} the lattices
\[ \text{$(\V(G_{K}, T_\gen), \XX({\overline{T_\gen}}))$
and
$ ( \V(G_{K}, T_{K}), \XX({\overline{T_K}}))$} \]
are isomorphic.
By part (a), the character lattice
$(\AAA, \XX(\overline{T_{\gen}}))$ of the generic torus
coincides with $(\V(G_{K}, T_\gen), \XX(\overline{{T_\gen}}))$.
On the other
hand, since the $k$-variety $G/N_G(T)$ is absolutely irreducible,
$k$ is algebraically closed in $K = k(G/N_G(T))$. Thus by
Corollary~\ref{cor.SW1},
$( \V(G_{K}, T_{K}),   \XX(\overline{{T_K}}) )$
coincides with $(\V(G, T), \XX(\Tbar))$.
We conclude that the character lattice $(\AAA, \XX(\overline{T_{\gen}}))$ of the generic torus
is isomorphic to the lattice  $(\V(G, T), \XX(\Tbar))$,
which is the character lattice of $G$.
\end{proof}

\section{Forms of reductive groups}
\label{sect.forms}

Let $\Gspl$ be a split  reductive $k$-group.
Recall that any $k$-form $G$ of $\Gspl$ is $k$-isomorphic to
a twisted group $_z \Gspl$ for some cocycle
$z\in Z^1(k,\Aut(\overline{\Gspl}))$.  Sending $z$ to $_z \Gspl$ gives
rise to a natural bijective correspondence between
the non-abelian Galois cohomology set $H^1(k, \Aut(\overline{\Gspl}))$
and the isomorphism classes of $k$-forms of $\Gspl$.
For details on this, see e.g.  \cite[\S\S11.3 and 12.3]{Springer-book}.

\subsection{Choosing a ``small" cocycle}\label{subsec:cocycle}

Let $G$ be a  reductive $k$-group, not necessarily split.
Let $T\subset G$ be a maximal torus, and let $\Bbar\supset \Tbar$ be a Borel subgroup.
Let $G_\spl$ be a split $k$-form of $G$.
We choose and fix
a $\kbar$-isomorphism $\theta \colon\overline{\Gspl}\to \Gbar$.
Choose a Borel pair $(T_\spl, B_\spl)$ in $G_\spl$.
After composing $\theta$ with an inner automorphism of
$\Gbar$, we may (and shall) assume that
$\theta$ takes $(\overline{\Tspl},\overline{\Bspl})$ to $(\Tbar,\Bbar)$.
Then $\theta$ induces
isomorphisms $\Aut(\overline{G}) \to \Aut(\overline{\Gspl})$,
$\rXi(\Gbar, \Tbar, \Bbar) \to \rXi(\overline{\Gspl}, \overline{\Tspl},
\overline{\Bspl})=\rXi({\Gspl}, {\Tspl},{\Bspl})$, etc.

\begin{definition}\label{def:M-G}
Let $G$, $G_\spl$ and $\theta$ be as above.
Let $A_{ T, \Bbar}$ denote
the image of $\Gal(\kbar/k)$ in $\Aut \rXi(\Gbar, \Tbar, \Bbar)$,
as in Definition~\ref{def:character-lattice}, it is a finite group.
Note that $\theta$ induces an isomorphism
$$\theta_*\colon\Aut \rXi(\Gbar, \Tbar, \Bbar)\isoto
\Aut \rXi({\Gspl}, \Tspl, \Bspl).$$
Set $\tA :=\theta_*(A_{T,\Bbar}) \subset \Aut \rXi({\Gspl}, \Tspl, \Bspl)$.
We define $M_G\subset \Aut({\Gspl})$ to be  the preimage  in $\Aut({\Gspl})$
of the finite group $\tA$  under
\[ \phi\colon\Aut(\Gspl) \to \Aut\rXi(\Gspl, \Tspl, \Bspl) \, ; \]
see  exact sequence \eqref{eq:SGA3} of Proposition~\ref{prop.auto}(b)
(for $\Gspl$).
Then $M_G$
is an algebraic group defined over $k$; see Corollary \ref{cor:M}.
Set
\[ ^\theta\V :=\theta_*(\V(G,T))\subset\Aut\,\XX(\Tspl) \, , \]
so that $^\theta \V=W(\Gspl,\Tspl)\cdot \tA$.
Note that
the group $^\theta\V$ acts multiplicatively (i.e., by group automorphisms)
on the split $k$-torus $T_\spl$.
\end{definition}

\begin{proposition}\label{prop:cocycle}
With the notation of Definition \ref{def:M-G}, $G$ is isomorphic to
$_z G_\spl$ for some cocycle $z\in Z^1(k,M_G)$.
\end{proposition}

\begin{proof}
For $\sigma\in\Gal(\kbar/k)$ denote by $\beta(\sigma)$ the semi-automorphism
of $\Gbar$ and by $\beta_\spl(\sigma)$ the semi-automorphism of
$\overline{\Gspl}$ induced by $\sigma$. Under the usual
correspondence between $k$-forms of $\Gspl$ and
$H^1(k, \Aut(\overline{\Gspl}))$, $G$ is $k$-isomorphic to
$_z G$, for the cocycle
$ z(\sigma) := \, ^\theta \! \beta(\sigma) \circ
\beta_\spl(\sigma)^{-1}
\colon \overline{\Gspl}\to\overline{\Gspl}$,
where $^\theta \! \beta(\sigma)$ is the image of
$\beta(\sigma)$ under the isomorphism
$\Aut(\overline{G}) \labelto{\simeq} \Aut(\overline{\Gspl})$
induced by $\theta$.

It remains to show that $z(\sigma)\in M_G(\kbar)$, or equivalently,
$z_{\rXi}(\sigma) : = \varphi_{\overline{\Tspl},
\overline{\Bspl}} \circ z(\sigma)$
lies in $\tA$, for every $\sigma\in\Gal(\kbar/k)$.
Consider the diagram
\[ \xymatrix{
& \SAut (\Gbar)  \ar[dd]_{\simeq}  \ar[rr]^-{\varphi_{\Tbar, \Bbar}} &
  & \Aut \rXi(\Gbar, \Tbar, \Bbar)  \ar[dd]_{\simeq}^{\theta_*}  \\
\Gal(\kbar/k) \ar@{^{(}->}[ru]^{\beta} \ar@{^{(}->}[rd]^{\beta_{\spl}} & & &  \\
  & \SAut (\overline{\Gspl})  \ar[rr]^-{\varphi_{\overline{\Tspl}, \overline{\Bspl}}} &
   &
  \Aut \rXi (\overline{\Gspl}, \overline{\Tspl}, \overline{\Bspl}),
} \]
where the vertical isomorphisms are induced by $\theta$. The
commutativity of this diagram tells us that
\[
z_\rXi(\sigma)= \theta_*(\gamma(\sigma))\circ
\gamma_{\spl}(\sigma)^{-1},
\]
where $\gamma := \varphi_{\Tbar, \Bbar} \circ \beta$
and $\gamma_{\spl} := \varphi_{\overline{\Tspl}, \overline{\Bspl}} \circ \beta_{\spl}$
denote the actions of $\Gal(\kbar/k)$ on $\rXi(\Gbar,\Tbar,\Bbar)$ and
on $\rXi(\overline{\Gspl}, \overline{\Tspl}, \overline{\Bspl})=\rXi(\Gspl, \Tspl, \Bspl)$,
respectively. Since
$\Gal(\kbar/k)$ acts trivially on
$\rXi(\Gspl,\Tspl,\Bspl)$, we see that $\gamma_{\spl}(\sigma)=\id$ and
$z_\rXi(\sigma)= \theta_*(\gamma (\sigma))$.
By definition,
the image of the homomorphism $\gamma$ is  $A_{T, \Bbar}$.
Thus the image of the homomorphism $z_\rXi$ is $\theta_*(A_{T, \Bbar}) =\tA$.
In particular, $z_\rXi(\sigma) \in \tA$, as desired.
\end{proof}

\begin{remark}
We can define the character lattice of $G$ using a split form $G_\spl$ of $G$
as follows.

Let $G,T,\Bbar, G_\spl,T_\spl,B_\spl,\theta$ be as at the beginning
of this Subsection \ref{subsec:cocycle}. Then we obtain a cocycle
$z$ with values in $\Aut(\overline{G_\spl})$ such that $G$ is
isomorphic to ${}_z G_\spl$, see the proof of Proposition
\ref{prop:cocycle}. Composing $z$ with the canonical homomorphism
{\footnotesize{
$$\Aut(\overline{G_\spl})\into \SAut(\overline{G_\spl})\to
\Aut \rXi(\overline{\Gspl}, \overline{\Tspl}, \overline{\Bspl})=\Aut
\rXi({\Gspl}, {\Tspl},{\Bspl}),$$ }} we obtain a cocycle
(homomorphism)
$$
z_\rXi\colon \Gal(\kbar/k)\to  \Aut\rXi({\Gspl}, {\Tspl},{\Bspl}).
$$
Set $$A':=\im\, z_\rXi\subset \Aut\rXi({\Gspl}, {\Tspl},{\Bspl})\subset\Aut\rPsi({\Gspl}, {\Tspl}),$$
and set $W'=W(G_\spl,T_\spl)\cdot A'\subset \Aut\rPsi({\Gspl}, {\Tspl})$,
then the proof of Proposition \ref{prop:cocycle} shows that $A'=\tA$, hence $W'=^\theta\!\V$,
and therefore the pair $(W', \XX(T_\spl))$ is isomorphic via $\theta_*$ to the character lattice $\X(G)$ of $G$.
\end{remark}

\subsection{Forms of Cayley groups}

\begin{lemma} \label{lem.outer-form}
Let $G$ be a split reductive $k$-group and
$M$ a closed algebraic $k$-subgroup  of the $k$-group
scheme $\Aut(G)$ such that  $\Inn(G)\subset M$.
Let $z \in Z^1(k, M)$.
\begin{enumerate}[\upshape(a)]
\item If there exists an $M$-equivariant birational isomorphism
$$f \colon G \dasharrow \Lie(G),$$ then $_z G$ is a Cayley group.

\item If there exists an $M$-equivariant birational
isomorphism $f \colon G \times_k \A^r \dasharrow \Lie(G) \times_k \A^r$
for some $r \ge 0$, where $M$ acts trivially on
the affine space $\A^r$, then $_z G$ is a stably Cayley group.

\item If $G$ is Cayley,
then any inner form of $G$ is also Cayley.
\item If $G$ is stably Cayley,
then any inner form of $G$ is also stably Cayley.
\end{enumerate}
\end{lemma}

\begin{proof}
(a) Since $f$ is $M$-equivariant,
we can twist $f$ by $z$ and obtain an $_z M$-equivariant
birational isomorphism
\[
  _z f \colon _z G \dasharrow _z\!\Lie(G) \, .
\]
By functoriality of the twisting operation,
$_z\!\Inn(G) = \Inn(_z G) \subset\, _z M$ (\cite[Lemma 16.4.6]{Springer-book})
and $_z\!\Lie(G) = \Lie(_z G)$. Thus $_z f$ is an $_z M$-equivariant
(and, in particular, $\Inn(_z G)$-equivariant)
rational map $_z G \dasharrow \Lie(_z G)$.  Twisting $f^{-1}$
by $z$ in a similar manner, we see that $_z f$ is, in fact,
a birational isomorphism, i.e.,  a Cayley map for $_z G$.

(b) Replace $G$ by $G \times \bbG_m^r$ and apply part (a).

(c) An inner form of $G$ is, by definition, a twisted form $_z G$,
where $z\in Z^1(k, \Inn(G))$. If $G$ is a Cayley group, then there
exists an $\Inn(G)$-equi\-va\-ri\-ant birational isomorphism $f
\colon G \dasharrow \Lie(G)$, hence by (a), $_z G$ is a Cayley
group.

(d) If $G$ is a stably Cayley group, then $G\times_k \bbG_m^r$ is Cayley for some $r$,
and we may identify $\Inn(G)$ with $\Inn(G\times_k \bbG_m^r)$.
If $z\in Z^1(k, \Inn(G))= Z^1(k, \Inn(G \times_k \bbG_m^r))$,
then by (b), the twisted group $_z(G\times_k \bbG_m^r)={}_z G\times_k
\bbG_m^r$ is Cayley, hence $_z G$ is stably Cayley.
\end{proof}

\section{$(G, S)$-fibrations and $(G, S)$-varieties}
\label{sect.(G,S)}

The proof of Theorem~\ref{thm.main1} in the next section relies on
the notions of $(G, S)$-fibration and $(G, S)$-variety. This section
will be devoted to preliminary material on these notions.

\subsection{$(G, S)$-fibrations}
Let $G$ be a linear algebraic $k$-group and $S$ a $k$-subgroup.
Recall that a $(G, S)$-fibration is a morphism of $k$-varieties
$\pi \colon X \to Y$, where $G$ acts on $X$ on the left, $\pi$ is constant
on $G$-orbits, and after a surjective \'etale base change $Y' \to Y$
there is a $G$-equivariant isomorphism between $G/S \times_k Y'$
and $X \times_Y Y'$ over $Y'$, cf. \cite[\S2.2]{CTKPR}.
If $S = \{ 1 \}$, then a $(G, S)$-fibration is the
same thing as a left $G$-torsor. Note that in general, $X \to Y$ can
be both a $(G, S_1)$-fibration and a $(G, S_2)$-fibration for
non-isomorphic $k$-subgroups $S_1, S_2 \subset G$. However
over an algebraic closure of $k$, $S_1$ and $S_2$ become
conjugate.

The following lemma generalizes well-known properties of torsors to
the category of $(G, S)$-fibrations.

\begin{lemma} \label{lem.fibrations1}
Let $\pi \colon X \to Y$, $\pi_1 \colon X_1 \to Y_1$ and $\pi_2
\colon X_2 \to Y_2$ be $(G,S)$-fibrations.
\begin{enumerate}[\upshape(a)]
\item[\rm (a)] Every $G$-equivariant morphism $f \colon X_1 \to X_2$ is a
morphism of $(G, S)$-fibrations, i.e., gives rise to a Cartesian
diagram
\[
  \xymatrix{ X_1 \ar[d]_{\pi_1} \ar[r]^{f}   & X_2 \ar[d]^{\pi_2} \\
  Y_1 \ar[r]^{\overline{f}}    & Y_2 \, . }
\]
In other words, $X_1 = X_2 \times_{Y_2} Y_1$, where the $G$-action on
$X_2 \times_{Y_2} Y_1$ is induced by the $G$-action on $X_2$.

\item[\rm (b)] Every $G$-invariant closed $($respectively, open$)$ subvariety
$X_0 \subset X$ is of the form $\pi^{-1}(Y_0)$ for some closed
$($respectively, open$)$ subvariety $Y_0$ of $Y$.
In particular,
$X_0$ is itself the total space of a $(G, S)$-fibration $\pi_{\, |
X_0} \colon X_0 \to Y_0$.
\item[\rm (c)] The map $f$ in part $($a$)$ is dominant
if and only if $\overline{f}$ is dominant.
\end{enumerate}
\end{lemma}

\begin{proof}
(a) We first define the map $\overline{f} \colon Y_1 \to Y_2$
locally in the \'etale topology on $Y_1$. Let $\{U_{\alpha}\}$ be an
\'etale open cover of $Y_1$ such that $X_1$ is $G$-equivariantly
isomorphic to $G/S \times_k U_{\alpha}$, over each $U_{\alpha}$. Then
over each $U_{\alpha}$, the map $\pi_1$ has a section $s_{\alpha}
\colon U_{\alpha} \to \pi_1^{-1}(U_{\alpha}) $, and we can
define $\overline{f}$ by composing $s$, $f$
and $\pi_2$. The resulting local map is independent of the choice
of $s$; these maps patch up to a $k$-morphism $\overline{f} \colon Y_1 \to
Y_2$ by \'etale descent.

By the universal property of fibered products there exists a
morphism $\phi \colon X_1 \to X_2 \times_{Y_2} Y_1$ over $Y_1$. This
morphism is unique and hence, $G$-equivariant.  Thus it suffices to show that
$\phi$ is an isomorphism. Note that $\phi$ is a $G$-equivariant
morphism between $(G, S)$-fibrations over $Y_1$.
We want to show that if $Y_1 = Y_2$ and $\overline{f} =
\operatorname{id}$ in the above diagram then $f$ is an isomorphism.
We do this by constructing $f^{-1}$. Let $\{ U_\alpha \}$ be an
\'etale local cover of $Y_1$, trivializing both $X_1$ and $X_2$.
That is,  over each $U_\alpha$, $X_1$ and $X_2$ are both
$G$-equivariantly isomorphic to $G/S \times_k U_\alpha$. Hence,
$f^{-1}$ is (uniquely) defined and is $G$-equivariant over each
$U_{\alpha}$.  Once again, using \'etale descent, we see that these
local inverses patch together to a  well-defined $G$-equivariant
$k$-morphism $f^{-1} \colon X_2 \to X_1$.

\smallskip
(b) Since open subsets are complements of closed subsets, it
suffices to consider the case where $X_0$ is closed. We claim that
$\pi(X_0)$ is closed in $Y$.  It is enough to check this claim locally in
the \'etale topology, so we may assume that $X = G/S \times_k Y$ and
$\pi$ is the projection onto the second factor. Since $X_0$ is
$G$-equivariant, $X_0$ contains $\{ 1 \} \times_k \pi(X_0)$. Moreover,
since $X_0$ is closed, $X_0$ contains $\{ 1 \} \times
\overline{\pi(X_0)}$. We conclude that $\overline{\pi(X_0)}$ is
contained in $\pi(X_0)$, i.e., $\pi(X_0)$ is closed, as claimed.

After replacing $Y$ by $\pi(X_0)$ and $X$ by $\pi^{-1}(\pi(X_0))$,
it now suffices to show that if $X_0 \subset X$ is closed and $G$-invariant
and $\pi(X_0) = Y$ then $X_0 = X$. To do this, we construct
the inverse to the inclusion map $X_0 \hookrightarrow X$. We first do
this \'etale-locally, where we may assume $X = G/S \times_k Y$ and hence,
$X_0 = X$, then use \'etale descent to patch together local
inverses into a morphism $X \to X_0$ defined over $Y$.

\smallskip
(c) By part (b), the closure of $f(X_1)$ in $X_2$ is of the
form $\pi_2^{-1}(C)$ for some closed subset $C \subset Y_2$.
Thus $f$ is dominant if and only if $C = Y_2$, that is, if and only if
$\overline{f}$ is dominant.
\end{proof}

Let $N := N_G(S)$ be the normalizer of $S$ in $G$, $W:= N/S$, and $X
\to Y$ a $(G, S)$-fibration. Note that $W$ is again a linear algebraic
group over $k$.  Denote the $S$-fixed point locus in
$X$ by $X^S$. The $G$-action on $X$ induces an $N$-action on $X^S$.
Since $S$ acts trivially on $X^S$, this $N$-action descends to a
$W$-action on $X^S$. By trivializing the $(G, S)$-fibration $X \to
Y$ over an \'etale cover $Y' \to Y$, we see that $X^S \to Y$ is in
fact a $W$-torsor; see~\cite[Proposition 2.9]{CTKPR}. Conversely,
starting with a $W$-torsor $Z \to Y$, we can build a $(G,
S)$-fibration $X \to Y$ by setting $X$ to be the ``homogeneous fiber
space" $G \times^N Z$, i.e., the quotient of $G \times_k Z$ by the
left $N$-action given by $n \cdot (g, x) \to (g n^{-1}, nx)$. This
quotient can either be constructed locally, in the \'etale topology
on $Y$, by descent, or globally as a geometric quotient in the sense
of geometric invariant theory.  For details on these constructions,
we refer the reader to~\cite[\S2.2]{CTKPR}.

\begin{proposition} \label{prop.fibrations2}
Let $Var_k$ be the category of quasi-projective varieties, and
$Fib_{(G, S)}$ the functor from $Var_k$ to the category of sets
which associates to a quasi-projective variety $Y$ the set of
isomorphism classes
of $(G,S)$-fibrations over $Y$, and to a
$k$-morphism of varieties $\tilde{Y} \to Y$ the pull-back morphism
which base-changes $(G, S)$-fibrations over $Y$ to $\tilde{Y}$. If
$S = \{ 1 \}$, we will write $Tor_G$ in place of $Fib_{(G, S)}$.

Then the two constructions described above give rise to an
isomorphism between the functors $Fib_{(G, S)}$ and $Tor_W$.
\end{proposition}

\begin{proof} See~\cite[Proposition 2.10]{CTKPR}.
\end{proof}

\subsection{$(G, S)$-varieties}
  A $k$-variety $X$ with a left action of $G$ is called a $(G, S)$-{\em variety}
  if it contains a dense open subset $X' \subset X$ which is
  the total space of a $(G, S)$-fibration $X' \to Y$.

\begin{lemma} \label{lem.reduction}
Let $G$ be a reductive $k$-group, $T\subset G$ a maximal $k$-torus,
and $M$ a closed algebraic $k$-subgroup
of the $k$-group scheme $\Aut(G)$ such that  $\Inn(G)\subset M$.
Then $G$  and its Lie algebra $\Lie(G)$ are both
$(M,T^\ad)$-varieties.
\end{lemma}

In the case where $M = \Inn(G)$, the lemma was proved in
\cite[Proposition 4.3]{CTKPR}.

\begin{proof} Being a $(G, S)$-variety is a geometric notion.
That is, suppose $k'/k$ is a field extension. Then $X$ is a $(G, S)$-variety
over $k$ if and only if $X_{k'}$ is a $(G_{k'}, S_{k'})$-variety over $k'$.
Thus, after replacing $k$ by a suitable $k'$, we may assume that
$G$ and $T$ are split.

We will only consider the $M$-action on $G$;
the case of the $M$-action on $\Lie(G)$ is similar.
By Corollary~\ref{cor:M}, $M = \Inn(G) \rtimes A$,
where $A$ is a finite group of automorphisms of $G$ and
every element of $A$ preserves $T$.

Our proof will rely on~\cite[Proposition 2.16]{CTKPR}.
To apply this proposition we need to check that the
$M$-action on $G$ is stable, i.e.,
the $M$-orbit of $x \in G(\bar{k})$ is closed
for $x$ in general position.
By \cite[Corollary 4.2]{CTKPR}, the conjugation action of $G$ on
itself is stable.  Since $A$ is a finite group,
the group $M$ contains $G^\ad$ as a subgroup of finite index,
and therefore the $M$-action on $G$ is also stable.

By~\cite[Proposition 2.15(i)]{CTKPR}, we can now conclude that
$G$ is an $(M, S)$-variety for some subgroup
$S \subset M$.  Moreover, by \cite[Proposition 2.16]{CTKPR},
in order to show that we may take $S = T^\ad$, it suffices
to exhibit a dense subset $D \subset G(k)$ defined over $k$
such that the stabilizer of every $p \in D$ in $M$
is conjugate to $T^\ad$.

In fact, it suffices to construct a dense open subset $U \subset T$
defined over $k$ such that the stabilizer of every  $p \in U(k)$ is
conjugate to $T^\ad$; we can then take $D$ to be the union of
$\Inn(G)$-translates of $U(k)$.

Consider the set $T^{\rm reg}$ of regular points of $T$.
By~\cite[\S12.2]{Borel}, $T^{\rm reg}$ is a dense open subset
of $T$ defined over $k$. We claim that for $t \in T^{\rm reg}$ in
general position, $\Stab_M(t) = T^\ad$. Indeed, suppose $g \in M$
stabilizes $t$.  Since $t$ lies in a unique
maximal torus of $G$ (see~\cite[Proposition 12.2(4)]{Borel}),
$g(T) = T$. Equivalently, $g$ lies in
$N_{G^\ad}(T^\ad) \rtimes A \subset M$. The latter group
acts on $T$ via its finite quotient $W \rtimes A$, and
the $W \rtimes A$-action on $T$ is faithful (see the proof of
Lemma~\ref{lem.SW2}).
The fixed points of each element of $W \rtimes A$ form a proper
closed subvariety of $T^{\rm reg}$. Removing these closed subvarieties
from  $T^{\rm reg}$, we obtain a dense open subset $U \subset T$
such that $\Stab_{W \rtimes A}(t) = \{ 1 \}$ or equivalently,
$\Stab_M(t) = T^{\ad}$ for every $t \in U$, as desired.
\end{proof}

\begin{proposition} \label{prop.fibrations3}
Suppose $X_1$ and $X_2$ are $(G, S)$-varieties such that
the fixed point loci $X_1^S$ and $X_2^S$ are irreducible.
Set $N :=N_G(S),\ W :=N/S$.
Then

\begin{enumerate}[\upshape(a)]
\item every $G$-equivariant dominant
rational map $\alpha \colon X_1 \dasharrow X_2$
restricts to a $W$-equivariant dominant rational map
$\beta \colon X_1^S \dasharrow X_2^S$.

\item Every $W$-equivariant dominant rational map
$\beta \colon X_1^S \dasharrow X_2^S$
lifts to a unique $G$-equivariant dominant
rational map $\alpha \colon X_1 \dasharrow X_2$.

\item Moreover, $\beta$ is a birational isomorphism if and only if
so is $\alpha$.
\end{enumerate}
\end{proposition}

\begin{proof} For $i = 1, 2$ let $X_i'$ be a $G$-invariant dense
open subset of $X_i$ which is the total space of a $(G, S)$-fibration,
$X_i' \to Y_i$. Since each $X_i^S$ is irreducible, the non-empty open
subset $(X_i')^S$ is dense in $X_i^S$.  Hence, the dominant rational map
$X_1^S \dasharrow X_2^S$ restricts to a dominant rational map
$(X_1')^S \dasharrow (X_2')^S$, and we may, without loss of
generality, replace $X_i$ by $X_i'$ and thus assume that $X_i$ is
the total space of a $(G, S)$-fibration $X_i \to Y_i$.
Lemma~\ref{lem.fibrations1}(b) now tells us that
after removing a proper closed subset from $Y_1$ (and its preimages
from $X_1$ and $X_1^S$), we may assume that $f$ is regular.
By Proposition~\ref{prop.fibrations2},
$X_i^S \to Y_i$ is a $W$-torsor for $i = 1, 2$.
By Lemma~\ref{lem.fibrations1}(a), $\alpha$ is a morphism
of $(G, S)$-fibrations, and $\beta = \alpha_{| \, X_1^S}\, \colon
X_1^S \to X_2^S$ is a morphism of $W$-torsors.
We thus obtain the following diagram
\[
  \xymatrix{ X_1 \ar[dd] \ar[rrr]^{\alpha} & & & X_2 \ar[dd] \\
& X_1^S \ar@{^{(}->}[ul] \ar[dl] \ar[r]^{\beta} & X_2^S \ar[dr]  \ar@{^{(}->}[ur] & \\
  Y_1 \ar[rrr]^{\overline{\alpha} = \overline{\beta}} & &   & Y_2 \, . }
\]
By Proposition~\ref{prop.fibrations2}, $\alpha$ restricts to $\beta$
and $\beta$ lifts to $\alpha$ in a unique way. Moreover, $\alpha$ and
$\beta$ induce the same morphism $\overline{\alpha}
= \overline{\beta} \colon Y_1 \to Y_2$.

By Lemma~\ref{lem.fibrations1}(c), $\alpha$ is dominant
if and only if $\overline{\alpha} = \overline{\beta}$ is
dominant if and only if $\beta$ is dominant. This proves (a) and (b).

\smallskip
(c) If $\alpha$ is a birational isomorphism, then restricting
$\alpha^{-1}$ to $X_1^S$, we obtain an inverse for $\beta$.
Similarly, if $\beta$ is a birational isomorphism, then extending
$\beta^{-1}$ to $X_2 \dasharrow X_1$, we obtain an inverse for
$\alpha$.
\end{proof}

\begin{corollary} \label{cor.reduction-to-torus}
Let $G$ be a  reductive $k$-group
and $T \subset G$  a maximal $k$-torus.
Then $G$ is Cayley if and only if there exists a $W(G, T)$-equivariant
birational isomorphism $T \stackrel{\simeq}{\dasharrow} \Lie(T)$
defined over $k$.
\end{corollary}

Note that here, as before, we view the Weyl group $W(G, T)$ as
an algebraic group over $k$.

\begin{proof} By Lemma~\ref{lem.reduction}, with $M = \Inn(G)$,
$X_1 = G$ and $X_2 = \Lie(G)$ are both $(\Inn(G), T^\ad)$-varieties.
The fixed point loci, $X_1^{T^{\ad}} = T$ and $X_2^{T^{\ad}} = \Lie(T)$,
are irreducible. The desired conclusion is now a direct consequence of
Proposition~\ref{prop.fibrations3}: there exists a $G$-equivariant
birational isomorphism
\[ \alpha \colon G = X_1 \stackrel{\simeq}{\dasharrow}
X_2 = \Lie(G) \]
(i.e., a Cayley map for $G$)
if and only if there exists a $W(T)$-equivariant birational isomorphism
$\beta \colon T = X_1^{T^{\ad}} \stackrel{\simeq}{\dasharrow} X_2^{T^{\ad}} = \Lie(T)$.
\end{proof}

\section{Proof of Theorem~\ref{thm.main1}}
\label{sect.proof-of-thm.main1}

(a) $\Longrightarrow$ (b). First suppose $G$ is Cayley over $k$.
Then $G_K$ is Cayley over $K$ for every field extension
$K/k$.  Then by Corollary~\ref{cor.reduction-to-torus},
every maximal $K$-torus $T$ of $G_K$ is $K$-rational.

Now suppose $G$ is stably Cayley over $k$, i.e., $G \times \mathbb G_m^r$
is Cayley
for some $r \ge 0$. Then the above argument shows that for every $K$-torus
$T$ of $G$, $T \times \bbG_m^r$ is $K$-rational. Hence, $T$ is stably
$K$-rational, as claimed.

(b) $\Longrightarrow $ (c) is obvious.

(c) $\Longleftrightarrow $ (d). By Proposition~\ref{prop:Voskr},
the  character lattice $\X(G)$ of $G$ is isomorphic to
the character lattice of the generic torus $T_\gen$ of $G$.
Since a torus $T$ is stably rational if and only if its character
lattice $\XX(T)$ is quasi-permutation
(see~\cite[Theorem 4.7.2]{Voskresenskii-book}), (c) and (d) are equivalent.

(d) $\Longrightarrow $ (a). Let $\Gspl$ be a split $k$-form of $G$.
Let $(\Tspl, \Bspl)$ be a Borel pair in $\Gspl$ defined over $k$,
$T$ a maximal $k$-torus of $G$, and
$\Bbar$  a Borel subgroup defined over the algebraic closure $\kbar$ and
containing $\Tbar$.
We choose and fix an isomorphism $\theta\colon \ov{\Gspl}\to\Gbar$ taking
$(\ov{\Tspl},\ov{\Bspl})$ to $(\Tbar,\Bbar)$, and we construct
the subgroup $M_G\subset\Aut(\Gspl)$ using $\theta$,
as in Subsection \ref{subsec:cocycle}.
By Proposition~\ref{prop:cocycle},
$G$ is isomorphic to $_z G_\spl$ for some cocycle $z \in Z^1(k,M_G)$.
By Lemma~\ref{lem.outer-form}(b), in order to show that $G$ is stably Cayley,
it suffices to construct an $M_G$-equivariant birational
isomorphism
\begin{equation} \label{e1.d=>a}
\Gspl \times \bbG_m^r \stackrel{\simeq}{\dasharrow} \Lie(\Gspl) \times \mathbb{A}^r
\end{equation}
for some $r \ge 0$, where $M_G$ acts trivially on the split torus
$\bbG_m^r$ and the affine space $\mathbb{A}^r$.  By
Lemma~\ref{lem.reduction}, $X_1:= \Gspl \times \bbG_m^r $ and $X_2
:= \Lie(\Gspl) \times \mathbb{A}^r$ are both $(M_G, S)$-varieties,
where  $S: =(\Tspl)^\ad$. By Proposition~\ref{prop.fibrations3}, in
order to construct an $M_G$-equivariant birational isomorphism
\eqref{e1.d=>a}, it suffices to construct an
$N_{M_G}(S)/S$-equivariant birational isomorphism $X_1^S \dasharrow
X_2^S$, where $X_1^S = \Tspl \times \G_m^r$, $X_2^S = \Lie(\Tspl)
\times \mathbb{A}^r$. Note that $N_{M_G}(S)/S$ is isomorphic to the
group $\tV\subset\Aut\,\XX(T_\spl)$ (see Subsection
\ref{subsec:cocycle}).

It thus remains to
show that there exists a $\tV$-equivariant birational isomorphism
\begin{equation} \label{e.cayley-split}
T_{\spl} \times \bbG_m^r \stackrel{\simeq}{\dasharrow} \Lie(\Tspl)
\times \mathbb{A}^r
\end{equation}
for some $r \ge 0$. By the definition of $\tV$, the
lattice $(\tV, \XX(T_\spl))$ is isomorphic to the character lattice
$(\V(G, T), \XX(\Tbar))$ of $G$. By condition (d) of the theorem, the
character lattice of $G$ is quasi-permutation, hence so is the
lattice $(\tV, \XX(T_\spl))$. By
Lemma~\ref{lem.stably-linearizable}(b), this implies that the
$\tV$-action on the split torus $\Tspl$ is stably linearizable. In
other words, $T_{\spl}$ is $\tV$-equivariantly stably birationally
isomorphic to a faithful linear representation $V$ of the finite
group $\tV$.  On the other hand, by Remark~\ref{rem.no-name}, the
vector space $V$ is $\tV$-equivariantly stably birationally
isomorphic to the faithful $\tV$-representation $\Lie(\Tspl)$. Composing these two $\tV$-equivariant
birational isomorphisms, we see that $\Tspl$ and $\Lie(\Tspl)$ are
$\tV$-equivariantly stably birationally isomorphic. In other
words,~\eqref{e.cayley-split} holds for a suitable $r \ge 0$, as
claimed. This completes the proof of Theorem~\ref{thm.main1}. \qed

\begin{corollary} \label{cor.rank2} Let $k$ be a field of characteristic $0$.
Then every reductive $k$-group $G$ of rank $\le 2$ is stably Cayley.
\end{corollary}

\begin{proof} By Lemma~\ref{lem:Voskresenskii}
the character lattice of $G$ is quasi-permutation.
Thus $G$ is stably Cayley by Theorem~\ref{thm.main1}.

Alternatively, the generic torus of $G$ is of dimension $\le 2$ and
hence, is rational, hence stably rational; see \cite[\S4.9, Examples 6, 7]{Voskresenskii-book}.
Once again, we conclude that $G$ is stably Cayley by Theorem~\ref{thm.main1}.
\end{proof}

The following Corollary amplifies Lemma~\ref{lem.outer-form}.

\begin{corollary}\label{cor:inner}
Let $\Gspl$ be a split reductive group over a field $k$ of characteristic $0$,
$\Ginn$ an inner $k$-form of $\Gspl$, and $G$ an arbitrary $k$-form of $\Gspl$.
\begin{enumerate}[\upshape(a)]
\item If $G$ is stably Cayley, then so is $\Gspl$.
\item $\Ginn$ is stably Cayley if and only if $\Gspl$ is stably Cayley.
\end{enumerate}
\end{corollary}

\begin{proof}
Set $\Gbar:=G \times_k \kbar$.
If $G$ is stably Cayley over $k$, then clearly  $\Gbar$
is stably Cayley over $\kbar$,
and by Theorem \ref{thm.main1} (or by \cite[Theorem 1.27]{LPR06})
the character lattice $\X(\Gbar)$ is quasi-permutation.
Since $\X(\Gbar) \simeq \X(\Gspl)$, assertion (a) follows from  Theorem \ref{thm.main1}.
Similarly, since $\X(\Gspl) \simeq \X(\Ginn)$, assertion (b) follows from  Theorem \ref{thm.main1}.
\end{proof}

\section{Proof of Theorem~\ref{cor.main2}}
\label{sect.proof-of-cor.main2}

To show that (a) $\Longrightarrow$ (b), suppose that
$G$ is stably Cayley over $k$. Then $G_{\bar{k}}$ is stably Cayley
over $\bar{k}$, where $\kbar$ denotes an algebraic closure of $k$.
By~\cite[Theorem 1.28]{LPR06}, $G_\kbar$ is one of
the following groups:
\begin{equation} \label{e.lpr}
  \text{ $\gSL_3$, $\gSO_{n}$ ($n\ge 5$),
$\gSp_{2n}$ ($n \ge 1$), $\gPGL_n$ ($n \ge 2$), $\gG_2$.}
\end{equation}
In other words, $G$ is a $k$-form of one of these groups. (Note that
the groups $\gSL_2$ and $\gSO_3$, which appear in the
statement of~\cite[Theorem 1.28]{LPR06},
are isomorphic to $\gSp_2$ and $\gPGL_2$, respectively.)
If $G$ is an outer form of $\gPGL_n$ where $n \ge 4$ is even,
then by \cite[Theorem 0.1]{CK} the generic torus of $G$ is not stably rational,
and by Theorem \ref{thm.main1}, $G$ is not stably Cayley.
Thus if $G$ is stably Cayley,
then $G$ is one of the groups listed in part (b).

It remains to prove that (b) $\Longrightarrow$ (a), i.e., that all groups
listed in part (b) are stably Cayley.

The classical Cayley transform shows that all forms of $\gSO_n$ and
$\gSp_{2n}$ are Cayley; see~\cite[Example 1.16]{LPR06}. All forms of
the groups $\gSL_3$ and $\gG_2$ are of rank $2$, hence their generic
tori are rational by \cite[Example 4.9.7]{Voskresenskii-book}, and
by Theorem~\ref{thm.main1}, these groups are stably Cayley. Every
inner form of $\gPGL_n$ is Cayley by \cite[Example 1.11]{LPR06};
cf.~also Lemma~\ref{lem.outer-form}(c). Finally, the generic torus
of any form of $\gPGL_n$ for $n$ odd is rational, hence stably
rational by~\cite[Corollary of Theorem 8]{VK}. By Theorem
\ref{thm.main1}, we conclude that outer forms of $\gPGL_n$  for $n$
odd are stably Cayley. This completes the proof of
Theorem~\ref{cor.main2}. \qed

\section{Statement of Theorem~\ref{thm:product-closed-field}
and first reductions}
\label{sect7}

In view of Theorem~\ref{cor.main2} it is natural to ask for a
classification of stably Cayley semisimple  groups, initially over an
algebraically closed field of characteristic zero. This problem
turns out to be significantly more complicated; a complete solution
is out of reach at the moment; cf. Remark~\ref{rem.conjecture}.
Fortunately, for the purpose of
proving Theorem~\ref{thm.main3}, we can limit our attention to
semisimple groups all of whose simple components are of the same
type. Theorem \ref{thm:product-closed-field} stated below gives a
classification of stably Cayley groups of this form; this theorem
will be a key ingredient in our proof of Theorem~\ref{thm.main3} in
Section~\ref{sect.proof-of-thm.main3}. The proof of Theorem
\ref{thm:product-closed-field} will occupy most of the remainder of
this paper.

\begin{theorem}\label{thm:product-closed-field}
Let $k$ be an algebraically closed field of characteristic $0$ and
$G$ a semisimple $k$-group of the form $H^m/C$, where $H$ is a
simple and simply connected $k$-group and $C$ is a central
$k$-subgroup of $H^m$. (In other words, the universal cover of $G$
is of the form $H^m$.) Then $G$ is stably Cayley if and only if $G$
is isomorphic to a direct product $G_1\times_k\dots\times_k G_s$,
where each $G_i$ is either a stably Cayley simple $k$-group
(i.e., is one of the groups listed in  \eqref{e.lpr}) or
$\gSO_4$.
\end{theorem}

Note that $\gSO_4$ is semisimple but not simple.
The ``if'' direction of Theorem~\ref{thm:product-closed-field}
is obvious, since the direct product of stably Cayley groups
is stably Cayley.
(As we mentioned in the previous section, $\gSO_4$ is Cayley via the
classical Cayley transform.)
Thus we only need to prove the ``only if'' direction.
The proof will proceed by case-by-case analysis, depending
on the type of $H$.  We begin with the following easy reduction.

\begin{lemma} \label{lem.types}
Let $H$ be a simply connected simple group over an algebraically
closed field $k$, and $C$ a central subgroup of $H^m$ for some $m
\ge 1$. Let $H_i$ denote the $i^{\textit{th}}$ factor of $H^m$,
$\pi_i$ denote the natural projection $H^m \to H_i$, and $C_i : =
\pi_i(C) \subset Z(H_i)$, where $Z(H_i)$ denotes the center of
$H_i$. Assume $H^m/C$ is stably Cayley. Then
\begin{enumerate}[\upshape(a)]
\item[\rm (a)] $H_i/C_i$ is stably Cayley;
\item[\rm (b)] $H$ is of type $\AA_n$ $(n \ge 1)$, $\BB_n$ $(n\ge 2)$,
$\CC_n$ $(n \ge 3)$ , $\DD_n$ $(n \ge 4)$, or $\GG_2$.
\end{enumerate}
\end{lemma}

\begin{proof}
Part (a) is a direct consequence of~\cite[Proposition 4.8]{LPR06}. To
prove part (b), note that by~\cite[Theorem 1.28]{LPR06}, $H_1/C_1$  is
of one of the types listed in the statement of the lemma.
\end{proof}

We will now settle two easy cases of
Theorem~\ref{thm:product-closed-field}, where $H$ is of type $\CC_n$
($n \ge 3$) and $\GG_2$.

\begin{proof}[Proof of Theorem~\ref{thm:product-closed-field}
for $H= \GG_2$.]
Here $Z(H) = \{ 1 \}$, so $C \subset Z(H)^m$ is trivial, and
\[ \text{$H^m/C = H^m = \GG_{2} \times_k \dots \times_k
\GG_{2}$ ($m$ times)} \]
is a product of stably Cayley simple groups.
\end{proof}

\begin{proof}[Proof of Theorem~\ref{thm:product-closed-field} for $H$ of type
$\CC_n$ $(n \ge 3)$.] Let $H = \Sympl_{2n}$ and $C$ be a subgroup of
$Z(H)^m = \mu_2^m$.  We will show that if $H^m/C$ is stably Cayley,
then $C = \{ 1 \}$.

Indeed, if $H^m/C$ is stably Cayley, then, by Lemma~\ref{lem.types},
so is $H_i/C_i$. Here $H_i = \Sympl_{2n}$, and $C_i$ is a central
subgroup (either $\mu_2$ or $\{ 1 \}$). On the other hand, by
\cite[Theorem 1.28]{LPR06}, if the group $\Sympl_{2n}/C_i$
is stably Cayley for some $n\ge 3$ then
$C_i = \{ 1 \}$. Thus $C$ projects trivially to
every $H_i$, which is only possible if $C = \{ 1 \}$.
We conclude that
\[ \text{$H^m/C = H^m = \Sympl_{2n} \times_k \dots \times_k
\Sympl_{2n}$ ($m$ times)} \]
is a product of Cayley simple groups, as desired.
\end{proof}

\begin{remark} \label{rem.conjecture}
We conjecture that Theorem~\ref{thm:product-closed-field} remains
true for every semi\-sim\-ple $k$-group $G$ over an algebraically
closed field $k$ of characteristic $0$, without any additional
assumption on the universal cover of $G$. That is, a semisimple
$k$-group is stably Cayley if and only if it is isomorphic to a
direct product $G_1 \times_k \dots \times_k G_s$, where each $G_i$
is either a stably Cayley simple group or $\gSO_4$.
\end{remark}

\section{Quasi-invertible lattices}
\label{sect8}

The proof of the ``only if" direction of Theorem
\ref{thm:product-closed-field} in the remaining cases, where $H$ is
of type $\AA_n$, $\BB_n$ or $\DD_n$, is more involved.  In this
section, in preparation for this proof, we will describe a general
method for showing that certain lattices are not quasi-permutation
(and more generally, cannot even be direct summands of
quasi-permutation lattices).

\begin{definition}
A $\Gamma$-lattice $L$ is called {\em quasi-invertible} if it is
a direct summand of a quasi-permutation $\Gamma$-lattice.
\end{definition}

\begin{lemma}[J.-L. Colliot-Th\'el\`ene] \label{lem:q-inv}
A $\Gamma$-lattice $L$ is  {\em quasi-invertible} if and only if it fits into a short exact sequence
\begin{equation}\label{eq:q-inv}
0\to L\to P\to I\to 0,
\end{equation}
where $P$ is a permutation $\Gamma$-lattice and $I$ is an invertible $\Gamma$-lattice,
i.e. a direct summand of a permutation $\Gamma$-lattice.
\end{lemma}

\begin{proof}
For a $\Gamma$-lattice $L$ we have a flasque resolution
$$
0\to L\to P\to F\to 0,
$$
where $P$ is a permutation $\Gamma$-lattice and $F$ is a flasque $\Gamma$-lattice,
see \cite[\S\,1]{CTS1} or \cite[Ch.~2]{Lorenz} for the theory of flasque resolutions.
We write $[L]^\fl$ for the class of $F$ up to addition of a permutation lattice
(note that $F$ is defined up to addition of a permutation lattice).
We have $[L\oplus L']^\fl=[L]^\fl\oplus [L']^\fl$.
If  $L$ is quasi-invertible, then  $L\oplus L'$ is quasi-permutation for some $L'$,
hence $[L]^\fl\oplus [L']^\fl=[L\oplus L']^\fl=0$.
We see that $[L]^\fl$ is invertible, hence $L$ fits into an exact sequence \eqref{eq:q-inv} with $I$ invertible.

Conversely, if $L$ fits into an exact sequence \eqref{eq:q-inv} with $I$ invertible, say $I\oplus J=P'$ is permutation,
then adding $I$ to \eqref{eq:q-inv} twice on the left, and then adding $J$ twice on the right, we obtain an exact sequence
$$
0\to L\oplus I\to P\oplus I\oplus J\to I\oplus J\to 0,
$$
which shows that $L$ is quasi-invertible.
\end{proof}

\begin{lemma}[J.-L. Colliot-Th\'el\`ene] \label{lem:q-inv-Gamma1}
Let $\Gamma_1\twoheadrightarrow \Gamma$ be a surjective homomorphism of finite groups, and let $L$ be a $\Gamma$-lattice.
Then $L$ is quasi-invertible as a $\Gamma_1$-lattice if and only if it is quasi-invertible as a $\Gamma$-lattice.
\end{lemma}

\begin{proof}
We argue as in the proof of Lemma \ref{lem:quasi-perm}.
It suffices to prove ``only if''.
Assume that $L$ is quasi-invertible as a $\Gamma_1$-lattice, then by Lemma \ref{lem:q-inv}, $L$ fits
into a short exact sequence \eqref{eq:q-inv} of $\Gamma_1$-lattices,
where $P$ is a permutation $\Gamma_1$-lattice and $I$ is an invertible $\Gamma_1$-lattice.
Set $\Gamma_0=\ker[\Gamma_1\to\Gamma]$.
From \eqref{eq:q-inv}
we obtain the  $\Gamma_0$-cohomology exact sequence
$$
0\to L\to P^{\Gamma_0}\to I^{\Gamma_0}\to 0
$$
(because $L^{\Gamma_0}=L$ and  $H^1(\Gamma_0,L)=0$), which is a short exact sequence of $\Gamma$-lattices.
It is easy to see that $P^{\Gamma_0}$ is a permutation $\Gamma$-lattice and  $I^{\Gamma_0}$ is an invertible $\Gamma$-lattice,
hence by Lemma \ref{lem:q-inv}, $L$ is a quasi-invertible $\Gamma$-lattice.
\end{proof}

Suppose that $(\Gamma, L)$ and $(\Gamma', L')$ are
$\varphi$-isomorphic for some isomorphism $\varphi \colon \Gamma \to
\Gamma'$; for a definition of $\varphi$-isomorphism, see the
beginning of Section~\ref{sect.lattices}.  Then clearly $L$ is
permutation (respectively, quasi-permutation, respectively,
quasi-invertible) if and only if so is $L'$.

\medskip
The Tate--Shafarevich group of  a $\Gamma$-lattice $L$ is defined as
\[ \Sh(\Gamma, L)=\ker\left[H^2(\Gamma,L)\to\prod_{\Gammac\subset
\Gamma} H^2(\Gammac,L)\right] , \]
where $\Gammac$ runs over the set of all {\em cyclic} subgroups of $\Gamma$.
If $L$ is a quasi-invertible $\Gamma$-lattice, then for any subgroup
$\Gamma'\subset\Gamma$ we have $\Sh(\Gamma',L)=0$, cf.
\cite[Proposition 2.9.2(a)]{Lorenz}.  Note however, that
there exist $\Gamma$-lattices
$L$ such that $\Sh(\Gamma',L)=0$ for every subgroup $\Gamma'$ of $\Gamma$
but $L$ is not quasi-invertible;
see  the end of the proof of Proposition \ref{prop:SL3}.
\smallskip

The following lemmas can be used to show that a given lattice is
not quasi-invertible.
Our approach is originally due to
Voskresenski\u\i. Proposition~\ref{thm:J-Gamma} is
essentially~\cite[Theorem~7 and its corollary]{Voskresenskii70}; see
also~\cite[Proposition~1(ii), p.~183]{CTS1} and \cite[Proposition~9.5(ii)]{CTS2}.
For the sake of completeness we supply short proofs
for Lemmas \ref{lem:Sha-two} and \ref{lem:H-three} below.

Let $\Gamma$ be a finite group. Consider the norm homomorphism
$$
N_\Gamma\colon \Z\to \Z[\Gamma], \quad N_\Gamma(a)=
a\sum_{s\in\Gamma} s \text{ for } a\in\Z,
$$
and the short exact sequence
\begin{equation}\label{eq:J}
0\to \Z\to \Z[\Gamma]\to J_\Gamma\to 0,
\end{equation}
where $J_\Gamma=\coker\, N_\Gamma$.

\begin{lemma} \label{lem:Sha-two}
Let $\Gamma$ be a finite group, and $\Gamma'\subset\Gamma$ any subgroup.
Then $\Sh(\Gamma',J_\Gamma)\cong H^3(\Gamma',\Z)$.
\end{lemma}

\begin{proof}
  From \eqref{eq:J} we obtain a cohomology exact sequence
\begin{equation}\label{eq:J-cohomology}
H^2(\Gamma',\Z[\Gamma])\to H^2(\Gamma', J_\Gamma)\to H^3(\Gamma',
\Z)\to H^3(\Gamma', \Z[\Gamma]).
\end{equation}
We have $H^i(\Gamma',\Z[\Gamma'])=0$ for $i\ge 1$, hence $H^i(\Gamma',\Z[\Gamma])=0$ for $i\ge 1$,
and we see from \eqref{eq:J-cohomology} that $H^2(\Gamma', J_\Gamma)\cong H^3(\Gamma', \Z)$.

Now let $\Gammac\subset\Gamma'$ be a {\em cyclic} subgroup. We have
$H^2(\Gammac, J_\Gamma)\cong H^3(\Gammac, \Z)$. By periodicity for
cyclic groups, cf. \cite[IV.8, Theorem 5]{CF}, we have
$$
H^3(\Gammac,\Z)\cong H^1(\Gammac,\Z)=\Hom(\Gammac,\Z)=0.
$$
Thus $H^2(\Gammac, J_\Gamma)=0$ and 
hence, $\Sh(\Gamma',J_\Gamma)=H^2(\Gamma',J_\Gamma)\cong
H^3(\Gamma',\Z)$.
\end{proof}

\begin{lemma} \label{lem:H-three}
Let $\Gamma=\Z/p\Z\times \Z/p\Z$, where $p$ is a prime.
Then $H^3(\Gamma,\Z) \cong\Z/p\Z$.
\end{lemma}

\begin{proof}
For any group $\Gamma$, the group $H^3(\Gamma, \Z)$ is
canonically isomorphic to $H^2(\Gamma, \C^\times)$.
The latter group is called the {\em Schur multiplier } of $\Gamma$.
For finite abelian groups, the Schur multipliers were computed by Schur
in \cite[\S4, VIII]{Schur}.
In particular, by \cite[\S4, VIII]{Schur}, the Schur multiplier of $\Z/p\Z\times \Z/p\Z$
is a cyclic group of order $p$, which proves the lemma.

An alternative proof based on modern references proceeds as follows.
For any finite group $\Gamma$, the group $H^3(\Gamma,\Z)$ is dual
to $H^{-3}(\Gamma,\Z)$,
cf. \cite[Theorem XII.6.6]{CE} or \cite[Theorem VI.7.4]{Brown}.
By definition $H^{-3}(\Gamma,\Z)=H_2(\Gamma,\Z)$.
For an {\em abelian} group $\Gamma$ we have $H_2(\Gamma,\Z)=\Lambda^2(\Gamma)$
(the second exterior power of the $\Z$-module $\Gamma$), see \cite[Theorem 3]{Miller}
or \cite[Theorem V.6.4(c)]{Brown}.
Clearly $\Lambda^2(\Z/p\Z\times \Z/p\Z)\cong\Z/p\Z$, hence $H_2(\Z/p\Z\times \Z/p\Z,\Z)\cong\Z/p\Z$
and $H^3(\Z/p\Z\times \Z/p\Z,\Z)\cong\Z/p\Z$.
\end{proof}

As an immediate consequence, we obtain the following

\begin{proposition}\label{thm:J-Gamma}
Let $\Gamma=\Z/p\Z\times \Z/p\Z$, where $p$ is a prime.
Then $\Sh(\Gamma,J_\Gamma)\cong \Z/p\Z$,
and therefore the $\Gamma$-lattice $J_\Gamma$ is not quasi-invertible.
\qed
\end{proposition}

The following example and subsequent proposition will be used in the proof
of Lemma \ref{lem:direct-product} in the next section, and thus in the proof of
Theorem~\ref{thm.main3} in Section~\ref{sect.proof-of-thm.main3}.

\begin{example}\label{ex.outerPGLn}
Let $H$ be an outer $k$-form of $\gPGL_n$ for some even integer $n\ge 4$.
Recall (see Section~\ref{sect.proof-of-cor.main2})
that by~\cite[Theorem 0.1]{CK} the character lattice of $H$  is not
quasi-permutation.
In fact, it is shown in \cite[\S5.1]{CK} that the character lattice of $H$
is not quasi-invertible. Indeed, let $T$ be a maximal $k$-torus of $H$.
Note that $\V(H, T)=\Sym_n \times \Z/ 2 \Z$
and $X(\Tbar)=\Z \AA_{2n-1}$ on which $\V(H,T)$ acts by permutations
and sign changes.
It is shown in \cite[\S5.1]{CK} that there exists a subgroup $\Gamma$
of $\Sym_n \times \Z/2\Z$
isomorphic to $\Z/2\Z\times\Z/2\Z$, and a direct
summand $M$ of the
$\Gamma$-lattice $\XX(\Tbar)$ isomorphic to $J_{\Gamma}$.
Then $\Sh(\Gamma,M)\ne 0$ and so $\Sh(\Gamma,X(\Tbar))\ne 0$.
This implies that the
$\V(H,T)$-lattice $X(\Tbar)$ is not quasi-invertible.
\end{example}

\begin{proposition} \label{cor.directproduct}
Let $k$ be a field of characteristic zero and
$H$ a reductive  $k$-group with maximal $k$-torus $T$
such that the character lattice $\X(H)=(\V(H,T),\XX(\Tbar))$
is not quasi-invertible.
Then $G := H \times H'$ is not stably Cayley for any
reductive  $k$-group $H'$.
\end{proposition}

\begin{proof} Let $T$, $T'$ and $S = T \times T'$ be maximal
$k$-tori in $H$, $H'$ and $G$,
respectively. By Theorem~\ref{thm.main1}, it suffices to show that
the character lattice $\X(G) = (\V(G, S), \XX(\overline{S}))$ is not
quasi-invertible (and hence, not quasi-permutation).
By Lemma~\ref{lem.SW1}, the extended Weyl
group $\V(G, S)$ is generated by the Weyl group $W(G, S)
= W(H, T) \times W(H', T')$ and the image of the natural action
$\lambda_S \colon \Gal(\kbar /k) \to \Aut(\XX(\overline{S}))$.
Since both $T$ and
$T'$ are defined over $k$, the image of $\lambda_S$
preserves the direct sum decomposition
\[ \XX(\overline{S}) = \XX(\Tbar) \oplus \XX(\overline{T}') \]
and hence, so does $\V(G, S)$.
Moreover, $\V(G, S)$ acts on $\XX(\Tbar)$ via a surjection
$\pi \colon \V(G, S) \to \V(H, T)$.
By Lemma~\ref{lem:q-inv-Gamma1}, $\XX(\Tbar)$ is not quasi-invertible as a $\V(G,S)$ lattice
and, since
$\XX(\Tbar)$ is a direct summand of $\XX(\overline{S})$,
we conclude that $\XX(\overline{S})$ is not quasi-invertible
as a $\V(G, S)$-lattice.
In other words, the character lattice $\X(G) = (\V(G, S), \XX(\overline{S}))$
is not quasi-invertible,
and therefore $G$ is not stably Cayley, as desired.
\end{proof}

\section{Simply connected and adjoint semisimple groups}
\label{sect.sc-ad}

We are now in a position to classify stably Cayley simply connected
and adjoint $k$-groups.

\begin{lemma}\label{lem:direct-product}
Let $G$ be a stably Cayley semisimple $k$-group.
Assume that $G\times_k \kbar$ is a direct product
of simple $\kbar$-groups.
Then $G$ is $k$-isomorphic to a direct product
$R_{l_1/k} G_1 \times \dots \times R_{l_r/k} G_r$,
where each $G_j$ is
a stably Cayley absolutely simple group defined
over a finite field extension $l_j/k$.
\end{lemma}

\begin{proof}
By assumption $G_\kbar=\prod_{i\in I} G_{i,\kbar}$ for some index set $I$,
where each $G_{i,\kbar}$ is a simple $\kbar$-group.
The Galois group $\Gal(\kbar/k)$ acts on $G_\kbar$, hence on $I$.
Let $\Omega$ denote the set of orbits of $\Gal(\kbar/k)$ in $I$.
For $\omega\in\Omega$ set $G^\omega_\kbar=\prod_{i\in\omega} G_{i,\kbar}$,
then $G_\kbar=\prod_{\omega\in \Omega} G^\omega_\kbar$.
Each $G^\omega_\kbar$ is $\Gal(\kbar/k)$-invariant,
hence it defines a $k$-form $G^\omega_k$ of $G^\omega_\kbar$.
We have $G=\prod_{\omega\in \Omega} G^\omega_k$.

Fix $\omega\in\Omega$ and choose $i\in\omega$.
Let $l_i/k$ denote the Galois extension in $\kbar$
corresponding to the stabilizer of $i$ in $\Gal(\kbar/k)$.
The group $G_{i,\kbar}$ is $\Gal(\kbar/l_i)$-invariant,
hence it comes from  an $l_i$-form $G_{i,l_i}$.
By the definition of Weil's restriction of scalars, see \cite[\S\,3.12]{Voskresenskii-book},
$G^\omega_k\cong R_{l_i/k} G_{i,l_i}$.

We next show that $G_{i,l_i}$ is a direct factor of $G_{l_i}:=G\times_k l_i$.
It is clear from the definition that $G_{i,\kbar}$ is a direct factor
of $G_\kbar$ with complement
$G'_\kbar=\prod_{j\in I\smallsetminus \{i\} } G_{j,\kbar}$.
Then $G'_\kbar$ is $\Gal(\kbar/l_i)$-invariant, hence
it comes from some $l_i$-group $G'_{l_i}$.
We have $G_{l_i}=G_{i,l_i}\times_{l_i} G'_{l_i}$,
hence $G_{i,l_i}$ is a direct factor of $G_{l_i}$.

It remains to show that $G_{i,l_i}$ is stably Cayley over $l_i$.
Since $G$ is stably Cayley over $k$, the group $G_\kbar$
is  stably Cayley over $\kbar$.
Since $G_{i,\kbar}$ is a direct factor of
the stably Cayley $\kbar$-group $G_\kbar$
over the algebraically closed field $\kbar$,
by \cite[Lemma 4.7]{LPR06} $G_{i,\kbar}$ is stably Cayley over $\kbar$.
Comparing \cite[Theorem 1.28]{LPR06} and Theorem \ref{cor.main2},
we see that $G_{i,l_i}$ is either stably Cayley over $l_i$ (in which case
we are done) or an outer form of $\gPGL_{n}$ for some even $n\ge 4$.
Thus assume, by way of contradiction, that $G_{i,l_i}$
is an outer form of $\gPGL_{n}$ for some even $n \ge 4$.
Then by Example~\ref{ex.outerPGLn}
the character lattice of $G_{i,l_i}$ is not quasi-invertible,
and by Proposition~\ref{cor.directproduct} the group  $G_{i,l_i}$
cannot be a direct factor of a stably Cayley $l_i$-group.
This contradicts  the fact that $G_{i,l_i}$ is a direct factor
of the stably Cayley $l_i$-group $G_{l_i}$.
We conclude that $G_{i,l_i}$ cannot be
an outer form of $\gPGL_{n}$ for any even $n \ge 4$.
Thus $G_{i,l_i}$ is stably Cayley over $l_i$, as desired.
\end{proof}

\begin{corollary}\label{cor:sc}
Let $G$ be a  {\em simply connected} semisimple $k$-group over
a field $k$ of characteristic 0.
Then the following conditions are equivalent:
\begin{enumerate}[\upshape(a)]
\item $G$ is stably Cayley over $k$;
\item $G$ is $k$-isomorphic to
$R_{l_1/k} G_1 \times \dots \times R_{l_r/k} G_r$
for some finite field extensions $l_1/k, \dots, l_r/k$,
where
each $G_i$ is an arbitrary $l_i$-form of
$\gSL_3$, or $\gSp_{2n}$ $(n\ge 1)$, or $\gG_2$.
\end{enumerate}
\end{corollary}

\begin{corollary}\label{cor:ad}
Let $G$ be an {\em adjoint} semisimple $k$-group
over a field $k$ of characteristic 0.
Then the following conditions are equivalent:
\begin{enumerate}[\upshape(a)]
\item $G$ is stably Cayley over $k$;
\item $G$ is $k$-isomorphic to
$R_{l_1/k} G_1 \times \dots \times R_{l_r/k} G_r$
for some finite field extensions $l_1/k, \dots, l_r/k$,
where
each $G_i$ is an arbitrary $l_i$-form of
$\gPGL_{2n+1}$ $(n\ge 1)$, or $\gSO_{2n+1}$ $(n\ge 1)$, or $\gG_2$,
or an {\em inner} form of $\gPGL_{2n}$ $(n\ge 2)$.
\end{enumerate}
\end{corollary}

\begin{proof}[Proof of Corollaries \ref{cor:sc} and \ref{cor:ad}]
Clearly, the implication (b)$\Longrightarrow$(a) holds in both
cases. To prove that (a)$\Longrightarrow$(b), note that since $G$ is
either simply connected or adjoint, $G_\kbar$ is either simply
connected or adjoint. Hence, $G_\kbar$ is a direct product of simple
normal subgroups $G_{i,\kbar}$, and Lemma \ref{lem:direct-product}
applies to $G$. It tells us that $G$ is a product of its $k$-simple
normal subgroups of the form $R_{l_j/k} G_j$, where each $G_j$ is
stably Cayley and absolutely simple over some finite field extension
$l_j/k$. In other words, $G_j$ is one of the groups listed in
Theorem \ref{cor.main2}.  Since $G_j$ is either simply connected or
adjoint, Corollaries \ref{cor:sc} and~\ref{cor:ad} follow.
\end{proof}

\section{A family of non-quasi-invertible lattices}
\label{sect.family}

We will now use the results of Section~\ref{sect8} to exhibit
a large family of non-quasi-invertible lattices
(i.e., lattices that are not direct summands of quasi-permutation
lattices). These lattices will be used to complete the proof of
Theorem~\ref{thm:product-closed-field}.
\medskip

Let $\Delta$ be a Dynkin diagram, $\Delta=\bigcup_{i=1}^m \Delta_i$,
where $\Delta_i$ are the connected components of $\Delta$. We assume
that each $\Delta_i$ is of type $\BB_{l_i}$ ($l_i\ge 1$) or of type
$\DD_{l_i}$ ($l_i\ge 3$). Note that $\BB_1=\AA_1$ and $\DD_3=\AA_3$
are allowed. The root system $R(\Delta_i)$ can be realized in a
standard way in the space $V_i:=\Q^{l_i}$ with standard basis
$(\ve_s)_{s\in S_i}$, where $S_i$ is an index set consisting of
$l_i$ elements, see \cite[Planches II,~IV]{Bourbaki}.

Let $S=\dot{\bigcup} S_i$ (disjoint union). Consider the vector
space $V:=\bigoplus_i V_i$ over $\Q$ with standard basis
$(\ve_s)_{s\in S}$. Set
\begin{equation} \label{e.beta}
\beta =\half\sum_{s\in S}\ve_s \, .
\end{equation}
We denote by $M$ the additive subgroup in $V$ generated
by $\beta$ and by the basis
elements $\ve_s$ for all $s\in S$. In other words, $M$ is generated
by the vectors of the form $\half\sum_{s\in S}\pm\ve_s$.

Denote the Weyl group $W(\Delta_i)$ by $W_i$ and the
Weyl group $W(\Delta)=\prod_{i=1}^m W_i$ by $W$.
Consider the natural action of $W$ on $M$.  For $s\in S_i$ let $c_s$
denote the automorphism of $V_i$ acting as $-1$ on $\ve_s$ and as 1
on all the other $\ve_t$ ($t\in S_i,\ t\neq s$). The Weyl group
$W_i=W(\Delta_i)$ is the semidirect product of the symmetric group
$\Sym_{l_i}$, acting by permutations of the basis vectors
$\ve_s$, and an abelian group $\Theta_i$. If
$\Delta_i\cong\BB_{l_i}$, then $\Theta_i=\langle c_s\rangle_{s\in
S_i}$, in particular $c_s\in W_i$. If $\Delta_i\cong\DD_{l_i}$, then
$\Theta_i=\langle c_s c_{s'}\rangle_{s,s'\in S_i}$. In this case
$c_s\notin W_i$, but $c_s c_{s'}\in W_i$.

\begin{proposition} \label{prop1.not-qp}
Let $\Delta$, $S$, $M$, and $W$ be as above.
Assume that $|\Delta|\ge 3$. Then the $W$-lattice $M$ is not
quasi-invertible.
\end{proposition}

\begin{remark}
Note that $\rank(M) = \dim(V) = |\Delta|$. If
$|\Delta|=1$ or $2$ then $M$ is quasi-permutation
by Lemma~\ref{lem:Voskresenskii}.
\end{remark}

\begin{proof}
First we consider the case $\Delta \cong\DD_4$. Then $M$ is not
quasi-per\-mu\-ta\-tion, see \cite[\S7.1]{CK}. We will show that $M$
is not quasi-invertible. Indeed, in \cite[\S7.1]{CK} the authors
construct a subgroup $U \subset W$ of order 8\footnote{This group of
order $8$ is actually denoted by $W_2$ in \cite{CK}. We use the
symbol $U$ here to avoid a notational clash with the Weyl group
$W_2:= W(\Delta_2)$.}, such that $M$ restricted to $U$ is a direct
sum of $U$-sublattices $M=M_1\oplus M_3$ of ranks 1 and 3,
respectively.  Now in \cite[Theorem~1]{K-Selecta} it is {\em stated}
that the $U$-lattice $M_3$ is not quasi-permutation, but it is
actually {\em proved} that $[M_3]^\fl$  is not invertible. Hence
$M_3$ is not a quasi-invertible $U$-lattice, and $M$ is not a
quasi-invertible $W$-lattice.

   From now on we will assume that $\Delta\not\cong\DD_4$. Let
$\Gamma=\Z/2\Z\times\Z/2\Z=\{e,\gamma_1,\gamma_2,\gamma_3\}$. Then
by Proposition~\ref{thm:J-Gamma}, $\Sh(\Gamma, J_\Gamma)\cong\Z/2\Z$.
The idea of our proof is to construct an embedding
\begin{equation} \label{e.gamma}
\iota\colon\Gamma\to W
\end{equation}
in such a way that $M$, restricted to $\iota(\Gamma )$, is
isomorphic to a direct sum of a submodule $M_0 \simeq J_{\Gamma}$ and $|S|-3$
$\Gamma$-lattices of rank $1$.  This will imply that
\[ \Sh(\Gamma, M)= \Sh (\Gamma, M_0) = \Sh (\Gamma ,J_{\Gamma })=\Z/2\Z\neq 0,
\]
and hence $M$ is not quasi-invertible.
We will now fill in the details of this argument
in two steps.
\medskip

\noindent
  {\bf Step 1. Construction of the embedding~\eqref{e.gamma}.}
We begin by partitioning each $S_i$ for $i=1,\dots,m$ into three (non-overlapping)
subsets $S_{i, 1}$, $S_{i, 2}$ and $S_{i, 3}$,
subject to the requirement that
\begin{equation} \label{align:vk}
|S_{i, 1}| \equiv
|S_{i, 2}| \equiv
|S_{i, 3}|
\equiv l_i \text{ (mod ${2}$), if  $\Delta_i$ is of type
$\DD_{l_i}$.}
\end{equation}
We then set $U_1$ to be the union of the $S_{i, 1}$,
$U_2$ to be the union of the $S_{i, 2}$,
and $U_3$ to be the union of the $S_{i, 3}$,
as $i$ ranges from $1$ to $m$.

\begin{lemma} \label{lem.partitions}
If $|S| \ge 3$ and $\Delta \not\cong \DD_4$ then
the subsets $S_{i, 1}$, $S_{i, 2}$ and $S_{i, 3}$ of $S_i$
can be chosen, subject to~\eqref{align:vk}, so that
$U_1, U_2, U_3 \neq \emptyset$.
\end{lemma}

To prove the lemma, note that if one of the $\Delta_i$, say $\Delta_1$,
is of type $\DD_l$, where $l \ge 3$ is odd, then we
partition $S_1$ into three non-empty sets of odd order.
If $m \ge 2$ then we
partition $S_{i}$ with $i \ge 2$ as follows:
\begin{equation} \label{e.Si}
  \text{$S_{i, 1} = S_{i, 2} = \emptyset$ and $S_{i, 3} = S_i$.}
\end{equation}
Clearly $U_1, U_2, U_3 \neq \emptyset$, as desired.

Similarly, if one of the $\Delta_i$, say $\Delta_1$, is
$\DD_l$, where $l \ge 6$ is even, then we
partition $S_1$ into three non-empty sets of even order,
and partition the other $S_i$ (if any) as in~\eqref{e.Si} for $i \ge 2$.
Once again, $U_1, U_2, U_3 \neq \emptyset$.

If one of the $\Delta_i$, say $\Delta_1$, is of type $\DD_4$, then
by our assumption $m \ge 2$. We can now partition $S_1$ so that
each of $S_{1, 1}$ and $S_{1, 2}$ has 2 elements and $S_{1, 3} = \emptyset$,
and partition $S_i$ as in~\eqref{e.Si} for every $i \ge 2$.
Once again, $U_1, U_2, U_3 \neq \emptyset$.

Thus we may assume without loss of generality that every
$\Delta_i$  is of type $\BB_{l_i}$. In this case condition~\eqref{align:vk}
doesn't come into play and the lemma is obvious. This completes
the proof of Lemma~\ref{lem.partitions}.
\qed
\smallskip

We now define the embedding $\iota$ of~\eqref{e.gamma} by
$$\iota(\gamma_\vk)=\prod_{s\in S \smallsetminus U_\vk}  c_s\in \Aut(M) \text{ for } \vk = 1, 2, 3.$$
  Recall that
$\Gamma=\Z/2\Z\times\Z/2\Z=\{e,\gamma_1,\gamma_2,\gamma_3\}$.
One easily checks that the map
$\iota \colon \Gamma \to \Aut(M)$ defined this way is
a group homomorphism.
By \eqref{align:vk}, its image is, in fact, in $W$. Moreover,
since $U_\vk \ne \emptyset$ for all $\vk=1,2,3$,
we have $S\smallsetminus U_\vk\ne\emptyset$, hence
$\iota(\gamma_\vk) \ne \id$, i.e.,
$\iota \colon \Gamma \to W$ is injective.
We identify $\Gamma$ with $\iota(\Gamma)\subset W$.

\medskip

\noindent
{\bf Step 2. Construction of the submodule $M_0$.}
Now let \[\beta_\vk:=\gamma_\vk (\beta) =
\half\left(\sum_{s \in U_\vk}\ve_s - \sum_{s \in S \smallsetminus U_\vk} \ve_s \right)
\]
for $\vk = 1, 2, 3$, where $\beta$ is as in~\eqref{e.beta}.
Since the set $\{\beta,\beta_1,\beta_2,\beta_{3}\}$ is the orbit of $\beta$ under $\Gamma$,
the sublattice $M_0:=\Span_{\mathbb Z}(\beta, \beta_1,\beta_2,\beta_3)\subset M$ is $\Gamma$-invariant.
Note that
\begin{equation} \label{eq:sigma-eps}
\beta + \beta_\vk =\sum_{s\in U_\vk} \ve_s\,.
\end{equation}
Since $U_1$, $U_2$ and $U_3$ are non-empty and disjoint,
$\beta +\beta_{1}$,   $\beta +\beta_{2}$, and  $\beta +\beta_{3}$
are linearly independent. On the other hand,
\[
  \beta +\beta_1+\beta_2+\beta_{3}=0.
\]
Therefore, the $\Gamma$-invariant sublattice $M_0 \subset M$ is of rank 3 and
is isomorphic (as a $\Gamma$-lattice) to
$J_\Gamma:=\Z[\Gamma]/\Z$.

It remains to show that $M$ can be written as a direct sum of
$M_0$ and $\Gamma$-lattices of rank $1$. Indeed,
for each $\vk =1,2,3$ choose an element $u_\vk \in U_\vk$ and set
$U'_\vk =U_\vk \smallsetminus \{u_\vk \}$. (Note that $U'_\vk$ may be
empty for some $\vk$). We set $S'=U'_1\cup U'_2\cup U'_3$. It follows from
\eqref{eq:sigma-eps} that the abelian group generated by the $\ve_s$,
as $s$ ranges over $S'$, together with
$\beta,\beta_1,\beta_2,\beta_{3}$, contains both $\beta$ and
$\ve_s$ for every $s\in S$ and hence, coincides with $M$.
Since $\rank(M) = |S|$, we conclude that
the set $\{\beta, \beta_1,\beta_2\}\cup \{\ve_s \ |\ s\in S'\}$
is a basis of
$M$. The group $\Gamma$ acts on $\ve_s$ by $\pm 1$. We see that the
$\Gamma$-lattice $M$ is a direct sum of
$M_0 =\Span_{\mathbb Z}(\beta, \beta_1,\beta_2)$ and
  the $\Gamma$-lattices $\mathbb Z e_s$ of rank $1$, as $s$ ranges over $S'$.
Thus
\[
\Sh(\Gamma, M)= \Sh(\Gamma, M_0)= \Sh(\Gamma, J_{\Gamma}) = \Z/2\Z,
\]
and therefore $M$ is not a quasi-invertible $W$-lattice, as desired.
\end{proof}

\section{More non-quasi-invertible lattices}

In this section we continue to create a stock of non-quasi-invertible
lattices which will be used in the proof of
Theorem~\ref{thm:product-closed-field}.

\begin{proposition} \label{prop2.not-qp}
Let $M= \{ (a_1,a_2,a_3) \in\Z^3\  |\   a_1+a_2+a_3\equiv 0 \pmod{2} \}$
be the $W:=(\Z/2\Z)^3$-lattice
with the action of $(\Z/2\Z)^3$ on $M\subset \Z^3$
coming from the non-trivial action of $\Z/2\Z$ on $\Z$.
Then $M$ is not quasi-invertible.
\end{proposition}

\begin{proof}
Let $\ve_1,\ve_2,\ve_3$ be the standard basis of $\Z^3$.
For $i=1,2,3$ let $c_i\in W$ denote the automorphism of $M$ taking $\ve_i$ to $-\ve_i$
and fixing each of  the  other two  $\ve_j$.
Set $\sigma=c_2 c_3$, $\tau=c_1c_2$, $\rho=c_1 c_2 c_3$.
We consider the following basis of $M$:
\[
e_1=\ve_2-\ve_1,\  e_2=\ve_2-\ve_3,\  e_3=-\ve_1-\ve_3.
\]
A direct calculation shows that in this new basis $\{e_1,e_2,e_3\}$,
the generators $\sigma,\tau,\rho$ of $W$ are given by the following matrices:
\[
\sigma=\left(
\begin{array}{rrr}
0 &0 &1\\
-1 &-1 &-1\\
1 &0 & 0
\end{array}\right),
\
\tau=\left(
\begin{array}{rrr}
-1 & -1 & -1\\
0 &  0 & 1\\
0 &1 &0
\end{array} \right),
\
\rho=\left(
\begin{array}{rrr}
-1 & 0 & 0\\
0 & -1 & 0\\
0 & 0 & -1
\end{array} \right).
\]
By \cite[Theorem~1, case $W_2$]{K-Selecta}, our $W$-lattice $M$ is
not quasi-permutation. Moreover, the pair $(W,M)$ is isomorphic to
$(U, M_3)$, where $M_3$ is the non-quasi-invertible $U$-lattice
we mentioned at the beginning of the proof of
Proposition~\ref{prop1.not-qp}. Therefore, $M$ is not
quasi-invertible.
\end{proof}
\medskip

Let $\Z \DD_3$ denote the root lattice of $\DD_3$.
Recall that
$$\Z\DD_3=\{a_1\ve_1+a_2\ve_2+a_3\ve_3\ |\ a_i\in \Z,\ a_1+a_2+a_3\in 2\Z\}\subset\Q^3,$$
where $\{\ve_1,\ve_2,\ve_3\}$ is the standard basis of
$\Q^3=\Q\DD_3$. The set
$$
\{\ve_1+\ve_2,\quad\ve_1-\ve_2,\quad\ve_2-\ve_3\}
$$
is a basis of $\Z \DD_3$.

Let $m\ge 2$. We consider $(\Z \DD_3)^m\subset (\Q \DD_3)^m$. Let
$L\subset (\Q \DD_3)^m$ be the lattice generated by $(\Z \DD_3)^m$
and the vector
$$
v_e:=\ve_1+\ve_4+\ve_7+\dots+\ve_{3m-2}.
$$
The group $W(\DD_3)^m$ acts on $L$.

\begin{proposition}\label{prop:SO6}
For $m\ge 2$, the  $W(\DD_3)^m$-lattice $L$ constructed above
is not quasi-invertible.
\end{proposition}

\begin{proof}
We consider the subgroup $\Gamma\subset W(\DD_3)^m$ of order 4
generated by the following two commuting elements of order 2:
\begin{align*}
a=&(12)\ c_4 c_5\ c_7 c_8\ \dots\ c_{3m-2} c_{3m-1},\\
b=&c_1c_2\ (45).
\end{align*}
Here $c_i$ takes $\ve_i$ to $-\ve_i$.
Thus $\Gamma=\{e,a,b,ab\}\subset W(\DD_3)^m$.
We show that
$\Sh(\Gamma,L)=\Z/2\Z$.

Indeed, let $V=(\Q \DD_3)^m$ with the basis $\ve_1,\dots,\ve_{3m}$.
Let $V_0$ be the subspace of $V$
spanned by
$$ \ve_1,\ve_2,\ \ve_4,\ve_5,\
\dots,\ve_{3m-2},\ve_{3m-1}.
$$
It is $\Gamma$-invariant.
Set $L_0=L\cap V_0$.
Clearly $L/L_0$ injects into $V/V_0$.
Since $\Gamma$ acts trivially on $V/V_0$,
we see that $L/L_0\cong\Z^m$ with trivial $\Gamma$-action.
Thus we have a short exact sequence of $\Gamma$-lattices
$$
0\to L_0\to L\to \Z^m\to 0.
$$
Since $\Z^m$ is a permutation $\Gamma$-lattice, we see that
$$
\Sh(\Gamma,L)\cong\Sh(\Gamma,L_0).
$$
We prove that $\Sh(\Gamma,L_0)=\Z/2\Z$.

For $\gamma\in \Gamma$ we set $v_\gamma=\gamma\cdot v_e$.
If $m>2$
we set
$$
\delta=\ve_7+\ve_{10}+\dots+\ve_{3m-2}.
$$
If $m=2$ we set $\delta=0$.
We obtain
\begin{align*}
&v_e=\ve_1+\ve_4+\delta,\\
&v_a=\ve_2-\ve_4-\delta,\\
&v_b=-\ve_1+\ve_5+\delta,\\
&v_{ab}=-\ve_2-\ve_5-\delta.
\end{align*}
Clearly
$$
v_e+v_a+v_b+v_{ab}=0.
$$
Set $M_0=\langle v_e,v_a,v_b,v_{ab}\rangle$, then $M_0\cong
J_\Gamma:=\Z[\Gamma]/\Z$, and by Proposition~\ref{thm:J-Gamma} we have
$\Sh(\Gamma,M_0)=\Z/2\Z$.

Set $\beta_1=v_e,\ \beta_2=v_a,\ \beta_3=v_b$.
We set
\begin{align*}
&\beta_4=\ve_4-\ve_5,\\
&\beta_5=\ve_7+\ve_8,\\
&\beta_6=\ve_7-\ve_8,\\
&\dots\dots  \dots\dots\\
&\beta_{2m-1}=\ve_{3m-2}+\ve_{3m-1},\\
&\beta_{2m}=\ve_{3m-2}-\ve_{3m-1}.
\end{align*}
By Lemma \ref{lem:basis} below, the set
$\bb:=\{\beta_1,\dots,\beta_{2m}\}$ is a basis of $L_0$. We have
$M_0=\langle\beta_1,\beta_2,\beta_3\rangle$. Our $\Gamma$-lattice
$L_0$ decomposes into a direct sum of $\Gamma$-sublattices
$$
L_0=M_0\oplus\langle\beta_4\rangle\oplus\dots\oplus\langle\beta_{2m}\rangle.
$$
For $4\le i\le 2m$ the $\Gamma$-lattice $\langle\beta_i\rangle$ is of rank 1,
hence quasi-permutation, and therefore
$\Sh(\Gamma,\langle\beta_i\rangle)=0$.
It follows that $\Sh(\Gamma,L_0)=\Sh(\Gamma,M_0)=\Z/2\Z$,
hence  $\Sh(\Gamma,L)=\Z/2\Z$.
Thus $L$ is not a quasi-invertible $W(\DD_3)^m$-lattice.
\end{proof}

\begin{lemma}\label{lem:basis}
The set $\bb:=\{\beta_1,\dots,\beta_{2m}\}$ is a basis of $L_0$.
\end{lemma}

\begin{proof}
First note that $\bb\subset L_0$.
Since the set $\bb$ has $2m$ elements and the lattice $L_0$ is of rank $2m$,
it suffices to show that $\bb$ generates $L_0$.

Recall that $L_0=L\cap V_0$ and that $L$ is generated by $(\Z \DD_3)^m$ and $v_e$.
Since $v_e\in V_0$, we see that $L_0$ is generated by $v_e$ and $(\Z \DD_3)^m\cap V_0$.
Since $v_e=\beta_1\in\bb$, it suffices to prove that $(\Z \DD_3)^m\cap V_0\subset \bgen$.
Clearly $(\Z \DD_3)^m\cap V_0$ is generated by the vectors
$$
\ve_1+\ve_2\ ,\ve_1-\ve_2,\ \ve_4+\ve_5, \ \ve_4-\ve_5,\ \dots,\ \ve_{3m-2}+\ve_{3m-1},\ \ve_{3m-2}-\ve_{3m-1}.
$$
Note that all the vectors in this list starting with $\ve_4-\ve_5$ are clearly contained in $\bb$.
It remains to show that the vectors $\ve_1+\ve_2\ ,\ve_1-\ve_2,\ \ve_4+\ve_5$ are contained in $\bgen$.

Note that $2\delta\in\bgen$ (because $2\ve_7\in\bgen,\dots,2\ve_{3m-2}\in\bgen$).
We have
$$
\beta_1+\beta_2=v_e+v_a=\ve_1+\ve_2,
$$
hence $\ve_1+\ve_2\in\bgen$.
We have
$$
\beta_1+\beta_3=v_e+v_b=\ve_4+\ve_5+2\delta,
$$
hence $\ve_4+\ve_5\in\bgen$.
Since also $\ve_4-\ve_5\in\bb\subset\bgen$,
we see that $2\ve_4\in\bgen$.
We have
$$
\beta_1-\beta_2=v_e-v_a=\ve_1-\ve_2+2\ve_4+2\delta,
$$
hence $\ve_1-\ve_2\in\bgen$.
We conclude that $(\Z \DD_3)^m\cap V_0\subset \bgen$, hence $\bb$ generates $L_0$ and is a basis of $L_0$.
This completes the proofs of Lemma \ref{lem:basis} and of Proposition  \ref{prop:SO6}.
\end{proof}
\medskip

We will now consider the root system $\AA_{n-1}$, which is
embedded in $\Z^n$, see \cite[Planche I]{Bourbaki}.
Let $\Z \AA_{n-1}$ denote the root lattice of $\AA_{n-1}$,
and let $\alpha_1,\alpha_2,\dots,\alpha_{n-1}$ denote the standard basis of the root system $\AA_{n-1}$
and of $\Z \AA_{n-1}$ ({\em loc.~cit}).
Let $\Lambda_n$ denote the weight lattice of $\AA_{n-1}$,
and let $\omega_1,\omega_2,\dots,\omega_{n-1}$ denote the standard basis of $\Lambda_n$
consisting of fundamental weights ({\em loc.~cit}).

Consider $\Z\AA_2\subset\Lambda_3$.
The nontrivial automorphism $\sigma$
of the basis $\Delta=\{\alpha_1,\alpha_2\}=\{\ve_1-\ve_2,\ve_2-\ve_3\}$ ({\em loc.~cit})
induces the automorphism $(-1)\circ(1,3)$ of $\Z\AA_2$
(where $-1\in\Aut\Z\subset\Aut\Z^3$, $(1,3)\in \Sym_3\subset\Aut\Z^3$),
and an automorphism $\sigma_*$ of $\Sym_3=W(\AA_2)$
(namely, the conjugation by the transposition $(1,3)$).

Let $m\ge 2$.
We consider $(\Z \AA_2)^m\subset (\Lambda_3)^m$. Let $(\Z
\AA_2)^{(i)}\subset\Lambda_3^{(i)}$ be the $i^{\textit{th}}$ factor.
Let $\omega_1^{(i)}, \omega_2^{(i)}$ be the basis of
$\Lambda_3^{(i)}$ consisting of fundamental weights.

Let $\ba=(a_1,\dots,a_m)$ (a row vector), where each  $a_i$ equals $1$ or $2$. In
particular, let  $\bo=(1,\dots,1)$. Let $L_\ba$ denote the
$(\Sym_3)^m$-lattice generated by  $(\Z \AA_2)^m$ and the vector
\[
x_\ba:=\sum_{i=1}^m a_i \omega_1^{(i)}.
\]

\begin{proposition}\label{prop:SL3}
For $m\ge 2$ and for any $\ba$ as above
$($i.e., each $a_i$ equals $1$ or $2)$, the $(\Sym_3)^m$-lattice
$L_\ba$  is not  quasi-invertible.
\end{proposition}

\begin{proof}
First we note that $L_\ba$ is $\varphi$-isomorphic
to $L_\bo$ with respect to some
automorphism $\varphi$ of $(\Sym_3)^m$ (for a definition
of $\varphi$-isomorphism, see the beginning of Section~\ref{sect.lattices}).

Indeed, let $\alpha_1,\alpha_2$ be the
standard basis of the root system $\AA_2$ (and of $\Z\AA_2$). Let
\[
\omega_1=\frac{1}{3}\left(2\alpha_1+\alpha_2\right),\ \omega_2=\frac{1}{3}\left(\alpha_1+2\alpha_2\right)
\]
be the fundamental weights, this is the standard basis of $\Lambda_3$ ({\em loc.~cit.}).
Let $\ombar_1,\ombar_2$ be their images in $\Lambda_3/\Z\AA_2\cong\Z/3\Z$.
Since $$\omega_1+\omega_2=\alpha_1+\alpha_2\in\Z\AA_2,$$
we have $\ombar_1+\ombar_2=0$, hence $\ombar_2=2\ombar_1$.
Thus the nontrivial automorphism $\sigma$ of the Dynkin diagram  $\AA_2$
takes $\ombar_1$ to $\ombar_2=2\ombar_1$ when acting on $\Lambda_3/\Z\AA_2$.

Now let $\ba$ be as above. Write $\Delta=(\AA_2)^m$,
$\Delta=\Delta_1\cup\dots\cup \Delta_m$. For each $i=1,\dots,m$ we define an
automorphism $\tau_i$ of $\Delta_i=\AA_2$. If $a_i=1$, we set
$\tau_i=\id$, while if $a_i=2$, we set $\tau_i=\sigma_i$, where
$\sigma_i$ is the nontrivial automorphism of $\Delta_i$. Then the
automorphism $\tau=\prod_i\tau_i$ of $\Delta=(\AA_2)^m$ acts on
$(\Lambda_3)^m$ and takes $L_\bo$ to $L_\ba$. We see that the
$(\Sym_3)^m$-lattices $L_\bo$ and $L_\ba$ are $\tau_*$-isomorphic,
where $\tau_*$ is the induced automorphism of $(\Sym_3)^m$. Thus,
in order to prove that the $(\Sym_3)^m$-lattice $L_\ba$ is not
quasi-invertible, it suffices to show that $L_\bo$ is not quasi-invertible.

Let $\alpha_1^{(i)},\alpha_2^{(i)}$ be the standard basis of
$(\Z \AA_2)^{(i)}$.
Let $\omega_1^{(i)},\omega_2^{(i)}$ be the standard basis of $\Lambda_3^{(i)}$, then
\[
\omega_1^{(i)}=\frac{1}{3}(2\alpha_1^{(i)}+\alpha_2^{(i)}).
\]

Let $\alpha_1,\dots,\alpha_{3m-1}$ be the standard basis of $\Z \AA_{3m-1}$.
We denote by $\lambda_1,\dots,\lambda_{3m-1}$ (rather than $\omega_1,\dots,\omega_{3m-1}$)
  the standard basis of $\Lambda_{3m}$ consisting of fundamental weights.
Then we have ({\em loc.~cit.})
\begin{equation}\label{eq:lambda}
\lambda_1=\frac{1}{3m}\left( (3m-1)\alpha_1+(3m-2)\alpha_2+\dots+2\alpha_{3m-2}+\alpha_{3m-1}\right).
\end{equation}

We embed $(\Z \AA_2)^m$ into $\Z \AA_{3m-1}$
as follows:
\[
\alpha_1^{(i)}\mapsto \alpha_{3(i-1)+1},\quad \alpha_2^{(i)}\mapsto \alpha_{3(i-1)+2}
\]
(i.e., $\alpha_1^{(1)}\mapsto\alpha_1$,
$\alpha_2^{(1)}\mapsto\alpha_2$, $\alpha_1^{(2)}\mapsto\alpha_4$,
$\alpha_2^{(2)}\mapsto\alpha_5$, etc.). This embedding induces an
embedding
\[
\psi\colon (\Q \AA_2)^m \into \Q \AA_{3m-1}.
\]
Set
\[
M=\Lambda_{3m}\cap\psi((\Q \AA_2)^m).
\]
We show that $M=\psi(L_\bo)$.
Since by \eqref{eq:lambda} the image of $\lambda_1$ in $\Lambda_{3m}/\Z \AA_{3m-1}$
is of order $3m$, we see that
$\Lambda_{3m}$ is generated by $\Z \AA_{3m-1}$ and $\lambda_1$,
hence the set $\{\alpha_1,\dots,\alpha_{3m-1},\lambda_1\}$ is a generating set for $\Lambda_{3m}$.
From \eqref{eq:lambda} we see that
\[
\alpha_{3m-1}=3m\lambda_1-(3m-1)\alpha_1-(3m-2)\alpha_2-\dots-2\alpha_{3m-2},
\]
hence the set $\Xi:=\{\alpha_1,\dots,\alpha_{3m-2},\lambda_1\}$ is a basis for $\Lambda_{3m}$.
The subset
$$
\Xi':=\{\alpha_1,\alpha_2,\ \alpha_4,\alpha_5,\ \dots,\
                  \alpha_{3m-5},\alpha_{3m-4},\ \alpha_{3m-2}\}
$$
of $\Xi$ is contained in $M$.
Set $N:=\Z[\Xi\smallsetminus\Xi']\cap M\subset\Q\AA_{3m-1}$, then
clearly $M=\Z\Xi'\oplus N$.
Since $\rank M=2m=|\Xi'|+1$, we see that $\rank N=1$.
The element
\begin{align*}
\mu:= &
m\lambda_1-(m-1)\alpha_3-(m-2)\alpha_6-\dots-\alpha_{3m-3}=\frac{1}{3}((3m-1)\alpha_1\\
& +(3m-2)\alpha_2
  +(3m-4)\alpha_4+(3m-5)\alpha_5+\dots +2\alpha_{3m-2}+\alpha_{3m-1})
\end{align*}
is contained in $N$ and indivisible in $M$, hence the one-element set $\{\mu\}$ is a basis of $N$,
and $\Xi'\cup\{\mu\}$ is a basis of $M$.
Now
\begin{align*}
&\mu-(m-1)(\alpha_1+\alpha_2)-(m-2)(\alpha_4+\alpha_5)-\dots-1(\alpha_{3(m-2)+1}+\alpha_{3(m-2)+2})\\
&=\frac{1}{3}((2\alpha_1+\alpha_2)+(2\alpha_4+\alpha_5)+\dots+(2\alpha_{3m-2}+\alpha_{3m-1}))\\
&=\psi(\omega_1^{(1)}+\omega_1^{(2)}+\dots+\omega_1^{(m)}).
\end{align*}
We see that $M$ is generated by $\psi((\Z \AA_2)^m)$ and $\psi(\omega_1^{(1)}+\omega_1^{(2)}+\dots+\omega_1^{(m)})$,
hence $M=\psi(L_\bo)$, thus $M$ is isomorphic to $L_\bo$.
Therefore, it suffices to prove that $M$ is not quasi-invertible.

The quotient lattice $\Lambda_{3m}/M$ injects into the $\Q$-vector
space $$\Q \AA_{3m-1}/\psi((\Q \AA_2)^m)$$ with basis
$\ov{\alpha_3},\ov{\alpha_6},\dots,\ov{\alpha_{3(m-1)}}$ on which
$(\Sym_3)^m$ acts trivially. Thus we obtain a short exact sequence
\[
0\to M\to \Lambda_{3m}\to \Z^{m-1}\to 0,
\]
where $\Z^{m-1}$ is a trivial, hence permutation, $(\Sym_3)^m$-lattice.
It follows that the  $(\Sym_3)^m$-lattices $M$ and $\Lambda_{3m}$ are equivalent,
and therefore it suffices to show that $\Lambda_{3m}$ is not  a quasi-invertible
$(\Sym_3)^m$-lattice.

Now we embed  $\Sym_3\times \Sym_3$ into $(\Sym_3)^m$ as follows: $(s,t)\in
\Sym_3\times \Sym_3$ maps to $(s,t,\dots,t)\in (\Sym_3)^m$. With the notation
of \cite[(6.4)]{LPR06} we have $\Lambda_{3m}=Q_{3m}(1)$. By
\cite[Proposition~7.1(b)]{LPR06}, with respect to the above embedding
$\Sym_3\times \Sym_3\into (\Sym_3)^m$, we have
\[
Q_{3m}(1)|_{\Sym_3\times \Sym_3}\sim \Lambda_6|_{\Sym_3\times \Sym_3} .
\]
By  \cite[Proposition~7.4(b)]{LPR06}, $\Lambda_6$ is not a
quasi-permutation $\Sym_3\times \Sym_3$-lattice, and it is actually
proved there that $[\Lambda_6]^\fl$ is not  an invertible
$\Sym_3\times \Sym_3$-lattice. It follows that $\Lambda_6$ is not a
quasi-invertible $\Sym_3\times \Sym_3$-lattice (although we have
$\Sh(\Gamma',\Lambda_6)=0$ for every subgroup $\Gamma'$ of
$\Sym_3\times \Sym_3$). Thus $\Lambda_{3m}$ is not a
quasi-invertible $\Sym_3\times \Sym_3$-lattice, hence it is not a
quasi-invertible $(\Sym_3)^m$-lattice. Thus $L_\bo$ is not a
quasi-invertible $(\Sym_3)^m$-lattice, and therefore $L_\ba$ is not
a quasi-invertible $(\Sym_3)^m$-lattice for any $\ba$ as above. This
completes the proof of Proposition \ref{prop:SL3}.
\end{proof}

\section{Standard subgroups}
\label{sect10}

In this and the next sections we will collect several elementary results from
combinatorial linear algebra, which will be needed to complete
the proof of~Theorem~\ref{thm:product-closed-field}.
\medskip

Let $e_1, \dots, e_m$ be the standard $\bbZ/n \bbZ$-basis of $(\bbZ/n \bbZ)^m$.
We say that a subgroup $S \subset (\bbZ/n \bbZ)^m$ is
{\em standard} if $S$ is generated by
$n_1 e_1, \dots, n_r e_r$ for some $1 \le r \le m$ and some
integers $n_1, \dots, n_r$, where $n_i$ divides $n_{i+1}$
for $i = 1, \dots, r-1$.

Let $W$ be a finite group, $P$ a $W$-lattice, and
$\lambda\colon P \to \bbZ/n \bbZ$
a surjective morphism of $W$-modules, where $W$ acts trivially on
$\bbZ/n \bbZ$. Given a subgroup $S$ of $(\bbZ/n \bbZ)^m$, let
$P^m_S$ denote the preimage of $S$ in $P^m$ with respect to the
homomorphism $\lambda^m\colon P^m\to(\Z/n\Z)^m$. We regard $P^m_S$
as a $W$-submodule of $P^m$, where $W$ acts diagonally on $P^m$.

\begin{lemma} \label{lem:standard}
Let $W$, $P$, $n$ and $\lambda$ be as above.
For every subgroup $S \subset (\bbZ/n
\bbZ)^m$ there exists a standard subgroup $S_\st \subset (\bbZ/n
\bbZ)^m$ with the following property: there exist an isomorphism
$g_P\colon P^m_S\isoto P^m_{S_\st}$ of  $W$-modules and an automorphism
$g$ of $(\Z/n\Z)^m$ taking $S$ to $S_\st$ such that the following
diagram commutes:
\[
\xymatrix{
P^m_S \ar[d]^{g_P}\ar[rr]^{\lambda^m} & & S\ar[d]^{}\ar@{^{(}->}[r] &(\Z/n\Z)^m\ar[d]^{g}\\
P^m_{S_\st}\ar[rr]^{\lambda^m} & & S_\st\ar@{^{(}->}[r] &(\Z/n\Z)^m.
}
\]
\end{lemma}

\begin{proof}
The homomorphism $\lambda^m\colon P^m\to(\Z/n\Z)^m$ can be written as
\[
\xymatrix{
\lambda^m=\lambda\otimes_\Z \id \colon P\otimes_\Z \Z^m \ar[r] &
\Z/n\Z\otimes_\Z \Z^m.}
\]
Since for any $g\in\GL_m(\Z)=\Aut(\Z^m)$  the diagram
\[
\xymatrix{
P\otimes_{\Z}\Z^m\ar[d]^{\id_P\otimes g}\ar[rr]^{\lambda\otimes \id_{\Z^m}} &
&\Z/n\Z\otimes_{\Z}\Z^m\ar[d]^{\id_{\Z/n\Z}\otimes g}\\
P\otimes_{\Z}\Z^m\ar[rr]^{\lambda\otimes \id_{\Z^m}}  &     &\Z/n\Z \otimes_{\Z}\Z^m
}
\]
commutes, it suffices to show that for every subgroup
$S \subset (\bbZ/n \bbZ)^m$ there exists
$g \in \GL_m(\bbZ)$ such that $g(S)$ is  standard.

Let $\pi \colon \bbZ^m \to (\bbZ/ n \bbZ)^m$ be the natural
projection. Then $\pi^{-1}(S)$ is a finite index subgroup of
$\bbZ^m$.  There exist a basis $b_1,
\dots, b_m$ of $\bbZ^m$ and integers $n_1 \,|\, n_2 \,| \dots \, |
\, n_m$,  such that $n_1 b_1, \dots, n_m b_m$ form a basis of
$\pi^{-1}(S)$; cf.~\cite[Theorem III.7.8]{lang}.
  Now let $g \in \GL_m(\bbZ)$ be the element that takes
the basis $b_1, \dots, b_m$ to the standard basis of $\bbZ^m$. Then
$g(\pi^{-1}(S))$ is the subgroup $n_1 \bbZ \times\dots \times n_m
\bbZ$ of $\bbZ^m$ and thus $S_\st:=g(S) = \langle n_1 e_1,\dots n_m e_m\rangle =\langle n_1 e_1, \dots,
n_r e_r \rangle$ is standard, where $r\le m$ is the largest integer
such that $n$ does not divide $n_r$.
\end{proof}

Set $Q=\ker\,\lambda\subset P$.
For a subgroup $S_1\subset \Z/n\Z$ we set $P^1_{S_1}=\lambda^{-1}(S_1)$,
so that $Q\subset P^1_{S_1}\subset P$.

\begin{corollary}\label{cor:standard1}
Assume that $S$ in Lemma \ref{lem:standard}
is cyclic. Then
\[
P^m_S\cong P^1_{S_1}\oplus Q^{m-1}
\]
for some subgroup $S_1\subset \Z/n\Z$ isomorphic to $S$.
\end{corollary}

\begin{proof}
By Lemma \ref{lem:standard}, we have $P^m_S\cong P^m_{S_\st}$. Since
$S$ is cyclic, say of order $s$, the group $S_\st$ is generated by
$(n/s)e_1$. Set $S_1=\langle (n/s)e_1\rangle\subset\Z/n\Z$, then
clearly
$$
P^m_{S_\st}=P^1_{S_1}\oplus Q^{m-1},
$$
and the corollary follows.
\end{proof}

\begin{corollary}\label{cor:standard2}
Assume that $S$ in Lemma \ref{lem:standard}
contains an element of order $n$. Then $P^m_S$ has a direct summand
isomorphic to $P$.
\end{corollary}

\begin{proof}
By Lemma \ref{lem:standard}, $P^m_S$ is isomorphic to $P^m_{S_\st}$
for some standard subgroup $S_\st\subset(\Z/n\Z)^m$. From the
definition of a standard subgroup we see that
\[
P^m_{S_\st}= P^1_{S_1}\oplus\dots\oplus P^1_{S_m}\;,
\]
where $S_i\subset\Z/n\Z$ is generated by $n_i e_i$ (for $i>r$ we
take $n_i=0$). Since $S_\st$ contains an element of order $n$, we
see that $n_1=1$, hence $S_1$ is generated by $e_1$, i.e., $S_1=\Z/n\Z$ and
$P^1_{S_1}=P$. Thus $P^m_S$ has a direct summand isomorphic to $P$.
\end{proof}

\section{Coordinate and almost coordinate subspaces}
\label{sect.coord}

Let $F$ be a field and let
$F^m$ be an $m$-dimensional $F$-vector space equipped with the
standard basis $e_1 = (1, 0, \dots, 0), \dots, e_m = (0, \dots, 0,
1)$.

Recall that the {\em Hamming weight} of a vector $v = (a_1, \dots,
a_m) \in F^m$ is defined as the number of non-zero elements among
$a_1, \dots, a_m$. We will say $v \in F^m$ is {\em defective} if its
Hamming weight is $< m$ or, equivalently, if at least one of its
coordinates is $0$. The following lemma is well known; a variant of
it is used to construct the standard open cover of the Grassmannian
$\operatorname{Gr}(m, d)$ by $d(m-d)$-dimensional affine spaces,
see, e.g., \cite[\S1.5]{GH}. For the sake of completeness, we
supply a short proof.

\begin{lemma} \label{lem.combinatorics0}
Let $V$ be a vector subspace of $F^m$ of dimension $d \ge 2$. Then
$V$ has a basis consisting of defective vectors.
\end{lemma}

\begin{proof} Let $A$ be a $d \times m$ matrix whose rows
form a basis of $V$. Then
\[
  V =\{ wA \; | \;  w \in F^d \} \, .
\]
Note that for any invertible $d \times d$ matrix $B$, the rows of
$BA$ will also form a basis of $V$. Since the rows of $A$ are
linearly independent, $A$ has a nondegenerate $d \times d$ submatrix
$M$. Let $B = M^{-1}$. Then the $d \times m$ matrix $BA$ has a $d
\times d$ identity submatrix. Since $d \ge 2$, this implies that
every row of $BA$ is defective. The rows of $BA$ thus give us a
desired basis of defective vectors for $V$.
\end{proof}

\begin{definition} \label{def.coordinate}
We will say that a subspace $V\subset F^m$ is a {\em coordinate subspace}
if $V$ has a basis of coordinate vectors $e_{i_1}, \dots, e_{i_d}$,
for some $I = \{ i_1, \dots, i_d \} \subset \{ 1, \dots, m \}$.
We will denote such a subspace by $F_I$.

In subsequent sections we will occasionally use this notation in the more
general setting, where $F$ is a commutative ring but not necessarily
a field. In this setting $F_I$ will denote the free $F$-submodule
of $F^n$ generated by  $e_{i_1}, \dots, e_{i_d}$.
\end{definition}

\begin{lemma} \label{lem.combinatorics3}
Let $V \subset F^m$ be an $F$-subspace. Suppose $V \cap F_I$ is
coordinate for every $I \subsetneq \{ 1, \dots, m \}$, then either
\begin{itemize}
\item $V$ is the $1$-dimensional subspace spanned by 
some $\ba=(a_1, \dots, a_m)$, where $a_1\ne 0,\ \dots,\  a_m \ne 0$,
or
\item $V$ is coordinate.
\end{itemize}
\end{lemma}

\begin{proof} Assume that $V$ is not of the form
$\Span_F (\bf{a})$, where ${\bf a} = (a_1, \dots, a_m)$ and
$a_1\ne 0,\dots ,a_m \ne 0$.
Then $V$ has a basis of defective vectors. Indeed, if $\dim(V) = 1$
this is obvious, since every vector in $V$ is defective. The case
where $\dim(V) \ge 2$ is covered by Lemma~\ref{lem.combinatorics0}.

Clearly $v \in F^m$ is defective if and only if $v \in F_I$ for some
$I \subsetneq \{ 1, \dots, m \}$. Thus $V$ is spanned by $V \cap
F_I$, as $I$ ranges over the proper subsets of $\{ 1, \dots, m \}$.
By our assumption, each $V \cap F_I$ is coordinate and therefore is
spanned by coordinate vectors.
We conclude that $V$ itself is spanned by coordinate vectors, i.e.,
is coordinate, as desired.
\end{proof}

\begin{definition} \label{def.almost-coordinate}
We will say that $V \subset F^m$ is {\em almost coordinate} if
$V$ has a basis of the form
\begin{equation} \label{e.basis}
e_{i_1}, \dots, e_{i_r}, e_{j_1} + e_{h_1}, \dots, e_{j_s} +
e_{h_s},
\end{equation}
where $i_1, \dots, i_r, j_1, \dots, j_s, h_1, \dots, h_s$
are distinct integers between $1$ and $m$.
We will refer to a basis of this form as an {\em almost coordinate
basis} of $V$.
\end{definition}

\begin{remark} \label{rem.almost-coordinate-uniqueness}
An almost coordinate subspace $V \subset F^m$ has
a unique almost coordinate basis. In other words,
the set of integers $\{i_1, \dots, i_r\}$ and the set of unordered pairs
$\{ \{j_1,h_1\}, \dots, \{j_s, h_s\}\}$  in~\eqref{e.basis}
are uniquely determined by $V$.
Indeed, $\{i_1, \dots , i_r \}$ is the set of subscripts $i \in \{
1, \dots, m \}$ such that the coordinate vector $e_i$ lies in $V$.
The set $\{ \{j_1, h_1 \} , \{j_2, h_2\}, \dots , \{j_s, h_s \} \}$
is then the set of unordered pairs $\{j, h \}$ such that $j, h \not
\in \{ i_1, \dots , i_r \}$ and $e_j + e_h \in V$.
\end{remark}

\begin{proposition} \label{prop.combinatorics}
Let $F=\bbZ/2 \bbZ$, and let $V \subset F^m$ be an $F$-subspace for
some $m \ge 4$. Assume $V \cap F_I$ is almost coordinate in
$F_I\cong(\Z/2\Z)^r$ for every $I = \{ i_1, \dots, i_r \} \subsetneq
\{1,\dots, m\}$. Then either
\begin{itemize}
\item $V$ is the $1$-dimensional subspace spanned by $(1,\dots,1)$,
or
\item $V$ is almost coordinate.
\end{itemize}
\end{proposition}

\begin{proof} Assume that $V$ is not of the form
$\Span_F \, \{ (1,\dots,1) \}$.
Then, once again,
Lemma~\ref{lem.combinatorics0} tells us that $V$ has a basis of
defective vectors, i.e., $V$ is spanned by $V \cap F_I$, as $I$
ranges over the proper subsets of $\{ 1, \dots, m \}$. By our
assumption, each $V \cap F_I$ is almost coordinate and therefore is
spanned by vectors of Hamming weight $1$ or $2$. We conclude that
$V$ itself is spanned by vectors of weight $1$ or $2$. Choose a
spanning set of the form
\begin{equation} \label{e.min-wt}
e_{i_1}, \dots, e_{i_r},
e_{j_1} + e_{h_1}, \dots, e_{j_s} + e_{h_s}
\end{equation}
of minimal total Hamming weight, i.e., with minimal value of $r + 2s$.
Here
\[ i_1, \dots, i_r, j_1, h_1, \dots, j_s, h_s \in \{ 1, \dots, m \} \]
and $j_1 \ne h_1, \dots, j_s \neq h_s$. We claim that~\eqref{e.min-wt}
is  an almost coordinate basis of $V$,
i.e., that the subscripts
\begin{equation} \label{e.subscripts}
i_1, \dots, i_r, j_1, \dots, j_s, h_1, \dots, h_s
\end{equation}
are all distinct.  Clearly,
Proposition~\ref{prop.combinatorics} follows from this claim.

It thus remains to prove the claim. The minimality of the total
Hamming weight of our spanning set~\eqref{e.min-wt} implies that we
cannot remove any vectors, i.e., that it is a basis of $V$.  In
particular, the subscripts $i_1, \dots, i_r$ and the pairs $(j_1,
h_1), \dots, (j_s, h_s)$ are distinct. If there is an overlap among
subscripts~\eqref{e.subscripts}, then, after permuting
  coordinates, we have either $i_1 = j_1$ or $j_1 = j_2$. We will now
show that neither of these equalities can occur.

If $i_1 = j_1$ then we may replace $e_{j_1} + e_{h_1}$ by
\[ e_{h_1} = (e_{j_1} + e_{h_1}) - e_{i_1} \in V \, . \]
We will obtain a new spanning set consisting of vectors
of weight $1$ or $2$ with smaller total weight,
a contradiction.

Now suppose $j_1 = j_2$. Denote this number by $j$. Then $V \cap F_{
\{j, h_1, h_2 \}}$ contains the vectors
\begin{equation} \label{e.combinatorics}
\text{$e_{j} + e_{h_1}$ and $e_j + e_{h_2} \in V$.}
\end{equation}
Since we are assuming that $m \ge 4$, $\{j, h_1, h_2 \} \subsetneq
\{ 1, \dots, m \}$ and hence, $V \cap F_{ \{j, h_1, h_2 \}}$ is
almost coordinate. The subspace in  $F_{ \{j, h_1, h_2 \}}$
generated by the two vectors \eqref{e.combinatorics} is cut by the
linear equation
$$
x_j+x_{h_1}+x_{h_2}=0
$$
and clearly is not almost coordinate. It follows that $V \cap F_{
\{j, h_1, h_2 \}}=F_{ \{j, h_1, h_2 \}}$, hence $V$ contains all
three of the coordinate vectors $e_j$, $e_{h_1}$ and $e_{h_2}$.
Replacing $e_j + e_{h_1}$ and $e_j + e_{h_2}$ by $e_j, e_{h_1}$ and
$e_{h_2}$ in our spanning set, we reduce the total weight by one, a
contradiction. This completes the proof of the claim and thus of
Proposition~\ref{prop.combinatorics}.
\end{proof}

\section{Coordinate subspaces and quasi-permutation lattices}
\label{sect13}

\begin{proposition}\label{prop:coordinate-qp}
Let $W$ be a finite group,
$M$  a $W$-lattice, and  $\lambda\colon M\to F:=\Z/p\Z$
a surjective morphism of $W$-modules,
where $p$ is a prime and $W$ acts trivially on $F$.
For any $m\ge 1$, and an $F$-subspace $S \subset V:=F^m$,
let $M^m_S$ be the preimage of $S \subset F^m$ under
the projection $\lambda^m \colon M^m \to F^m$.

Assume that
\begin{enumerate}[\upshape(a)]
\item[\rm(a)] $M$ and $\ker[M\to F]$ are  quasi-permutation $W$-lattices;
\item[\rm(b)] the $W^m$-lattice $M^m_{S_1}$ is {\em not} quasi-permutation
for any $1$-dimensional subspace $S_1$ of $F^m$
of the form $S_1 = \Span_F \{ (a_1, \dots, a_n) \}$,
where $a_1\ne 0,\ \dots, a_m\neq 0$.
\end{enumerate}
Then, given a subspace $S\subset F^m$, $M^m_S$ is a
quasi-permutation $W^m$-lattice if and only if $S$ is coordinate.
\end{proposition}

The following notation will be helpful in the proof of
Proposition~\ref{prop:coordinate-qp} and in the subsequent sections.

\begin{definition} \label{defn.subscript}
Let $W$ be a finite group, $M$ a $W$-module and
$m$ a positive integer.
Given a subset $I \subset \{ 1 , \dots, m \}$, we define
the ``coordinate subgroup" $W_I \subset W^m$ as
\[ W_I: = \{ (w_1, \dots, w_m)\in W^m \, | \, \text{ $w_i = \id$
if $i \not \in I$}
\}.
\]
We will also define the $W_I$-submodule $M_I$ of $M^m$ as
$$
M_I := \{
(a_1, \dots, a_m)\in M^m \, | \, \text{$a_i = 0$ if $i \not \in I$}
\} .
$$
  We shorten $W_{\{i\}}$, $M_{\{i\}}$ to $W_i$, $M_i$.
\end{definition}

\begin{proof}[Proof of Proposition \ref{prop:coordinate-qp}]
The ``if'' assertion is clear. We will prove ``only if'' by
induction on $m$. In the base case, $m = 1$, every subspace of $V$
is coordinate, so there is nothing to prove.

For the induction step, assume that
$m \ge 2$ and that our assertion has been established for
all $m' < m$.  Suppose that for some subspace $S\subset F^m$
the lattice $M^m_S$ is quasi-permutation.
We want to show that $S$ is  coordinate.

Since  $M^m_S$ is  quasi-permutation, Lemma \ref{LPR-reduction-two}
tells us that $M^m_S\cap M_I$ is a quasi-permutation $W_I$-lattice for
every $I \subsetneq \{ 1 , \dots, m \}$
(cf.~Definition~\ref{defn.subscript} above).
But $M^m_S\cap M_I=M^m_{S\cap F_I}$, and so
by the induction
hypothesis $S\cap F_I$ is a coordinate subspace
in $F_I$ (and hence, in $F^m$).

Now Lemma~\ref{lem.combinatorics3} tells us that either $S$ is a
$1$-dimensional subspace of $F^m$ which does not lie in any
coordinate hyperplane or $S$ is a coordinate subspace in $F^m$. Our
assumption (b) rules out the first possibility. Hence, $S$ is a
coordinate subspace of $F^m$, as claimed.
\end{proof}

\section{Proof of Theorem~\ref{thm:product-closed-field}
\label{sect14} for $H$ of types $\AA_{n-1}$ ($n\ge 5$), $\BB_n$ ($n
\ge 3$) and $\DD_n$ ($n \ge 4$)}

Starting from this section,
we will prove Theorem~\ref{thm:product-closed-field}
case by case.

\begin{notation}\label{subsec:Q-P-F-W}
Let $R$ be an irreducible reduced root system. We denote by $Q=Q(R)$
the root lattice of $R$ and by $P=P(R)$ the weight lattice of $R$,
both lattices regarded as $W:=W(R)$-lattices.
Given a positive integer $m$ and a subset
$I \subset \{ 1 , \dots, m \}$, we
define $W_I \subset W^m$, and the $W_I$-modules
$Q_I$, $P_I$, etc.,
as in  Definition~\ref{defn.subscript}. The base field $k$ is assumed
to be algebraically closed of characteristic zero.
\end{notation}

\subsection{Case $\AA_{n-1}$ $(n\ge 5)$}

\begin{theorem} \label{thm.SLm}
Let $G = (\gSL_n)^m/C$, where $n\ge5$ and  $C$ is
a subgroup of $(\mu_n)^m=Z(\gSL_n^m)$. Then
the following conditions are equivalent:
\begin{enumerate}[\upshape(a)]
\item[\rm(a)] $G$ is Cayley,
\item[\rm(b)] $G$ is stably Cayley,
\item[\rm(c)] the character lattice $\X(G)$ is quasi-permutation,
\item[\rm(d)] $\X(G)=Q^m$,
\item[\rm(e)] $G$ is isomorphic to  $(\gPGL_n)^m$.
\end{enumerate}
\end{theorem}

\begin{proof}

(a) $\Longrightarrow$ (b) is obvious.

(b) $\Longrightarrow$ (c) follows from \cite[Theorem 1.27]{LPR06}.

(d) $\Longrightarrow$ (e): clear.

(e) $\Longrightarrow$ (a): clear, because the group $\gPGL_n$ is
Cayley, see \cite[Theorem 1.31]{LPR06}, and a product of Cayley groups
is obviously Cayley.

The implication (c) $\Longrightarrow$ (d)
follows from the next proposition.
\end{proof}

\begin{proposition}\label{prop:An}
Let $R=\AA_{n-1}$, where  $n\ge5$.  Suppose an
intermediate $W^m$-lattice $L$ between $Q^m$ and $P^m$.
is quasi-permutation. Then $L=Q^m$.
\end{proposition}

\begin{proof}
We proceed by induction on $m$. The base case, $m = 1$,
follows from~\cite[Proposition 5.1]{LPR06}. For the induction step,
assume that $m\ge 2$ and that the proposition holds for
$m-1$. We will show that it also holds for $m$.

Set $I :=\{2,\dots,m\}\subset\{1,2,\dots,m\}$
and $F=P/Q=\mathbb Z/n\mathbb Z$.
In view of Lemma~\ref{LPR-reduction-two}, $L\cap P_I$ is a
quasi-permutation $W_I$-lattice. By the induction hypothesis, $L\cap
P_I=Q_I$. Set $S=L/Q^m\subset F^m$, then  $S\cap F_I=0$. It follows
that the canonical projection $S\to F_1$ is injective. As $F=\mathbb
Z/n\mathbb Z$, we have $S\cong\Z/d\Z$ for some divisor $d$ of $n$.

In the notation of the beginning of Section~\ref{sect10}, $L=P^m_S$ as  a
$W$-lattice (where $W$ acts on $P^m$ diagonally). By Corollary
\ref{cor:standard1},
\begin{equation}\label{eq:sum}
L\cong L_1\oplus Q^{m-1},
\end{equation}
where $Q_1\subset L_1\subset P_1$. Clearly  $Q^{m-1}$ is
quasi-permutation as a $W$-lattice because so is $Q=\ker[\mathbb
Z[\Sym_n/\Sym_{n-1}]\to\mathbb Z]$. By assumption, $L$ is a
quasi-permutation $W^m$-lattice, hence it is quasi-permutation as a
$W$-lattice. Since $L$ and $Q^{m-1}$ are quasi-permutation
$W$-lattices, we see from  \eqref{eq:sum} that $L_1\sim L_1\oplus Q^{m-1}\cong L\sim 0$,
so that $L_1$ is a quasi-permutation $W$-lattice.
By \cite[Proposition 5.1]{LPR06} it follows
that $L_1=Q_1$, hence $S=0$, and $P^m_S=Q^m$. Thus $L=Q^m$, which
proves (d) for $m$ and completes the proofs of Proposition
\ref{prop:An} and Theorem \ref{thm.SLm}.
\end{proof}

\subsection{Case $\BB_n$ $(n\ge 3)$ and $\DD_n$ $(n\ge 4)$}

Let $n \ge 7$, let $R$ be the root system of $\gSpin_n$
(of type  $\BB_{(n-1)/2}$ for $n$ odd or of type $\DD_{n/2}$ for $n$ even)
and let $M$ be the character lattice of $\gSO_n$.
If $n$ is odd, then $M=Q$; if $n$ is even, then
$Q\subsetneq M\subsetneq P$. Set $F:=P/M \cong \Z/2\Z$.

\begin{theorem} \label{thm.SpinGE7}
Let $G = (\gSpin_n)^m/C$, where $n\ge 7$ and  $C$ is
a central subgroup of $(\gSpin_n)^m$.
Then the following conditions are equivalent:
\begin{enumerate}[\upshape(a)]
\item $G$ is Cayley,
\item $G$ is stably Cayley,
\item the character lattice $\X(G)$ of $G$ is quasi-permutation,
\item $\X(G) \cong M^m$, where $M=\X(\gSO_n)$,
\item $G$ is isomorphic to $(\gSO_n)^m$.
\end{enumerate}
\end{theorem}

\begin{proof}
Only (c) $\Longrightarrow$ (d) needs to be proved;
the other implications are easy.

Assume (c), i.e., $\X(G)$ is a quasi-permutation $W^m$-lattice.
Clearly $Q^m \subset \X(G) \subset P^m$.
We claim that $\X(G) \supset M^m$.
If $n$ is odd this is obvious, because $M^m=Q^m$.
If $n$ is even then by Lemma \ref{LPR-reduction-two},
$\X(G) \cap P_i$ is a quasi-permutation $W_i$-lattice.
Now by \cite[Theorem 1.28]{LPR06}, we have $\X(G)\cap P_i=M_i$.
Thus $\X(G) \supset M_1\oplus\dots\oplus M_m = M^m$,
as claimed.
(In the case $\DD_4$ we have $\X(G)\cap P_i\cong M_i$, and by abuse of notation
we write $M^m$ for $(\X(G)\cap P_1)\oplus\dots\oplus(\X(G)\cap P_m)$.)

We will now show that $\X(G)=M^m$. Assume the contrary.
Consider the surjection
$\lambda\colon P\to P/M\cong\Z/2\Z$. Set $S=\X(G)/M^m\subset
(\Z/2\Z)^m$, then $S\neq 0$. In the notation of
Lemma~\ref{lem:standard}, we have $\X(G) =P^m_S$. Since $S\neq 0$,
by Corollary~\ref{cor:standard2}
$\X(G)$ has a direct $W$-summand isomorphic to $P$.
By Proposition \ref{prop1.not-qp}, $P$ is not quasi-invertible,
hence  $\X(G)$ is not quasi-invertible as a $W$-lattice. It follows that
$\X(G)$ is not a quasi-invertible $W^m$-lattice, which contradicts (c).
Thus (d) holds, as desired.
\end{proof}

\begin{remark}
Alternatively, we can prove Theorem \ref{thm.SpinGE7}
similar to the proof of Proposition~\ref{prop:An}.
Namely, we prove by induction that $\X(G)=M^m$ using Corollary \ref{cor:standard1}.
Here we make use of the fact that by Proposition \ref{prop1.not-qp},
$P$ is not  quasi-permutation.
\end{remark}

\begin{remark}
Proposition \ref{prop:An} cannot be proved by an argument analogous
to the proof of Theorem \ref{thm.SpinGE7}. Indeed,
the proof of Theorem \ref{thm.SpinGE7} relies on the fact
that $\X(\gSpin_n)$ {\em is not} quasi-invertible
for $n\ge 7$ (see Proposition \ref{prop1.not-qp}). On the other hand,
$\X(\gSL_n)$ {\em is} quasi-invertible (though it is not quasi-permutation)
whenever $n$ is a prime; see~\cite[Proposition 9.1 and Remark 9.3]{CTS2}.
\end{remark}

\section{Proof of Theorem~\ref{thm:product-closed-field}
for $H$ of type $\AA_1 = \BB_1 = \CC_1$}
\label{sect.sl2}

We will continue using Notation \ref{subsec:Q-P-F-W}.
Let $R=\AA_1$.  Set $F=Q/P=\mathbb Z/2\mathbb Z$.

Let  $G = (\gSL_2)^m/C$, where $C$ is a subgroup of $Z((\gSL_2)^m) =
(\mu_2)^m$.
We have $Q^m\subset \X(G) \subset P^m$.
Set $S:=\X(G)/Q^m\subset F^m= (\Z/2\Z)^m$.

\begin{theorem} \label{thm.sl2}
Let $G = (\gSL_2)^m/C$, where $C$ is a subgroup of $Z((\gSL_2)^m) =
(\mu_2)^m$. Then the following conditions are equivalent:
\begin{enumerate}[\upshape(a)]
\item $G$ is Cayley,
\item $G$ is stably Cayley,
\item the character lattice $\X(G)$ is a quasi-permutation $W^m$-lattice,
\item $S:=\X(G)/Q^m$ is an almost coordinate subspace
of $F^m = (\bbZ/2 \bbZ)^m$,
\item $G$ decomposes into a direct product of normal subgroups
$G_1 \times_k \dots \times_k G_s$, where
each $G_i$ is isomorphic to either $\gSL_2$, $\gPGL_2$ or $\gSO_4$.
\end{enumerate}
\end{theorem}

\begin{remark}\label{rem:uniqueness-direct-factors}
The set of normal subgroups $G_1, \dots, G_s$ in part (e)
is uniquely determined by $G$; see
Remark~\ref{rem.almost-coordinate-uniqueness}.
\end{remark}

\begin{proof}[Proof of Theorem~$\ref{thm.sl2}$]
Only the implication (c) $\Longrightarrow$ (d) needs to be proved;
all the other implications are easy.  The implication (c)
$\Longrightarrow$ (d) follows from the next proposition.
\end{proof}

\begin{proposition}\label{prop:A1}
Let $R=\AA_1$ and let $L$ be an intermediate $W$-lattice
between $Q^m$ and $P^m$, i.e., $Q^m\subset L\subset P^m$. Write
$S=L/Q^m\subset F^m=(\Z/2\Z)^m$. Then $L$ is quasi-permutation if
and only if  $S$ is almost coordinate.
\end{proposition}

\begin{proof}
The ``if" assertion follows easily from Lemmas
\ref{LPR-reduction-product} and \ref{lem:Voskresenskii}. To prove
the ``only if'' assertion, we begin by considering three special
cases which will be of particular interest to us.

\smallskip
{\bf Case 1}: $m \le 2$. Here every subspace
of $(\bbZ/2 \bbZ)^m$ is almost coordinate, and
condition (d) holds automatically.

\smallskip
{\bf Case 2}: $S$ is the line $\langle\bo\rangle=\{0,\bo\}
\subset(\bbZ/2 \bbZ)^m$, where $\bo=(1,\dots, 1)$.

This $S=\langle\bo\rangle$ is not almost coordinate for any $m \ge
3$. Thus we need to show that (c) does not hold, i.e., the lattice
$L=P^m_{\langle\bo\rangle}$ is not quasi-permutation. This lattice
is isomorphic to the lattice $M$ described at the beginning of
Section~\ref{sect.family}, in the case where $\Delta$ is the disjoint union
of $m$ copies of $\BB_1$ (or, equivalently, of $\AA_1$)
for $m\ge 3$. By Proposition~\ref{prop1.not-qp}, for $m \ge 3$,
the lattice $M \simeq L = P^m_{\langle\bo\rangle}$, is
not quasi-invertible, hence not quasi-permutation, as claimed.

\smallskip
{\bf Case 3}: $m = 3$. There are two subspaces $S$ of $(\bbZ/2 \bbZ)^3$
that are not almost coordinate:
(i) the line $\langle\mathbf{1}_3\rangle$ and
(ii) the 2-dimensional subspace
cut out by $x_1 + x_2 + x_3 = 0$.
Once again we need to show that in both
of these cases $L$ is not quasi-permutation.

(i) is covered by Case 2 (with $m = 3$). If $S$ is as in (ii),
then $L$ is isomorphic to the lattice $M$ defined
in the statement of Proposition~\ref{prop2.not-qp}.
By this proposition, $L$ is not quasi-invertible, hence not quasi-permutation, as claimed.

\smallskip
We now proceed with the proof of the proposition by induction on $m \ge 1$.
The base case, where $m \le 3$, is covered by
Cases 1 and 3 above.  For the induction step assume that
$m \ge 4$ and that the proposition has been established for all $m'\le m-1$.

Suppose that for some subspace $S=L/Q^m\subset (\bbZ/ 2\bbZ)^m$
we know that $L=P^m_S$ is quasi-permutation.
Our goal is to show that $S$ is almost coordinate.

Since  $L$ is  quasi-permutation, by Lemma \ref{LPR-reduction-two},
we conclude that $L\cap P_I$ is a quasi-permutation $W_I$-lattice
for every $I = \{ i_1, \dots, i_r \} \subsetneq \{ 1 , \dots, m \}$.
By the induction hypothesis, $(L\cap P_I)/Q_I=S\cap F_I$ is an almost
coordinate subspace in $F_I=(\Z/2\Z)^r$.

Now Proposition~\ref{prop.combinatorics} tells us that $S$ is either
the line $\go$, or almost coordinate. If $S$ is the line $\go$, then
$L$ is not quasi-permutation by Case 2, contradicting our
assumption. Thus  $S$ is almost coordinate, which completes the
proofs of Proposition \ref{prop:A1} and Theorem \ref{thm.sl2}.
\end{proof}

\section{Proof of Theorem~\ref{thm:product-closed-field}
for $H$ of types $\AA_2$, $\BB_2 = \CC_2$, and $\AA_3 = \DD_3$}
\label{sect.sl3}

\subsection{Case $\AA_2$} Once again, we will continue using Notation
\ref{subsec:Q-P-F-W}. Set $F:=P/Q\simeq\Z/3\Z$.

\begin{theorem} \label{thm.SL3}
Let $G = (\gSL_3)^m/C$, where  $C$ is a subgroup of
$(\mu_3)^m=Z((\gSL_3)^m)$. Then the following conditions are
equivalent:
\begin{enumerate}[\upshape(a)]
\item $G$ is Cayley,
\item $G$ is stably Cayley,
\item the character lattice $\X(G)$ is a quasi-permutation $W^m$-lattice,
\item $S:=\X(G)/Q^m$ is a coordinate subspace of $F^m\simeq (\bbZ/3 \bbZ)^m$,
\item $G$ decomposes into a direct product of normal subgroups
$G_1 \times_k \dots \times_k G_s$, where
each $G_i$ is isomorphic to either $\gSL_3$ or $\gPGL_3$.
\end{enumerate}
\end{theorem}

\begin{proof}
Once again, only the implication (c) $\Longrightarrow$ (d)
needs to be proved; the other implications are easy.

Clearly $Q^m\subset \X(G) \subset P^m$; assume
$\X(G)$ is quasi-permutation.  The $W$-lattices $P$ and $Q$ are
quasi-permutation, see \cite[Theorem 1.28]{LPR06}. If $S\subset F^m$ is the
$1$-dimensional subspace $\ga$ spanned by a vector  $\aa=(a_1,
\dots, a_m)$ such that  $a_1\ne 0,\ \dots, a_m \ne 0$, then from
Proposition \ref{prop:SL3} it follows that $\X(G) = P^m_\ga$ is not a
quasi-permutation $W^m$-lattice, a contradiction. Now by Proposition
\ref{prop:coordinate-qp}, $\X(G) =P^m_S$ is
quasi-permutation if and only if $S$ is coordinate. This
shows that (c) $\Longrightarrow$ (d).
\end{proof}

\subsection{Case $\BB_2=\CC_2$} Set $F:=P/Q=\Z/2\Z$.

\begin{theorem} \label{thm.Spin5}
Let $G = (\gSpin_5)^m/C$, where  $C$ is
a subgroup of the finite $k$-group
$(\mu_2)^m=\ker[(\gSpin_5)^m\to (\gSO_5)^m]$. Then
the following conditions are equivalent:
\begin{enumerate}[\upshape(a)]
\item $G$ is Cayley,
\item $G$ is stably Cayley,
\item the character lattice $\X(G)$ is quasi-permutation,
\item $S:=\X(G)/Q^m$ is a coordinate subspace of $F^m=(\bbZ/2 \bbZ)^m$,
\item $G$ decomposes into a direct product of normal subgroups
$G_1 \times_k \dots \times_k G_s$, where
each $G_i$ is isomorphic to either $\gSpin_5=\gSp_4$ or $\gSO_5$.
\end{enumerate}
\end{theorem}

\begin{proof}
As in the proof of Theorem \ref{thm.SL3}, we only need to establish
the implication (c) $\Longrightarrow$ (d). We have $Q^m\subset \X(G) \subset
P^m$. The $W$-lattices $P$ and $Q$ are quasi-permutation, see \cite[Theorem 1.28]{LPR06}.
If $S\subset F^m$ is the $1$-dimensional subspace
$\go$ spanned by the vector  $\bo=(1, \dots, 1)$ then by Proposition
\ref{prop1.not-qp}, $P^m_\go$ is not a quasi-invertible
$W^m$-lattice. Now by Proposition \ref{prop:coordinate-qp}, the
$W^m$-lattice $\X(G) =P^m_S$ is quasi-permutation if and only if $S$ is
coordinate, which completes the proof of the theorem.
\end{proof}

\subsection{Case $\AA_3=\DD_3$}

Here $P/Q\simeq\Z/4\Z$.
We set $M=\X(\gSO_6)$, then $M/Q\simeq\Z/2\Z$.

\begin{theorem} \label{thm.SO6}
Let $G = (\gSpin_6)^m/C$, where $C$ is
a subgroup of $Z(G)=(\mu_4)^m=\ker[(\gSpin_6)^m\to (\gPSO_6)^m]$.
We have $Q^m \subset \X(G) \subset P^m$, where $P$, $Q$ and $\X(G)$
are the character lattices of $\gPSO_6$,
$\gSpin_6$ and $G$, respectively.
Then the following conditions are equivalent:
\begin{enumerate}[\upshape(a)]
\item $G$ is Cayley,
\item $G$ is stably Cayley,
\item $\X(G)$ is quasi-permutation,
\item $\X(G) \subset M^m$ and $\X(G)/Q^m$ is a coordinate subspace
of $(M/Q)^m=(\bbZ/2 \bbZ)^m$,
\item $G$ decomposes into a direct product of normal subgroups
$G_1 \times_k \dots \times_k G_s$, where
each $G_i$ is isomorphic to either $\gSO_6$ or $\gPSO_6=\gPGL_4$.
\end{enumerate}
\end{theorem}

\begin{proof}
Both $\gSO_6$ and $\gPSO_6=\gPGL_4$ are Cayley;
see~\cite[Introduction]{LPR06}.  Consequently,
(e) $\Longrightarrow$ (a). Thus we only need to show that
(c) $\Longrightarrow$ (d); the other implications are immediate.
Assume that $\X(G)$ is quasi-permutation.

First we claim that $\X(G) \subset M^m$. Indeed, assume the contrary.
Then $\X(G)/Q^m$ contains an element of order 4.
By Corollary \ref{cor:standard2},
the $W^m$-lattice $\X(G)$ restricted to the diagonal subgroup $W$
has a direct summand  isomorphic to the character lattice $P$
of $\gSpin_6$.  By Proposition \ref{prop1.not-qp},
the $W$-lattice $P$ is not quasi-invertible. We conclude that
$\X(G)$ is not  quasi-invertible as a $W$-lattice
and hence not a quasi-invertible $W^m$-lattice, contradicting
our assumption that $\X(G)$ is quasi-permutation. This proves the claim.

As we mentioned above, $\gSO_6$ and $\gPSO_6$ are both Cayley.
Hence, the $W$-lattices $M$ and $Q$ are quasi-permutation.
Set $F=M/Q\simeq\Z/2\Z$.
If $S:=\X(G)/Q^m\subset F^m$ is the $1$-dimensional subspace
$\go$ spanned by the vector  $\bo=(1, \dots, 1)$, then
by Proposition \ref{prop:SO6}, $\X(G)$ is not a
quasi-invertible $W^m$-lattice, a contradiction.
Now Proposition
\ref{prop:coordinate-qp} tells us that the $W^m$-lattice
$\X(G)/Q^m$ is coordinate in $(M/Q)^m$, and (d) follows.
\end{proof}

This completes the proof of Theorem~\ref{thm:product-closed-field}.

\section{Proof of Theorem~\ref{thm.main3}}
\label{sect.proof-of-thm.main3}

\begin{proof}
Clearly (b) implies (a), so
we only need to show that (a) implies (b).

Let $G$ be a stably Cayley simple $k$-group (not necessarily absolutely simple).
Then $\Gbar :=G\times_k \kbar$ is stably
Cayley over $\kbar$ and is of the form $H^m/C$, where $H$ is a
simple and simply connected $\kbar$-group and $C$ is a central
$k$-subgroup of $H^m$.  By Theorem \ref{thm:product-closed-field},
$\Gbar= G_{1,\kbar} \times_\kbar \dots\times_\kbar G_{s,\kbar}$,
where each $G_{i,\kbar}$ is either a stably Cayley simple group or is
isomorphic to $\gSO_{4,\kbar}$.
(Recall that $\gSO_{4,\kbar}$ is stably Cayley and semisimple,
but is not simple.)
Here we write $G_{i,\kbar}$ for the factors in order to emphasize
that they are defined over $\kbar$.

If $H$ is not of type $\AA_1$, then the subgroups $G_{i,\kbar}$ are simple and hence,
intrinsic in $\Gbar$: they are the minimal closed connected normal subgroups
of dimension $\ge 1$. If $H$ is of type  $\AA_1$, this is no longer
obvious, since some of the groups $G_{i,\kbar}$ may not be simple
(they may be isomorphic to $\gSO_{4,\kbar}$). However, in this case
the subgroups $G_{i,\kbar}$ are intrinsic in $\Gbar$ as well by
Remark~\ref{rem:uniqueness-direct-factors}.  Hence, in all cases,
the Galois group $\Gal(\kbar/k)$ permutes $G_{1,\kbar}, \dots, G_{s,\kbar}$. Since $G$ is
simple over $k$, this permutation action is transitive.

Let $l\subset\kbar$ be the subfield corresponding
to the stabilizer of $G_{1,\kbar}$ in $\Gal(\kbar/k)$.
Then $G_{1,\kbar}$
is $\Gal(\kbar/l)$-invariant, and we obtain an $l$-form of this
$\kbar$-group, which we will denote by
$ G_{1,l}$.  Then $G=R_{l/k} ( G_{1,l})$, where
$G_{1,l}$ is either absolutely simple or an $l$-form of $\gSO_{4,l}$.
If $G_{1,l}$ is absolutely simple, then $G_\kbar$ is a product of simple $\kbar$-groups,
and by Lemma \ref{lem:direct-product} $G_{1,l}$ is stably Cayley over $l$.
If $G_{1,l}$ is an $l$-form of $\gSO_{4,l}$, then
it has to be an {\em outer} $l$-form of the split $l$-form of $\gSO_4$,
hence an outer $l$-form of $\gSO_4$;
otherwise $G_{1,l}$ will not be $l$-simple and consequently,
$G$ will not be $k$-simple.
This completes the proof of Theorem~\ref{thm.main3}.
\end{proof}

\bigskip
\noindent
{\bf Acknowledgements}
\smallskip

  The authors are very grateful to Jean-Louis Colliot-Th\'el\`ene
for proving Lemma \ref{lem:q-inv-Gamma1} and for stating and proving Lemma \ref{lem:q-inv}.
The authors also wish to thank the anonymous referee for helpful remarks.

\bigskip
\noindent
{\bf Funding}
\smallskip

{Borovoi was partially supported
by the Hermann Minkowski Center for Geometry.}
{Kunyavski\u\i\ was partially supported
by the Minerva Foundation through the Emmy Noether Institute for
Mathematics and by the Israel Science Foundation, grant 1207/12.}
{Lemire and Reichstein were partially supported by Discovery Grants
from the Natural Sciences and Engineering Research Council of Canada.}


\begin{thebibliography}{00}

\bibitem{CF}
{\em Algebraic Number Theory,} Proc. instructional conf. organized
by the London Math. Soc. (J.W.S.~Cassels, A.~Fr\"ohlich, eds.),
Academic Press, London, 1967.

\bibitem{Borel}
A. Borel, {\em Linear Algebraic Groups,} 2nd ed., Graduate Texts in
Math., vol.~126, Springer-Verlag, New York, 1991.

\bibitem{Borovoi-Duke}
M. Borovoi, {\it Abelianization of the second
nonabelian Galois cohomology},  Duke Math. J. {\bf 72} (1993), 217--239.

\bibitem{BD}
M. Borovoi, with an appendix by I. Dolgachev,
{\em Real reductive Cayley groups of rank $\le 2$,} \url{ arXiv:1212.1065v2 [math.AG]}.

\bibitem{Bourbaki}
N. Bourbaki, {\em  Groupes et alg\`ebres de Lie, Ch. 4--6,}  Hermann, Paris, 1968.

\bibitem{Brown}
K.S. Brown, {\em Cohomology of Groups,} Graduate Texts in Math.,
vol.~87, Springer-Verlag, New York--Berlin, 1982.

\bibitem{CE}
H. Cartan, S. Eilenberg, {\em Homological Algebra,} Princeton Univ.
Press, Princeton, NJ, 1956.

\bibitem{cayley}
A.~Cayley, \emph{Sur quelques propri\'et\'es des d\'eterminants
gauches}, J. reine angew. Math. {\bf 32} (1846), 119--123; reprinted
in:~{\it The Coll. Math. Papers of Arthur Cayley}, Vol.~I, No.~52,
Cam\-bridge Univ. Press, 1889, 332--336.

\bibitem{CTKPR}
J.-L. Colliot-Th\'{e}l\`ene, B.\`E. Kunyavski{\u\i}, V.L. Popov, Z.Reichstein, 
\emph{Is the function field of a reductive Lie algebra
purely transcendental over the field of invariants for the adjoint
action?}, Compos. Math. {\bf 147} (2011), 428--466.

\bibitem{CTS1}
J.-L.~Colliot-Th\'el\`ene, J.-J. Sansuc, {\em La $R$-\'equivalence sur
les tores,} Ann. Sci. \'Ecole Norm. Sup. (4) {\bf 10} (1977),
175--229.

\bibitem{CTS2}
J.-L.~Colliot-Th\'el\`ene, J.-J. Sansuc, {\em Principal homogeneous
spaces under flasque tori: applications,} J. Algebra {\bf 106}
(1987),  148--205.

\bibitem{Conrad}
B. Conrad, {\em Reductive group schemes (SGA3 Summer School, 2011),}
\url{http://math.stanford.edu/~conrad/papers/luminysga3.pdf}.

\bibitem{CK}
A. Cortella, B. Kunyavski{\u\i}, \emph{Rationality problem for
generic tori in simple groups}, J. Algebra  {\bf 225} (2000),
771--793.

\bibitem{FSS}
Y.Z. Flicker, C. Scheiderer, R. Sujatha,
{\em Grothendieck's theorem on non-abelian $H^2$ and local-global principles},
J. Amer. Math. Soc. {\bf 11} (1998),  731--750.

\bibitem{GH}
P. Griffiths, J. Harris, {\it Principles of Algebraic Geometry},
Wiley-Interscience, 2011.

\bibitem{Isk}
V.A.~Iskovskikh,  {\em Two nonconjugate embeddings of the group $S_3\times Z_2$ into the Cremona group,}
(Russian) Tr. Mat. Inst. Steklova 241 (2003), Teor. Chisel, Algebra i Algebr. Geom., 105--109; translation in
Proc. Steklov Inst. Math. {\bf 241},  no.~2 (2003), 93--97.


\bibitem{K-Selecta}
B.\`E. Kunyavski\u\i, {\em Three-dimensional algebraic tori,} in:
Investigations in Number Theory, Saratov. Gos. Univ., Saratov, 1987,
pp.~90--111; English transl.: Selecta Math. Sov. {\bf 21}(1990),
1--21.

\bibitem{lang}
S.~Lang,   {\em Algebra}, revised third edition, Graduate Texts in
Math., vol.~211, Springer-Verlag, New York, 2002.

\bibitem{LL00}
N. Lemire, M. Lorenz, {\it On certain lattices associated with
generic division algebras,} J. Group Theory {\bf 3} (2000),
385--405.

\bibitem{LPR06}
N. Lemire, V.L. Popov, Z. Reichstein, \emph{Cayley groups}, J.  Amer.
Math.  Soc. {\bf 19} (2006),  921--967.

\bibitem{Lorenz}
M.~Lorenz, {\em Multiplicative Invariant Theory,} Encycl. Math.
Sci., vol.~135, Invariant Theory and Algebraic Transformation
Groups, VI, Springer-Verlag, Berlin, 2005.

\bibitem{Miller}
C. Miller, {\em The second homology group of a group; relations
among commutators,} Proc. Amer. Math. Soc. {\bf 3} (1952), 588--595.

\bibitem{Schur}
I. Schur, {\em Untersuchungen \"uber die Darstellung der endlichen
Gruppen durch gebrochene lineare Substitutionen,} J. reine angew.
Math. {\bf 132} (1907), 85--137, see also: Gesam. Abh., Bd. I,
Springer-Verlag, Berlin-New York, 1973, 198--250.

\bibitem{SGA3}
{\em S\'eminaire de g\'eom\'etrie alg\'ebrique du Bois Marie
$1962$--$64$, Sch\'emas en groupes  (SGA3)} (M.~Demazure, A. Grothendieck,
eds.), Lecture Notes Math., vol.~151, 152, 153, Springer-Verlag,
Berlin--New York, 1970. Re-edition: Documents Math\'ematiques 7,8,
Soc. Math. France, Paris, 2011.

\bibitem{Springer79}
T.A. Springer, \emph{Reductive groups,} in: ``Automorphic forms,
representations and L-functions (Corvallis, Oregon, 1977)'', Proc.
Sympos. Pure Math., XXXIII, Part 1, Amer. Math. Soc., Providence,
RI, 1979, pp.~3--27.

\bibitem{Springer-book}
T.A.~Springer, \emph{Linear Algebraic Groups,} 2nd ed., Progr.
Math., vol.~9, Birkh\"auser, Boston, MA, 1998.

\bibitem{Tits}
J.~Tits, {\em Classification of algebraic semisimple groups,} in:
Algebraic Groups and Discontinuous Subgroups (Proc. Sympos. Pure
Math., Boulder, Colo., 1965), Amer. Math. Soc., Providence, R.I., 1966,
pp.~33--62.

\bibitem{Voskresenskii70}
V.E.~Voskresenski\u\i, {\em Birational properties of linear
algebraic groups,} Izv. Akad. Nauk SSSR Ser. Mat. {\bf 34} (1970),
3--19; English transl.: Math. USSR Isv. {\bf 4} (1970), 1--17.


\bibitem{Voskresenskii-book}
V.E. Voskresenski{\u\i}, {\it Algebraic Groups and Their Birational
Invariants}, Transl. Math. Monographs, vol.~179, Amer. Math. Soc.,
Providence, RI, 1998.

\bibitem{VK}
V.E. Voskresenski{\u\i}, A.A. Klyachko,  {\it Toroidal Fano varieties
and root systems},  Izv. Akad. Nauk SSSR Ser. Mat. {\bf 48} (1984),
237--263; English transl.:  Math. USSR Izv. {\bf 24} (1984),
221--244.

\end{thebibliography}
\end{document}